\documentclass[a4paper,11pt,abstract=on]{scrartcl}

\usepackage[T2A]{fontenc}
\usepackage[utf8]{inputenc}
\usepackage[russian,english]{babel}

\usepackage{graphicx}%
\usepackage{multirow}%
\usepackage{amsmath,amssymb,amsfonts}%
\usepackage{amsthm}%
\usepackage{mathrsfs}%
\usepackage[title]{appendix}%
\usepackage{xcolor}%
\usepackage{textcomp}%
\usepackage{manyfoot}%
\usepackage{booktabs}%
\usepackage{algorithm}%
\usepackage{algorithmicx}%
\usepackage{algpseudocode}%
\usepackage{listings}%
\usepackage{hyperref}
\usepackage{orcidlink}
\usepackage{caption}
\allowdisplaybreaks

\usepackage[
backend=biber,
style=numeric-comp,
giveninits,
]{biblatex}
\usepackage{mathtools}
\usepackage{enumitem}
\usepackage{bbm}
\usepackage{nicefrac}
\usepackage{threeparttable}
\usepackage{tikz}
\usepackage{esvect}

\newcommand{\rm}{\mathrm}

\newcommand{\linegray}{%
  \tikz[baseline=-0.75ex] \draw[-, gray, very thick] (0,0) -- (1,0);%
}
\newcommand{\linegraydashed}{%
  \tikz[baseline=-0.75ex] \draw[dashed, gray, very thick] (0,0) -- (1,0);%
}
\newcommand{\linegraydotted}{%
  \tikz[baseline=-0.75ex] \draw[dotted, gray, very thick] (0,0) -- (1,0);%
}

\DeclareMathOperator{\sign}{sign}
\DeclareMathOperator{\diag}{diag}
\DeclareMathOperator*{\argmax}{argmax}

\DeclareMathOperator*{\spec}{spec}
\DeclareMathOperator*{\dist}{dist}
\DeclareMathOperator*{\trace}{tr}
\DeclareMathOperator*{\tT}{T}
\DeclareMathOperator*{\rank}{rank}
\DeclareMathOperator*{\spa}{span}
\DeclareMathOperator*{\grad}{grad}
\DeclareMathOperator*{\hess}{Hess}

\newcommand{\symA}{A^{\textup{H}}}
\newcommand{\symC}{C^{\textup{H}}}

\newcommand{\scp}[2]{\left\langle{#1}, {#2}\right\rangle}

\newcommand{\bb}{\mathbb}
\newcommand{\mc}{\mathcal}
\newcommand{\mf}{\mathfrak}
\DeclarePairedDelimiter{\abs}{\vert}{\vert}

\DeclarePairedDelimiter{\norm}{\|}{\|}

\newtheorem{theorem}{Theorem}[section]
\newtheorem{proposition}{Proposition}[section]
\newtheorem{corollary}{Corollary}[section]
\newtheorem{lemma}[theorem]{Lemma}
\newtheorem{remark}{Remark}[section]

\begin{document}

\title{Generalization of Zeroth-Order Method for Quotients of Quadratic Functions}

\author{Jonas Bresch\,\orcidlink{0009-0009-3855-0933}%
\thanks{Ludwig-Maximilians-Universität München, 
\href{https://www.math.lmu.de/math4ai/en/}{Bavarian AI Chair for Mathematical Foundation of Artificial Intelligence}, Akademiestr. 7, 80799 Munich, Germany. 
}
\thanks{Munich Center for Machine Learning (MCML).} \\
{\footnotesize\href{mailto:jonas.bresch@math.lmu.de}{jonas.bresch@math.lmu.de}}
}

\maketitle

\begin{abstract}
Optimization of quadratic functions 
and their quotients 
is relevant in subspace and iterative optimization methods.
In this paper,
we consider the matrix-free computation
of the generalized operator norm 
and the maximization of a generalized Rayleigh quotient 
when only forward evaluations of
two linear operators $A$ and $B$ are available.
The proposed method samples search directions
uniformly from the full unit sphere, thereby avoiding tangent-space sampling and
explicit access to the metric matrix $B^{\tT}B$. 
The exact line search along each sampled direction 
reduces to a $2\times2$ generalized eigenvalue problem on specific Gram matrices.
In the generic case its maximizing step has a closed form.
We prove that the objective
values converge almost surely to the largest generalized eigenvalue and that the
distance of the iterates to the leading generalized eigenspace converges to zero.
We also relate full-sphere moment estimators to the Riemannian gradient and Hessian.
The numerical experiments on synthetic Gaussian operators 
illustrate the behavior of the methods, 
and further studies of the proposed algorithms 
indicate favorable empirical convergence 
and weak dependence on the tested problem dimensions.
\end{abstract}



\section{Introduction}
\label{sec: intro}

The computation of extremal spectral quantities of linear operators under restricted access models 
remains a fundamental challenge in applied mathematics, scientific computing, and data analysis. 
Quantities such as the operator norm, generalized Rayleigh quotients, numerical radius, 
and pseudospectral measures govern stability, transient amplification, and admissible step sizes 
in optimization and inverse problems 
\cite{TikArs43,TikArs77,trefethen2005spectrapseudospectra,ryu2022large,chambolle2011first,bredies2022degenerate}. 
In particular, for non-normal operators, eigenvalues alone fail to capture transient growth phenomena, 
making Rayleigh-type extremal formulations and pseudospectral diagnostics indispensable 
\cite{trefethen1999spectrapseudospectra,farrell1996generalizedstabilitytheorypartIInonautonomousoperators}.

A central object in this context is the generalized Rayleigh quotient,
which generalizes the operator norm and appears in a variety of applications. 
In statistics and machine learning, it underlies \textit{linear discriminant analysis} (LDA) 
and \textit{signal-to-noise ratio} (SNR) optimization, 
where between-class and within-class covariance matrices 
define operators \(A\) and \(B\) \cite{fisher1936taxonomic,rao1948bioclassification}. 
In signal processing, related formulations arise in beamforming and detection theory 
\cite{vantrees2002optimumarray}.
In inverse problems and PDE-constrained optimization, generalized Rayleigh quotients govern stability 
constants and convergence rates when weighted norms or preconditioners are involved 
\cite{engel2001aemigroups,lorenz2023chambolle}. 
Despite their importance, computing these quantities becomes highly nontrivial 
when only black-box access to the operators is available. \\[-20pt]

\paragraph{Adjoint-free and Matrix-free Constraints.}
In many modern applications, linear maps \(A : \bb R^d \to \bb R^m\) and \(B : \bb R^d \to \bb R^\ell\) 
are accessible only through forward evaluations. 
This setting arises, for example, in large-scale PDE solvers, time-stepping schemes, 
tomographic forward models, and legacy simulation codes 
\cite{benner2020operatorinference,Savanier2021ProximalGA,gunzburger2002perspectives,lorenz2018randomized}. 
In such scenarios, forming matrices, computing adjoints, or inverting operators is infeasible. 
Moreover, discretized adjoints may be inconsistent with the forward operator, 
leading to unreliable results in classical adjoint-based schemes 
\cite{trefethen2005spectrapseudospectra,huber2025convergence,fikl2016comprehensive,luchini2014adjointequation,luchini2024adjointequation}. 
These constraints rule out standard methods such as power-schemes, 
Lanczos procedures, and Krylov subspace methods, 
which rely on repeated adjoint applications or dense linear algebra 
\cite{krylov1921,lanczos1950,parlett1998symmetriceigenvalueproblem,knyazev2001matrixfreekrylovpcgmethod}.
Relying on inexact adjoints \cite[§~1]{chan1998transpose} 
for reducing numerical complexity \cite[§~2.9]{gunzburger2002perspectives}
in the computations \cite[§~28~\&~48,~p.~446]{trefethen2005spectrapseudospectra}
introduces the so-called adjoint mismatch
as studied in detail in \cite[§~1.1]{B8}.

This has motivated the development of \textit{zeroth-order} and \textit{derivative-free} methods, 
which use only function evaluations and random perturbations 
\cite{burke2002optimizingmatrixstability,li2023stochastic,balasubramanian2022zeroth}. 
Recent approaches combine random search, sketching, and consensus-based dynamics 
to approximate spectral quantities in a matrix-free fashion 
\cite{halko2011randomapproximatematrixdecomposition,riedl2024consensusbasedoptimization}. 
However, existing methods and novel, substantially improved algorithms \cite{B8,B9,B12}
for the operator norm, generalized Rayleigh quotient, and the adjoint mismatch, 
often rely on tangent-space sampling on manifolds, 
require access to adjoint operators or metric information, 
or employ heuristic curvature approximations. \\[-20pt]

\paragraph{Unconstrained Sampling and Zeroth-Order Quasi-Newton Methods.}
This work introduces a fundamentally different approach based on 
sampling on the entire unit sphere. 
Therefore, it can be seen as an unconstrained sampling approach compared to \cite{B8,B9,B12}.
Instead of drawing search directions from tangent spaces of a Riemannian manifold, 
we sample directions globally on the unit sphere,
which removes the need for conditional sampling or metric-dependent constructions. 
This modification decouples the sampling law from the operators and simplifies
implementation. The resulting moment identities provide gradient and curvature
surrogates with known scaling. The one-sided finite-difference gradient estimator
in §~\ref{sec: zeroth order} has an $\mathcal O(\mu)$ smoothing bias; it is therefore
not unbiased for a fixed smoothing radius. 
By contrast, 
the quadratic-form moment identities in §~\ref{sec: estimators}, 
with the directional curvature weight, 
give exactly unbiased estimators of the Riemannian gradient
and Hessian after the explicit deterministic scaling and projection.
Building on this sampling strategy, 
we develop a zeroth-order optimization framework
for generalized Rayleigh quotient maximization that uses only forward evaluations.
The first component is a randomized direction coupled with an exact projective line search. 
The second is a regularized quasi-Newton variant based on stochastic
curvature information. 
Unlike classical quasi-Newton methods on secant conditions \cite{gill1972quasinewton,Antoniou2007}, 
the curvature surrogate is formed
directly from quadratic function evaluations. \\[-20pt]

\paragraph{Relevance.}
The improved methods \cite{B8,B9,B12} already
blend random search \cite{li2023stochastic,balasubramanian2022zeroth}, 
sketching \cite{li2014sketching}, 
slicing \cite{quellmalz2023slicing,quellmalz2024slicing}, 
projection \cite{li2023stochastic,boumal2023introduction}, 
and consensus-style \cite{riedl2024consensusbasedoptimization} updates 
to balance implementability for legacy and black-box codes 
with provable approximation behavior and empirical robustness 
in non-normal and ill-conditioned settings. 
The proposed generalization and extension 
for zeroth-order and derivative-free \cite{B8,B9,B12},
as well as curvature-based methods \cite{li2023stochastic,Marrinan2026zerothorderquasinewton}
provide a viable path to reliable operator norms $\norm{A}$,
extremal Rayleigh quotient $\mc R(A,B)$, where $d = m = \ell$,
and generalized operator norm $\norm{A / B}$ diagnostics, 
and may also contribute to pseudospectral analysis. \\[-20pt]

\paragraph{Contributions, Roadmap, and Scope.}
The paper makes four specific contributions. 
First, it replaces metric-dependent
tangent sampling by full-sphere sampling while retaining a forward-only model 
in §~\ref{sec: preliminaries}.
Second, 
it identifies each exact line-search problem with a $2\times2$ generalized
Rayleigh quotient and gives the maximizing root in the generic case
in §~\ref{sec: construction}. 
Third,
it proves almost-sure convergence of the objective values to the leading generalized
eigenvalue and in distance to the leading generalized eigenspace
in §~\ref{sec: as optimal convergence}.
Fourth, 
it derives moment-based gradient and Hessian surrogates 
and provides a well-posed regularized implementation of the curvature-based method
in §~\ref{sec: estimators}. 
Last, 
in §~\ref{sec: num} proof-of-concept results,
qualitative comparisons with methods from the literature \cite{li2023stochastic,balasubramanian2022zeroth},
and real-world inspired \emph{signal noise-to-ratio} (SNR) computation
extend the theory to the complex setting 
and underline the success of the proposed methods in comparison.

\section{Preliminaries}
\label{sec: preliminaries}

In this section,
we introduce the considered problem of the generalized operator norm,
link it to the extremal generalized Rayleigh quotient,
and introduce tools for zeroth-, first-, and second-order 
Riemannian optimization techniques.

\subsection{Optimization Problem}
\label{sec: optimization problem}

The optimization problem 
\vspace{-0pt}
\begin{equation}
    \label{eq: prob}
    \norm{A/B} \coloneqq \max_{v \in \bb R^{d} \setminus \{0\}} \;
    \frac{\norm{A v}}{\norm{B v}}
\end{equation}
for two rectangular matrices $A \in \bb R^{m \times d}$
and $B \in \bb R^{\ell \times d}$,
which is considered in this paper,
is a generalization of the operator norm 
$\norm{A} \coloneqq \norm{A / I_d} = \max_{v \in \bb R^d \setminus \{0\}}
 \nicefrac{\norm{A v}}{\norm{v}}$
where $B = I_d$ is chosen.
Here, we have to guarantee that $B^{\tT}B \succ 0$,
so that \eqref{eq: prob} is well defined.
Hence, $\ell \geq d$ is a necessary assumption 
on the dimension of the space of the domain and range of $B$.
The sufficient condition on $B$
is equivalently given by $\det(B^{\tT}B) \neq 0$,
$\rank(B) = d$, or $\ker(B) = \{0\}$,
whenever $\ell \geq d$.
Considering the (generalized) spheres 
given by 
\begin{equation*}
    \label{eq: spheres}
    \bb S_D^{d-1} \coloneqq \{v \in \bb R^d \mid \norm{v}_D \coloneqq \sqrt{\scp{v}{D v}} = 1\},
    \quad \text{where} \quad 
    D \succ 0,
\end{equation*}
and we write for the unit sphere,
i.e. for $D = I_d$, $\bb S^{d-1} \coloneqq \bb S_{I_d}^{d-1}$,
yields 
\begin{equation}
    \label{eq: prop sphere gen}
    \norm{A/B}^2
    = \!\!\!\max_{v \in \bb R^d \setminus\{0\}}
    \frac{\scp{v}{A^{\tT}A v}}{\scp{v}{B^{\tT}B v}}
    = \max_{v \in \bb S^{d-1}}
    \frac{\scp{v}{A^{\tT}A v}}{\scp{v}{B^{\tT}B v}}
    = \!\!\max_{v \in \bb S_{B^{\tT}B}^{d-1}}
    \scp{v}{A^{\tT}A v}.
\end{equation}
The latter is due to the scale invariance of the considered (squared) objective
\vspace{-5pt}
\begin{equation}
    \label{eq: obj}
    f(v) 
    \coloneqq 
    \frac{\scp{v}{A^{\tT}A v}}{\scp{v}{B^{\tT}B v}}
    = \frac{\norm{A v}^2}{\norm{B v}^2}
    \vspace{-5pt}
\end{equation}
from \eqref{eq: prob}, respectively \eqref{eq: prop sphere gen}.
The function in \eqref{eq: obj}
is also known as the generalized Rayleigh quotient of $(A^{\tT}A, B^{\tT}B)$.
Furthermore, 
\eqref{eq: prob} is a generalization 
of the extremal generalized Rayleigh quotient computation 
\begin{equation}
    \label{eq: gen rayleigh}
    \mc R(C, D) \coloneqq 
    \max_{v \in \bb R^d \setminus\{0\}} 
    \smash{\frac{\scp{v}{C v}}{ \scp{v}{D v}}}
    = \max_{v \in \bb S_{D}^{d-1}} 
    \scp{v}{C v},
\end{equation}
where $C \in \bb R^{d \times d}$ is an arbitrary, quadratic matrix 
and $D \in \mathcal S_+^d$ is Hermitian positive definite,
considered in \cite{B9}.
Here, we denote
\begin{align*}
    \mc S_+^d \coloneqq 
    &\,\{X \in \mc S^d \mid \scp{v}{X v} > 0 \; \forall v \in \bb R^{d} \setminus\{0\}\} 
    \subset 
    \mc S^d \coloneqq \{X \in \bb R^{d \times d} \mid X = X^{\tT}\}.
\end{align*}
More precisely,
we have the following relation 
$\norm{A/ B}^2
= \mc R(A^{\tT}A, B^{\tT}B)$,
if 
$B^{\tT}B \succ 0$.
If we are not aware of the adjoint $B^{\tT}$ of $B$ \cite{chan1998transpose,trefethen2005spectrapseudospectra,gunzburger2002perspectives,B8}
then the computation and construction for solving 
the extremal generalized Rayleigh quotient
and hence our considered problem \eqref{eq: prob} 
via the generalized $B^{\tT}B$-sphere, i.e. $\bb S_{B^{\tT}B}^{d-1}$,
from \cite{B9} cannot be applied.
Therefore,
we focus on the optimization problem 
\begin{equation}
    \label{eq: prob sphere}
    \norm{A / B}^2 
    = \max_{v \in \bb S^{d-1}} \;
    f(v) 
    = \max_{v \in \bb S^{d-1}} \;
    \frac{\norm{A v}^2}{\norm{B v}^2},
\end{equation}
where we are also interested in the set of maximizers.

\subsection{Riemannian Optimization Techniques}
\label{sec: riemannian optimization}

The problem \eqref{eq: prop sphere gen}
could be solved by Riemannian optimization methods \cite{boumal2023introduction}
via the generalized $B^{\tT}B$-sphere.
Therefore,
we would rely on the tangent space of the (generalized) sphere $\bb S_D^{d-1}$ 
at some sphere point $v \in \bb S_D^{d-1}$ given by
\begin{equation*}
    \label{eq: tangent space}
    \bb T_{Dv} \coloneqq \ker(D h_C[v])
    = \{x \in \bb R^d \mid \scp{x}{D v} = 0\}
    = \spa\{Dv\}^\perp
    \quad \text{for} \quad
    v \in \bb S_D^{d-1},
\end{equation*}
where $h_D(v) \coloneqq \scp{v}{D v} - 1$
is the so-called local defining function.
Hence $\bb S_D^{d-1} = h_D^{-1}(\{0\})$
is a smooth, embedded, $(d-1)$-dimensional Riemannian manifold.
Here, 
for solving \eqref{eq: prob sphere},
we rely on $D = I_d$ 
such that $\bb T_v = \{x \in \bb R^d \mid \scp{x}{v} = 0\} = \spa\{v\}^\perp$
for $v \in \bb S^{d-1}$.
The orthogonal projection onto the tangent space 
for the generalized sphere
is then defined by 
\begin{equation*}
    \label{eq: orth proj}
    P_{Dv} : \bb R^d \to \bb T_{Dv},
    \quad x \mapsto \Bigl(I_d - \tfrac{Dv (Dv)^{\tT}}{\norm{Dv}^2}\Bigr) x
    \quad \text{for} \quad
    v \in \bb S_D^{d-1},
\end{equation*}
and simplifies for \eqref{eq: prob sphere}
to $P_v (x) = (I_d - vv^{\tT}) x$ for $v \in \bb S^{d-1}$.
We further rely on the retraction 
at some sphere point $v \in \bb S_D^{d-1}$
defined by 
\begin{equation}
    \label{eq: retraction}
    R_{Dv} : \bb R^d \to \bb S_D^{d-1},
    \quad x \mapsto \tfrac{v + x}{\norm{v + x}_D}.
\end{equation}
and simplifies to $R_{Dv}(x) = \nicefrac{v + x}{\sqrt{1 + \norm{x}_D^2}}$ 
if $x \in \bb T_{Dv}$.
For solving
the optimization problem \eqref{eq: prob sphere},
two possible iterative approaches are
Riemannian first-order \cite{boumal2023introduction}
and zeroth-order Riemannian ascent methods \cite{li2023stochastic}.
We roughly outline them in the following.

\subsubsection{First-Order Riemannian Ascent Methods}
\label{sec: first order}

The general construction 
of first-order Riemannian gradient ascent methods \cite{boumal2023introduction}
for solving \eqref{eq: prob sphere}
utilizes the so-called Riemannian gradient of 
\eqref{eq: obj}
given by 
\begin{equation}
    \label{eq: riemannian gradient}
    \grad f(v) 
    = P_v \circ \nabla f (v),
    \quad v \in \bb S^{d-1},
\end{equation}
where $\nabla f(v)$ denotes the Euclidean gradient and reads as 
\begin{equation}
    \label{eq: euclidean grad}
    \nabla f(v) = \tfrac{2}{\norm{B v}^4}\bigl(\norm{B v}^2 A^{\tT}A v - \norm{A v}^2 B^{\tT}B v\bigr).
\end{equation}
Plugging \eqref{eq: euclidean grad} into \eqref{eq: riemannian gradient} yields 
$\grad f(v) = \nabla f(v)$ 
since $\nabla f(v) \in \spa\{v\}^\perp = \ker(x \mapsto vv^{\tT}x)$
such that $vv^{\tT} \nabla f(v) = 0 \in \bb R^d$.
The iterative algorithm breaks down to 
Alg.~\ref{alg: first order} in §~\ref{sec: alg}.
The critical points of \eqref{eq: prob sphere}
are therefore given by 
\begin{align}
    \label{eq: eigen}
    0 = \grad f(v)
    &\; \Leftrightarrow \; 
    \norm{B v}^2 A^{\tT}A v = \norm{A v}^2 B^{\tT}B v 
    \; \Leftrightarrow \; 
    (B^{\tT}B)^{-1}A^{\tT}A v = \tfrac{\norm{A v}^2}{\norm{B v}^2} v \notag \\
    &\; \Leftrightarrow \; 
    (B^{\tT}B)^{-1/2}A^{\tT}A (B^{\tT}B)^{-1/2} w = \tfrac{\norm{A (B^{\tT}B)^{-1/2} w}^2}{\norm{w}^2} w,
\end{align}
by a change of variables $(B^{\tT}B)^{1/2} v = w$.
Hence, the critical points 
are the eigenvectors 
of $(B^{\tT}B)^{-1/2}A^{\tT}A(B^{\tT}B)^{-1/2}$,
also known as generalized eigenvectors of $(A^{\tT}A, B^{\tT}B)$.
The generalized eigenspaces are $G_\lambda\setminus\{0\}$ defined by 
\begin{equation*}
    \label{eq: gen eigen space}
    G_\lambda \coloneqq \{v \in \bb R^d \mid A^{\tT}A v = \lambda B^{\tT}B v\},
\end{equation*}
where $\lambda$ is a generalized eigenvalue given in the set of generalized eigenvalues 
\begin{align*}
    \spec((B^{\tT}B)^{-1/2}A^{\tT}A(B^{\tT}B)^{-1/2}) 
    &= \{\lambda \in \bb R \mid \exists v \in \bb S^{d-1} , A^{\tT}A v = \lambda B^{\tT}B v\} \\
    &= \{\lambda_1,...,\lambda_d\},
    \qquad \lambda_1 \geq \lambda_2 \geq ... \geq \lambda_d,
\end{align*}
by the symmetry of $(B^{\tT}B)^{-1/2}A^{\tT}A(B^{\tT}B)^{-1/2} \in \mc S^d$.
Therewith,
it is 
\begin{equation*}
    \eqref{eq: prob}^2
    = \eqref{eq: prob sphere} 
    = \lambda_1 
    = \max_{\lambda \in \spec((B^{\tT}B)^{-1/2}A^{\tT}A(B^{\tT}B)^{-1/2})} \lambda
\end{equation*}
and $\arg\eqref{eq: prob} \cap \bb S^{d-1} = G_{\lambda_1} \cap \bb S^{d-1}$.

\subsubsection{Zeroth-Order Riemannian Ascent Methods}
\label{sec: zeroth order}

To avoid applications of $A^{\tT}$ and $B^{\tT}$, 
one may replace the
Riemannian gradient \eqref{eq: riemannian gradient} by a finite-difference estimator. 
Let
\begin{equation}
    \label{eq: num grad}
    \widehat{\grad}_{1,\mu} f(v)
    \coloneqq \frac{f(R_v(\mu \tilde x)) - f(v)}{\mu}\,\tilde x,
    \qquad
    \tilde x=P_vx,
    \qquad x\sim\mc N(0,I_d).
\end{equation}
The smoothing radius $\mu>0$ controls the finite-difference bias; 
it does not, by itself, 
guarantee that every realization is an ascent direction. 
Since $\bb E[\tilde x\tilde x^{\tT}]=P_v$ and $P_v\nabla f(v)=\grad f(v)$, 
a Taylor expansion of $f\circ R_v$ gives
\begin{equation}
    \label{eq: num grad bias}
    \bb E[\widehat{\grad}_{1,\mu}f(v)]
    =\grad f(v)+\mathcal O(\mu),
    \qquad \mu\downarrow0,
\end{equation} 
where the remainder is uniform for $v\in\bb S^{d-1}$ because the sphere is compact
and $f$ is smooth under $B^{\tT}B\succ0$. Thus the one-sided estimator is
asymptotically unbiased, but is not unbiased for fixed $\mu$.

For a sample count $q\in\bb N$, define the averaged estimator
\begin{equation}
    \label{eq: num grad red}
    \widehat{\grad}_{q,\mu} f(v)
    \coloneqq \tfrac{1}{q}\sum_{i=1}^{q}
    \tfrac{f(R_v(\mu\tilde x_i))-f(v)}{\mu}\,\tilde x_i,
    \quad
    \tilde x_i=P_vx_i,
    \quad x_i\overset{\mathrm{iid.}}{\sim}\mc N(0,I_d).
\end{equation} 
Independence yields
$\operatorname{Cov}[\widehat{\grad}_{q,\mu}f(v)] =\tfrac{1}{q}\operatorname{Cov}[\widehat{\grad}_{1,\mu}f(v)]$
the exact covariance reduction,
and hence the trace of the covariance decreases by the factor $\tfrac{1}{q}$. The
corresponding zeroth-order ascent method is summarized in
Alg.~\ref{alg: zeroth order} in Appendix~\ref{sec: alg}.

\subsubsection{Second-Order Riemannian Ascent Methods}
\label{sec: second order}

For capturing even more information concerning the curvature of the considered objective,
one might use the Hessian and Riemannian Hessian.
Here, for the function of interest $f$, cf.~\eqref{eq: obj},
we have for $v \in \bb R^d\setminus\{0\}$ that
\begin{align}
    \label{eq: hess special}
    \nabla^2 f(v) = \tfrac{2}{\norm{B v}^2}\Bigl[
        &(A^{\tT}A - f(v) B^{\tT}B) + \tfrac{4\norm{Av}^2}{\norm{B v}^4} (B^{\tT}B v)(B^{\tT}B v)^{\tT} \notag \\
        &\quad- \tfrac{2}{\norm{B v}^2}\bigl[(A^{\tT}A v)(B^{\tT}B v)^{\tT} + (B^{\tT}B v)(A^{\tT}A v)^{\tT}\bigr] \Bigr]
\end{align}
for the (Euclidean) Hessian 
and the Riemannian Hessian 
is given \cite{boumal2023introduction} by
\begin{equation}
    \label{eq: riemannian hess}
    \hess f(v) = P_v \circ \nabla^2 f(v) \circ P_v,
    \qquad v \in \bb S^{d-1}.
\end{equation}
Therefore, plugging in \eqref{eq: hess special},
we obtain the simplification of \eqref{eq: riemannian hess} for \eqref{eq: obj} as
\begin{equation}
    \label{eq: riemannian hess special}
    \hess f(v)
    = \scalebox{0.88}{$\frac{2}{\norm{Bv}^2} P_v
    \Bigl(
    \underbracket{
        \overbracket{
            A^{\tT}A-f(v)B^{\tT}B
        }^{=:\,C(v)\,\in\,\mc S^d}
        - \nabla f(v)(B^{\tT}Bv)^{\tT}
        - B^{\tT}Bv\,\nabla f(v)^{\tT}
    }_{=:\,H(v)\,\in\,\mc S^d}
    \Bigr) P_v,$}
\end{equation}
where the sphere correction $-\scp{v}{\nabla f(v)}P_v$ 
vanishes because $f$ is homogeneous of degree zero.
The simpler inducing matrix $C(v)$ yields the same projected Hessian at
critical points, where $\grad f(v)=0$, or when $P_vB^{\tT}Bv=0$.

\subsection{Properties of the Objective}
\label{sec: prop obj}

In this subsection 
we summarize some properties of the considered (squared) objective \eqref{eq: obj}
for the optimization problem \eqref{eq: prob sphere}, respectively \eqref{eq: prob}.

\begin{lemma}
    \label{lem: lipschitz func}
    The objective from \eqref{eq: prob sphere}, i.e. \eqref{eq: obj}, is
    \smash{$2\norm{A}^2\left(
        \tfrac{1}{\lambda_d(B^{\tT}B)}
        +\tfrac{\norm{B}^2}{\lambda_d(B^{\tT}B)^2}
        \right)$}-Lipschitz continuous 
    on $\bb S^{d-1}$.
\end{lemma}
\begin{proof}
    See Appendix~\ref{sec: app}.
\end{proof}

\begin{lemma}
    \label{lem: lipschitz grad}
    The (extended) gradient 
    of the objective from \eqref{eq: prob sphere}, i.e. \eqref{eq: obj},
    is $\tfrac{20}{\lambda_{d}(B^{\tT}B)^4} \norm{A}^2\norm{B}^6$-Lipschitz continuous 
    on $\bb S^{d-1}$.
\end{lemma}
\begin{proof}
    See Appendix~\ref{sec: app}.
\end{proof}

\section{General Construction}
 \label{sec: construction}

In this section, 
we derive the key geometric observation.
After a direction $x^k$ is sampled, the
entire update remains in the two-dimensional subspace $\spa\{v^k,x^k\}$. 
Maximizing the quotient along this projective
line is therefore exactly a $2\times2$ generalized Rayleigh quotient problem.
This interpretation is both the source of the closed-form line search and the
most robust way to implement degenerate cases.

In \cite{B9,B8},
a random step direction for a current iterate $v^k$, 
i.e. a surrogate of the Riemannian gradient, is constructed
in each iteration $k \in \bb N$
in $\bb T_{v^k} \cap \bb S^{d-1}$.
More precisely, it is proposed \cite[Rem.~4.1]{B9} to take 
\begin{equation}
    \label{eq: tangent surrogate}
    x^k \coloneqq \tfrac{P_{B^{\tT}Bv^k} x^k}{\norm{P_{B^{\tT}Bv^k} x^k}} 
    \sim \bb P_{X^k \mid V^k = v^k} = \mc U (\bb T_{B^{\tT}B v^k} \cap \bb S^{d-1})
\end{equation}
with $x^k \coloneqq \frac{x}{\norm{x}}$, where $x \sim \mc N(0, I_d)$.
This is motivated by the property of the Riemannian gradient,
cf. \eqref{eq: riemannian gradient}~f.,
and the fact that \eqref{eq: tangent surrogate} is an unbiased estimator
of the (true) Riemannian gradient \cite[Lem.~4.1]{B9} for solving \eqref{eq: prop sphere gen},
i.e. it holds w.r.t. \eqref{eq: tangent surrogate} that
\begin{equation}
    \label{eq: connection grad length}
    \bb E_{x^k \sim X^k \mid V^k = v^k}[\scp{A^{\tT}A v^k}{x^k}x^k] 
    = \tfrac{P_{B^{\tT}B v^k} A^{\tT}A v^k}{d-1}
    = \tfrac{1}{2(d-1)} \grad \norm{A v^k}^2.
\end{equation}
Here, 
since we are not aware of $B^{\tT}B$
and, hence, working on the generalized sphere \cite{B9} is not possible,
we propose a more general 
and unconstrained sampling approach.
More precisely, we choose $x^k \sim \mc U(\bb S^{d-1})$
uniformly distributed on the entire unit sphere \cite[Ex. 3.3.7]{vershynin2018highdimprob}
given by 
\begin{equation}
    \label{eq: x^k}
    x^k \coloneqq \tfrac{x}{\norm{x}},
    \qquad x \sim \mc N(0, I_d)
\end{equation}
with property
$\bb E_{x \sim \mc N(0, I_d)}[x^k(x^k)^{\tT}] = \bb E_{x^k \sim \mc U(\bb S^{d-1})}[x^k(x^k)^{\tT}] = \tfrac{1}{d} I_d$.
Therefore, 
we resolved the issue 
for discussing conditional distributions
as we have $\bb P_{X^k \mid V^k=\cdot} = \bb P_{X^k} = \mc U(\bb S^{d-1})$,
which is necessary for the construction in \cite{B9,B8}, cf.~\eqref{eq: tangent surrogate}.
Afterwards,
we aim to solve the following optimal step size problem 
in each iteration $k \in \bb N$
given by %
\begin{equation}
    \label{eq: sub prob}
    \tau_k \in \argmax_{\tau \in \bb R} s_k(\tau)
    \quad \text{where} \quad 
    s_k (\tau)
    \coloneqq 
    f(R_{v^k}(\tau x^k)),
\end{equation}
which simplifies to a quotient of quadratic functions 
\begin{equation}
    \label{eq: sub prob func}
    s_k(\tau)
    = f(v^k + \tau x^k)
    = \frac{\norm{A(v^k + \tau x^k)}^2}{\norm{B(v^k + \tau x^k)}^2}
    = \frac{a_k + 2 b_k \tau + c_k \tau^2}{d_k + 2 e_k \tau + f_k \tau^2},
\end{equation}
where 
\begin{equation}
    \label{eq: para mgRq}
    \begin{array}{c c c}
        a_k \coloneqq \norm{A v^k}^2, 
        & b_k \coloneqq \scp{A v^k}{A x^k}, 
        & c_k \coloneqq \norm{A x^k}^2, \\
        d_k \coloneqq \norm{B v^k}^2, 
        & e_k \coloneqq \scp{B v^k}{B x^k}, 
        & f_k \coloneqq \norm{B x^k}^2.
    \end{array}
\end{equation}
The proposed algorithm is summarized in Alg.~\ref{alg: main}.
Notably,
the construction 
via the tangent spaces in \cite{B9,B8}, cf.~\eqref{eq: tangent surrogate},
allows for a much simpler quotient in \eqref{eq: sub prob func},
where the denominator reduces to $\tau \mapsto 1 + f_k \tau^2$, cf.~\eqref{eq: retraction}~f.,
and $f_k > 0$.
Now,
the analysis for solving \eqref{eq: sub prob}
involves even more theoretical background 
than just calculating the roots of
\begin{align}
    \label{eq: sub prob stationary}
    s_k'(\tau)
    & = 2 \frac{b_kd_k - a_ke_k
        + \tau (c_kd_k - a_kf_k) 
        + \tau^2 (c_ke_k - b_kf_k)}{(d_k + 2 \tau e_k + \tau^2 f_k)^2}.
\end{align}
Note that 
$d_k = \norm{v^k}_{B^{\tT}B}^2 \geq \lambda_d(B^{\tT}B) > 0$ 
and, for any $\tau \in \bb R \setminus\{0\}$,
whenever $x^k \notin \spa\{v^k\}$,
$d_k + 2 \tau e_k + \tau^2 f_k = \norm{v^k + \tau x^k}^2_{B^{\tT}B}$ and 
\begin{equation}
    \label{eq: sum d e f}
    \norm{v^k + \tau x^k}^2_{B^{\tT}B}
    \geq \lambda_d(B^{\tT}B)^2
    \geq \lambda_d(B^{\tT}B) (1 - \langle x^k,v^k\rangle^2)
    > 0.
\end{equation}

\begin{algorithm}[htb]
    \caption{Full-Sphere Random Line-Search Method}
    \label{alg: main}
    \begin{algorithmic}
        \State{\textbf{Given} $A \in \bb R^{m \times d}$ and $B \in \bb R^{\ell \times d}$ such that $\rank(B) = d$.}
        \State{\textbf{Initialize} with $v^0 \in \bb S^{d-1}$, for instance $v^0 = \tfrac{x}{\norm{x}}$ for $x \sim \mc N(0, I_d)$.}
        \For{$k = 0,1,2,\ldots$}
         \State{\textbf{Sample} $x^k = \tfrac{x}{\norm{x}}$ via $x \sim \mc N(0, I_d)$ as in \eqref{eq: x^k}.}
        \State{\textbf{Calculate} the maximizing projective step from the leading generalized eigenvector of $(M_k,N_k)$ as in Thm.~\ref{thm: svd}; in the generic case use \eqref{eq: opt tau}.}
        \State{\textbf{Update} $v^{k+1}=x^k$ if $\tau_k=\infty$; otherwise set
        $v^{k+1}=R_{v^k}(\tau_kx^k)=\tfrac{v^k+\tau_kx^k}{\norm{v^k+\tau_kx^k}}$.}
        \EndFor
        \State{\textbf{Return} estimate $\norm{A/B} \approx \tfrac{\norm{A v^k}}{\norm{B v^k}} = \sqrt{\tfrac{a_k}{d_k}}$.}
    \end{algorithmic}
\end{algorithm}

\begin{remark}[Connection to \cite{B8} and \cite{B9}]
    \label{rem: connection}
    All theoretical results are done for the case $C = A^{\tT}A$ and $D = B^{\tT}B$ for a direct connection to \cite{B8}
    as adjoint-free approach and connects to \cite{B9} for $\nicefrac{\scp{v}{C v}}{\scp{v}{D v}}$.
    However, all the novel, presented results remain true for general matrices $(C, D)$,
    where $D \in \mc S_+^d$ just has to be Hermitian, positive definite 
    and we may have 
    \begin{equation*}
    \begin{array}{c c c}
        a_k \coloneqq \langle v^k, C v^k\rangle, 
        & b_k \coloneqq \langle v^k, C x^k\rangle, 
        & c_k \coloneqq \langle x^k, C x^k\rangle, \\
        d_k \coloneqq \langle v^k, D v^k\rangle, 
        & e_k \coloneqq \langle v^k, D x^k\rangle, 
        & f_k \coloneqq \langle x^k, D x^k\rangle.
    \end{array}
\end{equation*}
\end{remark}

\begin{remark}[Distribution of $(v^k, x^k)_{k \in \bb N}$]
    \label{rem: dist x^k v^k}
    The generated sequence 
    of iterates $(v^k)_{k \in \bb N}$ from Alg.~\ref{alg: main}
    is a sequence of samples from a Markov chain $(V^k)_{k \in \bb N}$.
    Here, for each $k \in \bb N$ the law is given by \smash{$P_{V^k} \coloneqq (V^k)_{\#} \bb P = \bb P \circ (V^{k})^{-1}$}
    on a probability space $(\Omega, \mf A, \bb P)$.
    As discussed \eqref{eq: x^k}~ff., 
    the step directions are samples from a random variable $X^k : \Omega \to \bb S^{d-1}$
    with law $\bb P_{X^k} = \mc U(\bb S^{d-1})$
    so that we deal with the surface measure $\sigma_{\bb S^{d-1}}$,
    which resolves the conditional properties from \cite[Rem.~4.1]{B9}.
    The iterates are samples from a conditional distribution \smash{$\bb P_{V^{k+1} \mid V^k}$},
    which defines the Markov kernel.
    Hence, \smash{$\bb P_{V^{k+1} \mid V^k=v^k}$} is a measure for each $v^k \in \bb S^{d-1}$
    and \smash{$\bb P_{V^{k+1} \mid V^k=\cdot}(A)$} is a measurable function 
    for any Borel set $A \in \mc B(\bb S^{d-1})$.
    Throughout the paper 
    we call an event almost surely (a.s.)
    if it happens with probability one,
    i.e. the counter event is with probability zero,
    with respect to $\bb P_{X^k}$.
\end{remark}

For the probabilistic proofs,
we rely on two theoretical results 
concerning the distribution 
of the random step directions, from Rem.~\ref{rem: dist x^k v^k}.

\begin{proposition}[{\cite[Lem.~4.2]{B9}}]
    \label{prop: measure zero affine space}
    Let $M \subset \bb R^d$ be an affine subspace of $\bb R^d$
    and 
    \begin{equation*}
        \varphi : M \setminus \{0\} \to \bb S^{d-1}, 
        \qquad x \mapsto \tfrac{x}{\norm{x}}.
    \end{equation*}
    Then, w.r.t. $\mc U(\bb S^{d-1})$ the uniform distribution on $\bb S^{d-1}$, holds
    \begin{enumerate}
        \item[i)] $\varphi(M)$ is a measure zero set if $0 \in M$ and $\dim(M) < d$,
        \item[ii)] $\varphi(M)$ is a measure zero set if $0 \not\in M$ and $\dim(M) < d-1$.
    \end{enumerate}
\end{proposition}

\begin{proposition}[{\cite[Thm.~A.4]{langrenez2024setkirkwooddiracpositivestates},\cite{mityagin2020measurezero}}]
    \label{prop: measure zero roots}
    Let $\varrho : \bb R^{d} \to \bb R$ be a polynomial.
    Then the set 
    \begin{equation*}
        \varrho^{-1}(\{0\}) \cap \bb S^{d-1} = \{x \in \bb S^{d-1} : \varrho(x) = 0 \}
    \end{equation*} 
    is either 
    \begin{enumerate}
        \item[i)] a null set w.r.t. $\mc U(\bb S^{d-1})$ the uniform distribution on $\bb S^{d-1}$,
        or 
        \item[ii)] the entire sphere $\bb S^{d-1}$, i.e. $\varrho \equiv 0$ on $\bb S^{d-1}$.
    \end{enumerate} 
\end{proposition}

\begin{remark}[Well-definiteness of the update scheme]
    \label{rem: norm not zero}
    For a well defined next iterate $v^{k+1}$,
    we have to ensure 
    that the denominator, i.e. the norm of $v^k + \tau x^k$,
    is not zero, to have, 
    even without normalization 
    for the retraction \eqref{eq: retraction},
    a well defined objective value, cf.~\eqref{eq: sum d e f}~f.
    Therefore,
    we observe that $\norm{v^k + \tau x^k} = 0$ 
    for some $\tau \in \bb R$ holds 
    if and only if
    \begin{equation*}
        x^k \in \spa\{v^k\}
        \quad \text{with} \quad
        \bb P_{X^k}(x^k \in \{v^k,-v^k\}) 
        = 0,
    \end{equation*}
    using Prop.~\ref{prop: measure zero affine space}~i), 
    since $0 \in \spa\{v^k\}$ is a one-dimensional subspace in $\bb R^d$.
\end{remark}

\subsection{Well-Definiteness of the Sequence}

In contrast to \cite{B8,B12,B9}
where the existence of the unique maximizer is clear under certain constructions, 
cf.~\eqref{eq: retraction} and \eqref{eq: sub prob func}~f.,
the existence and well-definiteness for \eqref{eq: sub prob} 
is shown by reformulating the problem 
to an equivalent extremal generalized Rayleigh quotient problem
for ($2\times2$)-matrices.

\begin{theorem}
    \label{thm: svd}
    Assume $N_k\succ0$, with $a_k,b_k,c_k,d_k,e_k,f_k$ as in \eqref{eq: para mgRq}. 
    Then the projective line-search problem \eqref{eq: sub prob} 
    is equivalent to the largest generalized eigenvalue
    problem for
    \begin{equation*}
        M_k=\begin{bmatrix}a_k&b_k\\ b_k&c_k\end{bmatrix},
        \qquad
        N_k=\begin{bmatrix}d_k&e_k\\ e_k&f_k\end{bmatrix}.
    \end{equation*}
    If $\xi_k=(\xi_{k,1},\xi_{k,2})^{\tT}$ is a generalized eigenvector
    associated with the largest generalized eigenvalue, then a maximizing
    projective step is $\tau_k=\xi_{k,2}/\xi_{k,1}$ when $\xi_{k,1}\neq0$;
    the case $\xi_{k,1}=0$ represents $\tau_k=\infty$ and the direction $x^k$.

    In the generic case $\gamma_k\neq0$, 
    the maximizer is finite and equals
    \begin{equation}
        \label{eq: opt tau}
        \tau_k=-\frac{\beta_k}{2\gamma_k}
        -\sign(\gamma_k)
        \sqrt{\frac{\beta_k^2}{4\gamma_k^2}-\frac{\alpha_k}{\gamma_k}},
        \qquad
        \left\{\begin{aligned}
            \alpha_k&=b_kd_k-a_ke_k,\\
            \beta_k&=c_kd_k-a_kf_k,\\
            \gamma_k&=c_ke_k-b_kf_k.
        \end{aligned}\right.
    \end{equation}
\end{theorem}
\begin{proof}
    Firstly, 
    notice for $\xi \colon \bb R \to \bb R^2, \tau \mapsto [1,\tau]^{\tT}$,
    it holds
    \begin{equation*}
        \frac{\scp{\xi(\tau)}{M_k \xi(\tau)}}{\scp{\xi(\tau)}{N_k \xi(\tau)}}
        = \frac{a_k + 2 \tau b_k + \tau^2 c_k}{d_k + 2 \tau e_k + \tau^2 f_k} = s_k(\tau).
    \end{equation*}
    Since $N_k\succ0$, 
    this quotient is well defined on the real projective line,
    which is compact. 
    Its maximum is therefore the largest generalized
    eigenvalue of $(M_k,N_k)$, and a maximizing projective coordinate is obtained
    from the corresponding generalized eigenvector as stated.

    Secondly, we discuss the real-valued property 
    of the solution, i.e. the non-negativity of the argument of the square-root:
    The Gram matrices $M_k, N_k \in \bb R^{2\times 2}$ 
    are equivalently given by 
    \begin{align*}
        M_k & \coloneqq [v^k, x^k]^{\tT} A^{\tT}A [v^k, x^k]
        = \left[\begin{smallmatrix}
            \scp{u^k}{u^k} & \scp{u^k}{w^k} \\ \scp{u^k}{w^k} & \scp{w^k}{w^k}
        \end{smallmatrix}\right]
        = [u^k, w^k]^{\tT} [u^k, w^k]
    \end{align*}
    and
    \begin{align*}
        N_k & \coloneqq [v^k, x^k]^{\tT} B^{\tT}B [v^k, x^k]
        = \left[\begin{smallmatrix}
            \scp{y^k}{y^k} & \scp{y^k}{z^k} \\ \scp{y^k}{z^k} & \scp{z^k}{z^k}
        \end{smallmatrix}\right]
        = [y^k, z^k]^{\tT} [y^k, z^k],
    \end{align*}
    where 
    \begin{equation}
        \label{eq: y z}
        u^k \coloneqq A v^k,
        \quad w^k \coloneqq A x^k,
        \quad y^k \coloneqq B v^k,
        \quad \text{and} \quad 
        z^k \coloneqq B x^k.
    \end{equation}
    Recall, Gram matrices 
    always satisfy $M_k, N_k \succeq 0$.
    If $N_k \succ 0$, 
    then the stationary condition $s_k'(\tau) = 0$ from \eqref{eq: sub prob stationary}
    is equivalent to the 2d generalized Rayleigh quotient problem given by 
    \begin{displaymath}
        M_k \xi(\tau) = \eta N_k \xi(\tau), 
        \qquad \text{with} \qquad 
        \xi(\tau) = \left[\begin{smallmatrix}
            1 \\ \tau
        \end{smallmatrix}\right]
        \qquad \text{and optimal} \qquad \eta, \tau \in \bb R, 
    \end{displaymath}
    For $\gamma_k\neq0$, 
    the stationary equation is
    \begin{equation*}
        0=\alpha_k+\beta_k\tau+\gamma_k\tau^2,
        \qquad
        \tau_k^{\pm}=-\frac{\beta_k}{2\gamma_k}
        \pm\sqrt{\frac{\beta_k^2}{4\gamma_k^2}-\frac{\alpha_k}{\gamma_k}}.
    \end{equation*}
    Thirdly, 
    the discriminant is nonnegative because the stationary projective points of
    a real symmetric generalized eigenvalue problem are real. 
    At a stationary point the denominator of $s_k'$ is positive, 
    and it holds
    \begin{equation*}
        \sign\bigl(s_k''(\tau_k^{\pm})\bigr)
        =\sign\bigl(\beta_k+2\gamma_k\tau_k^{\pm}\bigr)
        =\pm\sign(\gamma_k).
    \end{equation*}
    Hence the maximizing root is the smaller root when $\gamma_k>0$ and the
    larger root when $\gamma_k<0$, which is exactly \eqref{eq: opt tau}.
\end{proof}

\begin{remark}[Sub-SVD Problem for \cite{B9}]
    Considering the extremal generalized Rayleigh quotient problem $\mc R(C, D)$ 
    from \eqref{eq: gen rayleigh} \cite{B9},
    then the optimal step size problem 
    reads as sub-SVD problem w.r.t.
    \begin{equation*}
        M_k \coloneqq \left[\begin{smallmatrix}
            a_k & \nicefrac{b_k}{2} \\ \nicefrac{b_k}{2} & c_k
        \end{smallmatrix}\right]
        \quad \text{and} \quad 
        N_k \coloneqq \diag(1, f_k) \succ 0,
        \quad \text{where} \quad 
        d_k = 1, \quad e_k = 0,
    \end{equation*}
    by construction due to $v^k \in \bb S_D^{d-1}$
    and $x^k \in \bb T_{Dv^k}$, cf.~Rem.~\ref{rem: connection}.
\end{remark}

For ensuring that $N_k \succ 0$ from Thm.~\ref{thm: svd},
and translating the property to properties of $\{v^k, x^k\}_{k \in \bb N}$,
recall that $N_k \succ 0$,
if and only if $Bv^k$ and $Bx^k$ are linearly independent,
which is equivalent to $v^k$ and $x^k$ being $B$-linearly independent.

\begin{lemma}
    \label{lem: N_k pos def}
    If $\ker(B) = \{0\}$, 
    then $v^k$ and $x^k$ are $B$-linearly independent a.s.,
    hence $y^k$ and $z^k$ are linearly independent a.s.,
    and $N_k \succ 0$ a.s.
\end{lemma}
\begin{proof}
    Let $v^k \in \bb S^{d-1}$, and notice that $B v^k \neq 0$ as well as $B^{\tT}B v^k$.
    To show that $x^k$ and $v^k$ are not $B$-colinear,
    i.e. $y^k$ and $z^k$ are not colinear and hence linearly independent such that $N^k \succ 0$ (a.s.),
    the contrary is assumed.
    To this end, suppose there exists $\lambda \in \bb R\setminus\{0\}$ such that 
    \begin{equation*}
        \lambda y^k = z^k 
        \quad\Leftrightarrow\quad 
        \lambda B v^k = B x^k
        \quad\Rightarrow\quad
        \lambda B^{\tT}B v^k = B^{\tT}B x^k
        \quad\Leftrightarrow\quad 
        \lambda v^k = x^k
    \end{equation*}
    under the assumption that $\ker(B) = \{0\}$ 
    i.e. $B^{\tT} B \succ 0$,
    and hence $(B^{\tT} B)^{-1}$ exists.
    By construction,
    we then have $1 = \norm{x^k} = \norm{\lambda v^k} = |\lambda| \cdot \norm{v^k} = |\lambda|$
    and hence $x^k \in \{v^k, - v^k\} = \spa\{v^k\} \cap \bb S^{d-1}$.
    Similar to the argument in Rem.~\ref{rem: norm not zero},
    we obtain by Prop.~\ref{prop: measure zero affine space}~i) that
    \begin{equation*}
        \bb P_{X^k}(x^k \in \spa\{v^k\} \cap \bb S^{d-1})
        = \bb P_{X^k}(x^k \in \{v^k, -v^k\})
        = 0,
    \end{equation*}
    which yields the assertion.
\end{proof}

Next,
the first case $\gamma_k = 0$ is studied,
where $\gamma_k \neq 0$ would ensure under $N_k \succ 0$
the closed form for the optimal step sizes, see Thm.~\ref{thm: svd} and Lem.~\ref{lem: N_k pos def}.
The second case $\alpha_k = 0$ detects a generalized eigen pair
given by $(v^k, \nicefrac{\norm{A v^k}^2}{\norm{B v^k}^2})$.
Therefore,
we translate the problem to properties of the random variables $v^k \in \bb S^{d-1}$.
In the following, we always assume that $\ker(B) = \{0\}$.
Here, we mainly utilize Prop.~\ref{prop: measure zero roots}.

\begin{lemma}
    \label{lem: gamma_k}
    Let $v^k \in \bb S^{d-1}$ be not a generalized eigenvector of $(A^{\tT}A, B^{\tT}B)$.
    If $x^k$ is sampled as in \eqref{eq: x^k},
    then $\gamma_k \neq 0$ 
    and $\alpha_k \neq 0$ a.s.
\end{lemma}
\begin{proof}
    For the first assertion
    assume the contrary $\gamma_k = 0$
    and observe 
    \begin{align*}
    0 &= \gamma_k = \scp{x^k}{A^{\tT}Ax^k}\scp{B^{\tT}Bx^k}{v^k} - \scp{x^k}{B^{\tT}Bx^k}\scp{A^{\tT}Ax^k}{v^k} \notag \\
    \Leftrightarrow \quad 
    0 &= \scp{x^k}{\Bigl(\tfrac{\scp{x^k}{A^{\tT}Ax^k}}{\scp{x^k}{B^{\tT}Bx^k}} B^{\tT}B - A^{\tT}A\Bigr)v^k}.
    \label{eq:gamma_k}
    \end{align*}
    To this end, 
    define the cubic function $\gamma(x)$ in terms of Prop.~\ref{prop: measure zero roots}
    \begin{equation*}
        \gamma(x) \coloneqq \scp{x}{A^{\tT}Ax}\scp{B^{\tT}Bx}{v^k} - \scp{x}{B^{\tT}Bx}\scp{A^{\tT}Ax}{v^k}.
    \end{equation*}
    Then, by Prop.~\ref{prop: measure zero roots}
    the set of roots on the sphere are either measure zero 
    w.r.t. uniform distribution 
    or $\gamma \equiv 0$.
    Therefore, 
    let $B^{\tT}B = \tilde B^{1/2} \tilde B^{1/2}$ due to the Hermitian, positive definiteness, 
    and define 
    \begin{equation}    
        \label{eq: para alpha beta}
        y \coloneqq \tilde B^{1/2}x, 
        \; \tilde A \coloneqq \tilde B^{-1/2} A^{\tT}A \tilde B^{-1/2},
        \; a^k \coloneqq \tilde B^{-1/2} A^{\tT}A v^k, 
        \; b^k \coloneqq  \tilde B^{1/2}v^k,
    \end{equation}
    such that
    \begin{equation*}
        \gamma(y) = \scp{y}{\tilde A y} \scp{y}{b^k} - \norm{y}^2 \scp{y}{a^k}.
    \end{equation*}
    On $\bb T_{b^k}\setminus\{0\}$, 
    it holds $\gamma(y) = - \norm{y}^2 \langle y, a^k\rangle$.
    Under contrary assumption in Prop.~\ref{prop: measure zero roots}, 
    we then have $0 = \norm{y}^2 \langle y, a^k\rangle$ on $\bb T_{b^k}\setminus\{0\}$.
    Since $\norm{y} \neq 0$,
    it necessarily holds $\langle y, a^k\rangle = 0$ on $\bb T_{b^k}\setminus\{0\}$.
    Furthermore this holds for any $y \in \bb T_{b^k}$,
    hence $a^k$ and $b^k$ are colinear,
    i.e. there exists $\lambda \in \bb R\setminus\{0\}$,
    such that 
    \begin{align*}
        \lambda a^k = b^k 
        \quad \Leftrightarrow \quad 
        \lambda \tilde B^{-1/2} A^{\tT}A v^k = \tilde B^{1/2} v^k 
        \quad \Leftrightarrow \quad 
        \lambda A^{\tT}A v^k = B^{\tT}B v^k.
    \end{align*}
    In turn, this implies
    that $v^k$ is a generalized eigenvector of $(A^{\tT}A, B^{\tT}B)$
    and contradicts the assumption.
    This yields the first assertion.
    For the second assertion, observe similarly that 
    \begin{align*}
        0 &= \alpha_k = \scp{v^k}{A^{\tT}Ax^k}\scp{B^{\tT}Bv^k}{v^k} - \scp{v^k}{B^{\tT}Bx^k}\scp{A^{\tT}Av^k}{v^k} \notag \\
        \Leftrightarrow \quad 
        0 &= \scp{x^k}{\Bigl(\tfrac{\scp{v^k}{A^{\tT}Av^k}}{\scp{v^k}{B^{\tT}Bv^k}} B^{\tT}B - A^{\tT}A\Bigr)v^k}.
    \end{align*}
    To this end,
    define the linear function $\alpha(x)$ in terms of Prop.~\ref{prop: measure zero roots}
    \begin{equation*}
        \alpha(x) \coloneqq \scp{v^k}{A^{\tT}Ax}\scp{B^{\tT}Bv^k}{v^k} - \scp{v^k}{B^{\tT}Bx}\scp{A^{\tT}Av^k}{v^k}
    \end{equation*}
    and substituting with the parameters from \eqref{eq: para alpha beta}
    yields 
    \begin{equation*}
        \alpha(y) = \scp{b^k}{\tilde A y} \norm{b^k}^2 - \scp{y}{b^k} \scp{a^k}{b^k}.
    \end{equation*}
    Again, 
    on $\bb T_{b^k}\setminus\{0\}$, 
    one obtains $\alpha(y) = \langle b^k, \tilde A y\rangle \norm{b^k}^2$
    and by taking the contrary to Prop.~\ref{prop: measure zero roots} into account, 
    it holds $0 = \langle b^k,\tilde A y\rangle$,
    since $\norm{b^k}^2 \neq 0$.
    By the symmetry of $\tilde A$
    and since $0 = \langle \tilde A b^k,y\rangle = 0$ for all $y \in \bb T_{b^k}$,
    the colinearity of $b^k$ and $\tilde A b^k$ follows.
    Hence,
    there exists some $\lambda \in \bb R \setminus\{0\}$ such that
    \begin{equation*}
        \lambda \tilde A b^k = b^k
        \;\Leftrightarrow \; 
        \lambda \tilde B^{-1/2}A^{\tT}A \underbracket{\tilde B^{-1/2} \tilde B^{1/2}}_{= I_d} v^k = \tilde B^{1/2} v^k
        \;\Leftrightarrow\; 
        \lambda A^{\tT}A v^k = B^{\tT}B v^k,
    \end{equation*}
    yielding a generalized eigenvector $v^k$ of $(A^{\tT}A, B^{\tT}B)$,
    contradicting the assumption, again.
\end{proof}

\subsection{Relation to Generalized Eigenspaces}
\label{sec: gen eigenspaces}

In this section,
we translated the non-generalized eigenvector condition of the iterates $v^k$
to an assumption on the size of the spectrum of \eqref{eq: eigen}.

\begin{proposition}
    \label{prop: update}
    If $\# \spec((B^{\tT}B)^{-1/2}A^{\tT}A(B^{\tT}B)^{-1/2}) > 1$,
    then 
    \begin{enumerate}[label=\roman*)]
        \item $v^0 = \tfrac{\hat v}{\norm{\hat v}}$, 
        where $\hat v \sim \mc N(0,I_d)$, is not a generalized eigenvector a.s.,
        \item if $\dim(G_{\lambda_1}) = d-1$,
        then $v^1 \in G_{\lambda_1}$ is a leading generalized eigenvector a.s.,
        \item if $\dim(G_\lambda) < d-1$ 
        and $v^k \not \in G_\lambda$, 
        then $v^{k+1} \not\in G_\lambda$ a.s.
    \end{enumerate}
    If $\# \spec((B^{\tT}B)^{-1/2}A^{\tT}A(B^{\tT}B)^{-1/2}) = 1$,
    then $v^0 \in G_{\lambda_1}$ is a generalized eigenvector a.s.
\end{proposition}

The latter proposition 
claims that the proposed method converges 
either after zero or one iteration 
(the convergence after finitely many iterations)
or proceeds, inductively,  for $k \to \infty$
by Lem.~\ref{lem: gamma_k}.
For proving Prop.~\ref{prop: update},
the next lemma discusses the dimension of the projected sets of search directions
as affine subspaces.

\begin{lemma}
    \label{lem: set directions}
    Let $v \in \bb S^{d-1}$ be no generalized eigenvector of $(A^{\tT}A, B^{\tT}B)$,
    i.e. it holds $v \not\in G_\lambda$ 
    for any $\lambda \in \spec((B^{\tT}B)^{-1/2}A^{\tT}A(B^{\tT}B)^{-1/2})$.
    Then 
    \begin{equation*}
        S_{v,\lambda} \coloneqq \{x \in \bb R^d \mid v + x \in G_\lambda\} \not\ni 0
    \end{equation*}
    is an affine space and has dimension $\dim(G_\lambda)$.
\end{lemma}
\begin{proof}
    Recall that the set $S_{v,\lambda}$ can be rewritten in the following manner:
    \begin{equation*}
        x \in S_{v,\lambda} \quad \Leftrightarrow \quad 
        \exists u \in G_\lambda : v + x = u
        \quad \Leftrightarrow \quad 
        x \in (-v) + G_\lambda.
    \end{equation*}
    This directly provides by the dimension of $G_\lambda$ the dimension of $S_{v,\lambda}$.
    Assume $0 \in S_{v,\lambda}$, this implies $v \in G_\lambda$,
    and contradicts the assumption; 
    this finishes the proof.
\end{proof}

Now Prop.~\ref{prop: measure zero affine space}
i.e.~\cite[Prop.~6.5 \& Cor.~6.12]{Lee2012smoothmanifolds},
can be utilized 
on the finite union of the sets $S_{v,\lambda}$ 
for all $\lambda \in \spec((B^{\tT}B)^{-1/2}A^{\tT}A(B^{\tT}B)^{-1/2})$
and one fixed $v \in \bb S^{d-1}$.
Therefore,
we consider the union of finitely many generalized eigenspaces
\begin{equation*}
    G 
    \coloneq 
    \bigcup_{\lambda \in \spec((B^{\tT}B)^{-1/2}A^{\tT}A(B^{\tT}B)^{-1/2})} G_\lambda
\end{equation*}

\begin{proof}[of Prop.~\ref{prop: update}]
    Assume, firstly,
    $\#\spec((B^{\tT}B)^{-1/2}A^{\tT}A
    (B^{\tT}B)^{-1/2})>1$.

    For part i).
    By Prop.~\ref{prop: measure zero affine space}~i),
    the sets
    $G_\lambda\cap\bb S^{d-1}
    =\varphi(G_\lambda\setminus\{0\})$
    have measure zero with respect to
    $\mc U(\bb S^{d-1})$,
    since $G_\lambda$ is a proper linear subspace.
    According to the construction of $v^0$,
    $\bb P_{V^0}=\mc U(\bb S^{d-1})$.
    Hence
    \begin{equation}
        \label{eq: prop V^0}
        \bb P(v^0\in G_\lambda)=0,
        \quad \text{such that} \quad
        \bb P(v^0\in G)=0,
    \end{equation}
    by $\sigma$-subadditivity of the probability measure 
    over the finitely many eigenspaces.
    
    For part ii).
    By (i), 
    $v^0$ is not a generalized eigenvector a.s.
    Moreover, by Lem.~\ref{lem: set directions}, it holds
    \[
        \dim(S_{v^0,\lambda_1})
        =\dim(G_{\lambda_1})=d-1.
    \]
    Since $v^0\notin G_{\lambda_1}$,
    there is an orthonormal basis
    $\{u_j\}_{j=1}^{d-1}$ of $G_{\lambda_1}$ such that
    $\{v^0,u_1,\ldots,u_{d-1}\}$ is a basis of $\bb R^d$.
    In particular, we obtain
    \[
        S_{v^0,\lambda_1}
        =\left\{-v^0+\sum_{j=1}^{d-1}\eta_ju_j:
        \eta_j\in\bb R\right\}.
    \]
    Writing the sampled direction as
    \[
        x^0=\eta_0v^0+\sum_{j=1}^{d-1}\eta_ju_j,
    \]
    conditional uniform sampling gives $x^0\notin G_{\lambda_1}\cup\,\spa\{v^0\}$
    almost surely, 
    because both are proper linear subspaces.
    Hence $\eta_0\neq0$ and at least one coefficient
    $\eta_j$, $j\geq1$, is nonzero.
    Taking $\tau=-1/\eta_0$ gives
    \[
        \tau x^0\in S_{v^0,\lambda_1},
        \qquad
        v^0+\tau x^0
        =-\frac{1}{\eta_0}\sum_{j=1}^{d-1}\eta_ju_j
        \in G_{\lambda_1}\setminus\{0\}.
    \]
    Thus the sampled projective line contains a global maximizer. 
    Exact maximization
    and $\bb P_{X^0} = \mc U(\bb S^{d-1})$, cf.~Rem.~\ref{rem: dist x^k v^k},
    implies
    \[
        \bb P_{V^1 \mid V^0 = v^0}(v^1\in G_{\lambda_1})
        = \bb P_{X^0}(x^0 \in \varphi(S_{v^0,\lambda_1}) \cup \varphi(-S_{v^0,\lambda_1}))
        = 1,
        \qquad v^0\notin G.
    \]
    By the law of total probability and
    \eqref{eq: prop V^0}, we derive
    \[
        \bb P_{V^1}(v^1\in G_{\lambda_1})
        =\int_{\bb S^{d-1}\setminus G}
        \bb P_{V^1 \mid V^0 = v^0}(v^1\in G_{\lambda_1})
        \,\mathrm d\bb P_{V^0}(v^0)
        = \bb P_{V^0}(v^0 \not \in G)
        = 1.
    \]
    This proves ii) without an assumption that
    the leading generalized eigenvalue is simple.

    For part iii),
    fix $v\notin G_\lambda$.
    The affine-space calculation in
    Lem.~\ref{lem: set directions} gives
    $S_{v,\lambda}=-v+G_\lambda$, with
    $0\notin S_{v,\lambda}$ and
    $\dim(S_{v,\lambda})=\dim(G_\lambda)<d-1$.
    The set
    \[
        \varphi(S_{v,\lambda})
        =\bigl\{\tfrac{z}{\norm{z}} \mid
        z\in S_{v,\lambda}\bigr\}
        \subset\bb S^{d-1}
    \]
    has zero measure w.r.t. $\bb P_{X^k} = \mc U(\bb S^{d-1})$.
    Indeed, it holds
    \[
        \varphi(S_{v,\lambda})
        \subset
        (G_\lambda+\spa\{v\})\cap\bb S^{d-1},
    \]
    where $\dim(G_\lambda+\spa\{v\}) = \dim(G_\lambda)+1<d$.
    Hence, Prop.~\ref{prop: measure zero affine space}~i) applies.
    Therefore, we obtain for $v^k \not\in G_\lambda$ that
    \begin{equation*}
        \bb P_{V^{k+1} \mid V^k = v^k}(v^{k+1} \in G_\lambda)
        \leq  \bb P_{X^k}(x^k \in \varphi(S_{v,\lambda}) \cup \varphi(-S_{v,\lambda}) \cup (G_\lambda \cap \bb S^{d-1})
        = 0
    \end{equation*}
    as finite union of measure zero sets w.r.t. $\bb P_{X^k} = \mc U(\bb S^{d-1})$,
    cf.~Rem.~\ref{rem: dist x^k v^k}.
    By the law of total probability we get
    \[
        \bb P_{V^{k+1}}(v^{k+1}\in G_\lambda)
        =
        \int_{\bb S^{d-1}\setminus G_\lambda}
        \bb P_{V^{k+1} \mid V^k = v^k}(v^{k+1}\in G_\lambda)
        \,\mathrm d\bb P_{V^k}(v)
        =0.
    \]
    Together with (i), 
    induction gives
    $\bb P(v^k\in G_\lambda)=0$ for every $k\geq0$
    whenever $\dim(G_\lambda)<d-1$.
    Finally,
    by the $\sigma$-subadditivity of the probability measure 
    it holds for the finite unions in $G$ that
    $$
        \bb P_{V^{k+1}}(v^{k+1} \in G) = 0,
    $$
    yielding the assertion.

    Assume, 
    secondly, $\# \spec((B^{\tT}B)^{-1/2}A^{\tT}A(B^{\tT}B)^{-1/2}) = 1$: 
    This is the pathological case, 
    and any $v \in \bb S^{d-1}$ is a generalized eigenvector, 
    and so a.s. is $v^0$ by construction.
\end{proof}

As a first and immediate consequence
we can provide the following convergence result for the generated sequence 
of functional values 
\begin{equation}
    \label{eq: update}
    \bigl(f(v^k)\bigr)_{k \in \bb N} 
    = \Bigl(\tfrac{\norm{A v}^2}{\norm{B v}^2}\Bigr)_{k \in \bb N}
    = \Bigl(\tfrac{\scp{v}{A^{\tT}A v}}{\scp{v}{B^{\tT}B v}}\Bigr)_{k \in \bb N}
    = \bigl(\tfrac{a_k}{d_k}\bigr)_{k \in \bb N}.
\end{equation}
Due to Prop.~\ref{prop: update}
it suffices to discuss the cases
where we do not have convergence to the leading generalized eigenspace $G_{\lambda_1}$
after finitely many iterations,
i.e. where $\dim(G_\lambda) < d-1$, or, more specifically, $\dim(G_{\lambda_1}) < d-1$.

\begin{corollary}
    \label{cor: monoton conv}
    Let $\dim(G_{\lambda_1}) < d-1$ 
    and $(a_k)_{k \in \bb N}$ and $(d_k)_{k \in \bb N}$ are the sequences
    of random variables 
    generated by Alg.~\ref{alg: main}. 
    Then the sequence $(\tfrac{a_k}{d_k})_{k \in \bb N}$
    from \eqref{eq: update} converges monotonically nondecreasing, a.s.
    More precisely, 
    it is strictly increasing a.s. if $v^k$ is not critical,
    i.e. $\grad f(v^k) = \nabla f(v^k) \neq 0$.
\end{corollary}

\subsection{Almost Sure Optimal Convergence}
\label{sec: as optimal convergence}

Different from the theoretical results from
\cite[Lem.~2.12~\&~2.13]{B9}
and
\cite[Thm. 5.1 \&~5.2]{B9},
it is not possible to directly prove 
from the strict convergence of the sequence 
of the functional values, cf.~Cor.~\ref{cor: monoton conv},
the convergence $\alpha_k = b_kd_k - a_ke_k \to 0$ for $k \to \infty$, 
almost surely.
However, 
the next result proves directly that the directional derivative at the current iterate vanishes asymptotically.
This argument does not require boundedness of the
exact line-search steps, which can fail near degenerate stationary equations.
From this the following statement can be proven.

\begin{theorem}
    \label{thm: a_k conv}
    Let $\dim(G_{\lambda_1}) < d-1$
    and $(a_k)_{k \in \bb N}$, $(b_k)_{k \in \bb N}$, $(c_k)_{k \in \bb N}$,
    $(d_k)_{k \in \bb N}$, $(e_k)_{k \in \bb N}$, and $(f_k)_{k \in \bb N}$
    be the sequences of random variables generated by Alg.~\ref{alg: main}.
    Then $\alpha_k = b_kd_k - a_ke_k \to 0$ for $k \to \infty$, a.s.
\end{theorem}

For proving the almost sure convergence of
\[
    \alpha_k = \tfrac{d_k^2}{2}s_k'(0) \to 0,
    \quad k \to \infty
    \quad \text{(a.s.)}
\]
we use a uniform local Taylor bound.
For the equation, see \eqref{eq: sub prob stationary}.

\begin{corollary}
    \label{cor: R_k}
    There exists a constant $L>0$, independent of $k\in \bb N$, such that
    \begin{equation*}
        s_k(\tau) - s_k(0)
        \geq s_k'(0) \tau - \tfrac{L}{2} \tau^2
        \qquad
        \forall |\tau| \leq \tfrac{1}{2}.
    \end{equation*}
\end{corollary}
\begin{proof}
    For \smash{$|\tau| \leq \tfrac{1}{2}$}
    and $v^k,x^k\in\bb S^{d-1}$ it holds 
    \smash{$\frac{1}{2} \leq \norm{v^k+\tau x^k} \leq\frac{3}{2}$}.
    By the scale invariance
    consider \smash{$\tau \mapsto s_k(\tau) = f(v^k + \tau x^k) = f(\tfrac{v^k + \tau x^k}{\norm{v^k + \tau x^k}})$}.
    This is a concatenation of the objective $f$, cf. ~\ref{eq: obj},
    and $\tau \mapsto v + \tau x$ for $v, x \in \bb S^{d-1}$.
    Then,
    by Taylor's expansion,
    independent from $k \in \bb N$,
    we obtain
    \begin{equation}
        \label{eq: taylor hess}
        s_k(\tau) - s_k(0) 
        = \tau s_k'(0) + \tfrac{\tau^2}{2} s_k''(\vartheta \tau),
        \quad \text{for some} \quad
        \vartheta \in (0,1).
    \end{equation}
    By the chain rule 
    and Lem.~\ref{lem: lipschitz grad}, we have 
    \begin{displaymath}
        |s_k''(\vartheta \tau)|
        \leq |x^{\tT} \nabla^2 f(v + \vartheta \tau x) x|
        \leq L \norm{x}^2
        = L,
    \end{displaymath}
    Hence all line segments under consideration remain in the compact annulus
    $\{z\in\bb R^d: \tfrac{1}{2}\leq\norm z\leq \tfrac{3}{2}\}$, on which the extended quotient
    $f(z)=\norm{Az}^2/\norm{Bz}^2$ has a uniformly bounded Hessian because
    $B^{\tT}B\succ0$, i.e. the extended gradient has a Lipschitz constant, 
    cf.~Lem.~\ref{lem: lipschitz grad}.
    Taylor's theorem, i.e.~\eqref{eq: taylor hess}, 
    gives
    \begin{equation*}
        s_k(\tau)-s_k(0)
        =s_k'(0)\tau+\frac{\tau^2}{2}
        (x^k)^{\tT}\nabla^2f(v^k+\vartheta\tau x^k)x^k
        \geq s_k'(0)\tau-\frac{L}{2}\tau^2,
    \end{equation*}
    where $L > 0$ is the Lipschitz constant from Lem.~\ref{lem: lipschitz grad}.
\end{proof}

\begin{proof}[of Thm.~\ref{thm: a_k conv}]
    Assume that $\abs{s_k'(0)}$ does not converge to zero.
    Then there are
    $\varepsilon>0$ and a subsequence \smash{$(k_j)_{j \in \bb N}$} such that
    \smash{$|s_{k_j}'(0)|\geq\varepsilon$}. 
    Choose \smash{$0<\tau<\min\{\tfrac{1}{2},\tfrac{\varepsilon}{L}\}$} and set
    $\widetilde\tau_k=\sign(s_k'(0))\tau$. 
    Then, Cor.~\ref{cor: R_k} yields
    \begin{equation*}
        s_{k_j}(\tilde \tau_{k_j}) - s_{k_j}(0)
        \geq |s'_{k_j}(0)| \cdot |\tau| - \tfrac{\varepsilon}{2}|\tau| 
        \geq \varepsilon \abs{\tau} - \tfrac{\varepsilon}{2}|\tau| 
        \geq \tfrac{\varepsilon}{2}\abs{\tau}
        > 0.
        \quad j \in \bb N.
    \end{equation*}
    Since $\tau_{k_j}$ maximizes the line search, 
    the same positive lower bound
    holds for
    \begin{align*}
        s_{k_j}(\tau_{k_j}) - s_{k_j}(0) 
        \geq s_{k_j}(\tilde \tau_{k_j}) - s_{k_j}(0)
        \geq \tfrac{\varepsilon}{2}\abs{\tau}
        > 0,
        j \in \bb N.
    \end{align*}
    This contradicts convergence of
    the monotone bounded sequence in Cor.~\ref{cor: monoton conv}. 
    Thus $s_k'(0)\to0$. Since $\alpha_k = \tfrac{d_k^2}{2}s_k'(0)$ and
    $d_k\leq\norm B^2$, the assertion follows.
\end{proof}

\begin{lemma}
    \label{lem: EE_a_k}
    Let $\dim(G_{\lambda_1}) < d-1$
    and $(v^k)_{k \in \bb N}$
    and $(x^k)_{k \in \bb N}$ be the sequence of random variables 
    generated by Alg.~\ref{alg: main}.
    It holds 
    \begin{align*}
        \bb E_{x^k \sim \mc U(\bb S^{d-1}) \mid V^k = v^k}\Bigl[\abs{\tfrac{\alpha_k}{d_k}}^2\Bigr] 
        &= \tfrac{1}{d}\norm[\Big]{\Bigl(\tfrac{\scp{v^k}{A^{\tT}Av^k}}{\scp{v^k}{B^{\tT}Bv^k}} B^{\tT}B - A^{\tT}A\Bigr)v^k}^2 \\
        &= \tfrac{d_k^2}{4d}  \norm{\grad f(v^k)}^2
        \to 0,
        \qquad k \to \infty.
    \end{align*}
\end{lemma}
\begin{proof}
    For the $L^1$-convergence,
    recall that Thm.~\ref{thm: a_k conv} shows
    $\alpha_k = b_k d_k - a_k e_k \to 0$ for $k \to \infty$ (a.s.).
    Since 
    \begin{align*}
        \abs[\Big]{\Bigl\langle x^k, \Bigl(\tfrac{\scp{v^k}{A^{\tT}Av^k}}{\scp{v^k}{B^{\tT}Bv^k}} B^{\tT}B - A^{\tT}A\Bigr)v^k \Bigr\rangle}
        &= \frac{\abs{b_k d_k - a_k e_k}}{d_k} %
        = \frac{\abs{\alpha_k}}{d_k} \\
        &\leq \frac{\abs{\alpha_k}}{\lambda_d(B^{\tT}B)} 
        \leq 2\kappa(B^{\tT}B) \norm{A^{\tT}A},
    \end{align*}
    Lebesgue's dominated convergence theorem
    together with the law of total expectation yields the first assertion.
    For the second equality,
    rewrite the conditional expectation 
    w.r.t. the conditional law, which is independent of $v^k$ 
    and given by the uniform distribution on the unit sphere
    \begin{align*}
        &\bb E_{x^k \sim \mc U(\bb S^{d-1}) \mid V^k = v^k}[(\tfrac{\alpha_k}{d_k})^2]\\
        &= \bb E_{x^k \sim \mc U(\bb S^{d-1}) \mid V^k = v^k}
        \scalebox{0.78}{$\Bigl[\Bigl\langle \Bigl(\tfrac{\scp{v^k}{A^{\tT}Av^k}}{\scp{v^k}{B^{\tT}Bv^k}} B^{\tT}B - A^{\tT}A\Bigr)v^k, x^k \Bigr\rangle
        \Bigl\langle x^k, \Bigl(\tfrac{\scp{v^k}{A^{\tT}Av^k}}{\scp{v^k}{B^{\tT}Bv^k}} B^{\tT}B - A^{\tT}A\Bigr)v^k \Bigr\rangle\Bigr]$}
        \\
        &= \scalebox{0.87}{$(v^k)^{\tT}\Bigl(\tfrac{\scp{v^k}{A^{\tT}Av^k}}{\scp{v^k}{B^{\tT}Bv^k}} B^{\tT}B - A^{\tT}A\Bigr)$}
        \bb E_{x^k \sim \mc U(\bb S^{d-1})}[x^k(x^k)^{\tT}] 
        \scalebox{0.87}{$\Bigl(\tfrac{\scp{v^k}{A^{\tT}Av^k}}{\scp{v^k}{B^{\tT}Bv^k}} B^{\tT}B - A^{\tT}A\Bigr)v^k$} 
        \\
        &= \tfrac{1}{d} \norm[\Big]{\Bigl(\tfrac{\scp{v^k}{A^{\tT}Av^k}}{\scp{v^k}{B^{\tT}Bv^k}} B^{\tT}B - A^{\tT}A\Bigr)v^k}^2 \\
        &= \tfrac{1}{d} \norm[\big]{(A^{\tT}A - f(v^k) B^{\tT}B) v^k}^2
        = \tfrac{1}{d} \norm[\big]{C(v^k) v^k}^2,
    \end{align*}
    using \eqref{eq: x^k}~f. 
    Finally,
    we utilize \eqref{eq: riemannian gradient}~f., 
    i.e. $\tfrac{d_k}{2}\grad f(v^k)=C(v^k)v^k$, cf.~\eqref{eq: euclidean grad}.
\end{proof}

\begin{figure}[h]
    \centering
    \includegraphics[width=\linewidth]{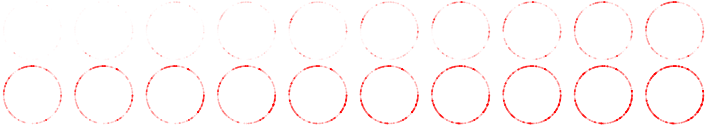}
    \caption{Visualization of the great circle $\bb S^1$ of the $\bb S^3$ 
    with the iterates $v^k \in \bb S^3 \hookrightarrow \bb S^1$,
    whenever $\abs{v_3^k}+\abs{v_4^k} < 10^{-3}$, generated by Alg.~\ref{alg: main}
    for $A = \diag(1,1,0,0)$ and $B = I_4$.
    Here visualizations are provided for the iterations $k \in \{1.000, 1.500,2.000,...,10.000\}$
    (starting: top left, towards: bottom right).}
    \label{fig: leading eigenspace iter}
\end{figure}

The existence of a subsequence, proven in the following proposition,
can be observed in Fig.~\ref{fig: leading eigenspace iter}.
As Alg.~\ref{alg: main} converges after finitely many iterations 
or $\gamma_k \neq 0$ for all $k \in \bb N$ (a.s.), see Prop.~\ref{prop: update} and Cor.~\ref{cor: monoton conv},
the following \textit{global} convergence result can be claimed.
Here, the distance of the iterates $v^k$ to the leading set of generalized eigenspaces
is measured by
\begin{equation*}
    \dist(\mc C, y) \coloneqq \inf_{x \in \mc C} \norm{y - x},
\end{equation*}
which is finite for closed convex sets in finite dimensions 
and attained at some $s \in \mc C$.
Differently \cite{B8,B9},
we are not aware of the convergence for the sequence of random vectors,
even if the dimension of the leading generalized eigenspace is one.
This is due to the unconstrained sampling technique, cf.~Fig.~\ref{fig: leading eigenspace iter}.
However, the convergence is \textit{global} in the sense
that the functional value converges to the global maximum
and the sequence of random variables converges 
to the leading generalized eigenspace.
This is, in general, absent in the literature for zeroth-order methods, 
cf.~\cite{li2023stochastic,balasubramanian2022zeroth} 
and the discussion in \cite[§~2.2.]{B9} for the related generalized Rayleigh quotient.

\begin{proposition}
    \label{prop: as conv eigenspace}
    Let $\dim(G_{\lambda_1}) < d-1$
    and $(v^k)_{k \in \bb N}$ be the sequence of random variables 
    generated by Alg.~\ref{alg: main}.
    Then there exists a subsequence \smash{$(v^{k_j})_{j \in \bb N}$}
    and a random variable $\lambda$ a.s. taking values in \smash{$\spec((B^{\tT}B)^{-1/2}A^{\tT}A(B^{\tT}B)^{-1/2})$}
    such that 
    \smash{$\dist(G_\lambda, v^{k_j}) \to 0$ for $j \to \infty$} a.s.
    Moreover,
    the sequence \smash{$(\tfrac{a_k}{d_k})_{k \in \bb N}$} from \eqref{eq: update}
    converges to $\lambda$ a.s.
\end{proposition}
\begin{proof}
    The proof is similar to \cite[Prop.~5.4]{B9};
    see Appendix~\ref{sec: app}.
\end{proof}

\begin{theorem}
    \label{thm: convergence as}
    Let $(v^k)_{k \in \bb N}$ be the sequence of random variables generated by Alg.~\ref{alg: main}.
    Then 
    \begin{equation*}
        \tfrac{\norm{A v^k}^2}{\norm{B v^k}^2} 
        \nearrow \lambda_1,
        \quad \text{and} \quad
        \dist(G_{\lambda_1}, v^k) \to 0,
        \quad \text{as} \quad 
        k \to \infty,
        \quad (a.s.).
    \end{equation*}
\end{theorem}

For proving the main convergence result Thm.~\ref{thm: convergence as}, 
the following constructive lemma is required,
which is visualized in Fig.~\ref{fig: sphere_proj}.
It is similar to \cite[Lem.~5.5]{B9} and \cite[Lem.~2.17]{B8}
using the geodesic distances.
It shows 
that the probability to sample into an $\varepsilon$-neighborhood 
$\bb B_\varepsilon(\tilde v) \coloneq \{x \in \bb R^d \mid \norm{\tilde v - x} < \varepsilon\}$
of an optimal generalized eigenvector is uniformly lower bounded away from zero
w.r.t. the uniform distribution on the unit sphere.

\begin{figure}
    \centering
    \includegraphics[width=0.55\linewidth,clip=true, trim=5pt 5pt 5pt 30pt]{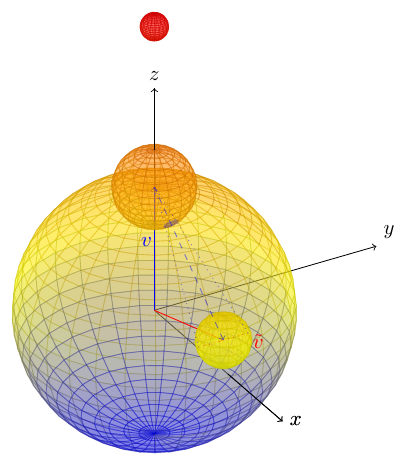}
    \caption{Visualization of the claim in Lem.~\ref{lem: uniform lower bound} 
    for a uniform lower bound on the probability of reaching $\bb B_\varepsilon(\tilde v)$ 
    for any $v \in \bb S^{d-1}$ is provided.
    This is the situation in the proof of Thm.~\ref{thm: convergence as} 
    as it is used in \eqref{eq: uniform lower bound}}.
    \label{fig: sphere_proj}
\end{figure}

\begin{lemma}
    \label{lem: uniform lower bound}
    Let $v, \tilde v \in \bb S^{d-1}$ and $\varepsilon > 0$. 
    Define 
    \begin{align*}
        \mc D_v \coloneqq \left\{ x \in \bb S^{d-1} 
        \;\middle|\; 
        \exists \tau \in \bb R, 
        R_v(\tau x) 
        = \tfrac{v + \tau x}{\norm{v + \tau x}} \in \bb B_\varepsilon(\tilde v) \cup \bb B_\varepsilon(-\tilde v)\right\}.
    \end{align*}
    There exists some $p_{\varepsilon,d} > 0$ 
    such that $\inf_{v \in \bb S^{d-1}} \sigma_{d-1}(\mc D_v) > p_{\varepsilon,d}$.
\end{lemma}
\begin{proof}
    See Appendix~\ref{sec: app}.
\end{proof}

Finally, for proving Thm.~\ref{thm: convergence as},
we consider the set 
\begin{align*}
    \label{eq: set L}
    \mc L 
    \coloneqq 
    \left\{v \in \bb S^{d-1} 
    \;\middle|\; 
    \tfrac{\norm{A v^k}^2}{\norm{B v^k}^2} 
    = \tfrac{\scp{v}{A^{\tT}Av}}{\scp{v}{B^{\tT}Bv}} 
    = f(v) 
    > \tilde \lambda \right\},
    \quad \tilde \lambda \coloneqq \max_{\lambda_i < \lambda_1 } \lambda_i.
\end{align*}
and utilize Lem.~\ref{lem: uniform lower bound}
and refer to \cite[Thm.~2.17]{B8}
in light of a similar proof technique. 
On the other hand, one might prove it quite similar by utilizing 
the Borel--Cantelli lemma \cite[VII.~§~5]{renyi1966wahrscheinlichkeitsrechnung}.

\begin{proof}[of Thm.~\ref{thm: convergence as}]
    Take the set $\mc L$ (above) into account 
    and define the events $E_k \coloneqq \{V^k \not\in \mc L\}$.
    Then,
    by the monotonicity of the generated functional values, cf.~Cor.~\ref{cor: monoton conv},
    we have $E_{k+1} \subseteq E_k$.
    Furthermore,
    we have 
    \begin{equation}
        \label{eq: E_k+1}
        \bb P(E_{k+1})
        = \bb P_{V^{k+1}}(v^{k+1} \not\in \mc L)
        = \int_{\bb S^{d-1} \setminus \mc L} \hspace{-10pt}
        \bb P_{V^{k+1} \mid V^k = v^k} (v^{k+1} \not\in \mc L)
        \; \mathrm d \bb P_{V^k}(v^k),
    \end{equation}
    since we have 
    \begin{equation*}
        \bb P_{V^{k+1} \mid V^k = v^k} (v^{k+1} \not\in \mc L)
        = 0,
        \qquad \text{if} \qquad 
        v^k \in \mc L.
    \end{equation*}
    Furthermore,
    we have by the continuity of $f$, cf.~\ref{eq: obj}, from Lem.~\ref{lem: lipschitz func}
    some $\varepsilon > 0$ and $\tilde v \in \bb S^{d-1}$ such that 
    \begin{equation*}
        \bb B_{\varepsilon,\pm \tilde v} 
        \coloneqq (\bb B_{\varepsilon}(\tilde v) \cup \bb B_{\varepsilon}(-\tilde v)) 
        \cap \bb S^{d-1} 
        \subset \mc L
    \end{equation*}
    and hence 
    \begin{align*}
        \bb P_{v^{k+1} \sim V^{k+1} \mid V^k = v^k} (v^{k+1} \not \in \mc L)
        &= \bb P_{x^{k} \sim \mc U(\bb S^{d-1}) \mid V^k = v^k} (R_{v^k}(\tau_k x^k) \not \in \mc L) \\
        &\leq \bb P_{x^k \sim \mc U(\bb S^{d-1})} (x^k \not \in \mc D_{v^k}) \\
        &\leq \bb P_{x^k \sim \mc U(\bb S^{d-1})} (x^k \not \in \bb B_{\varepsilon,\pm \tilde v})
        < 1 - p_{\varepsilon,d}
    \end{align*}
    by utilizing Lem.~\ref{lem: uniform lower bound}
    for some $p_{\varepsilon,d} \in (0,1)$.
    Therewith,
    by \eqref{eq: E_k+1} we obtain 
    \begin{equation*}
        \bb P(E_{k+1})
        < (1 - p_{\varepsilon, d}) \bb P_{V^k}(v^k \not \in \mc L)
        = (1 - p_{\varepsilon, d}) \bb P(E_{k}),
        \qquad k \in \bb N.
    \end{equation*}
    Hence, we iteratively have 
    $\bb P(E_k) < (1 - p_{\varepsilon, d})^k \bb P(E_0)$, 
    where $\bb P(E_0) \leq 1$.
    Now, Borel--Cantelli lemma \cite[VII.~§~5]{renyi1966wahrscheinlichkeitsrechnung}
    yields the assertion,
    due to 
    \begin{equation*}
        \sum_{k = 0}^\infty \bb P (E_k)
        < \bb P(E_0)\sum_{k = 1}^\infty (1 - p_{\varepsilon,d})^k
        = \bb P(E_0) \frac{1}{1 - (1 - p_{\varepsilon,d})}
        = \frac{\bb P(E_0)}{p_{\varepsilon, d}}
        < \infty,
    \end{equation*}
    such that 
    $f(v^k) = \nicefrac{\scp{v^k}{A^{\tT}A v^k}}{\scp{v^k}{B^{\tT}B v^k}} \nearrow \lambda_1$ as $k \to \infty$ a.s.,
    by Prop.~\ref{prop: as conv eigenspace}.

    For the second assertion,
    let $\{u_1,...,u_d\} \subset \bb S^{d-1}$ be an orthonormal basis 
    of eigenvectors of $(B^{\tT}B)^{-1/2} A^{\tT}A (B^{\tT}B)^{-1/2}$.
    We define, due to \eqref{eq: eigen} ff., $w^k \coloneqq (B^{\tT}B)^{1/2} v^k$ such that 
    \begin{align*}
        f(v^k)
        &= \frac{\scp{v^k}{A^{\tT}A v^k}}{\scp{v^k}{B^{\tT}B v^k}}
        = \frac{\scp{w^k}{(B^{\tT}B)^{-1/2}A^{\tT}A (B^{\tT}B)^{-1/2} w^k}}{\norm{w^k}^2} \\
        &= \frac{\sum_{i,j = 1}^d \scp{w^k}{u_i} \scp{w^k}{u_j} \lambda_j \scp{u_i}{u_j}}{\norm{w^k}^2}
        = \frac{\sum_{i = 1}^d \scp{w^k}{u_i}^2 \lambda_i}{\norm{w^k}^2}
        \leq \lambda_1. 
    \end{align*}
    As we know the convergence a.s. for $f(v^k) \nearrow \lambda_1$ as $k \to \infty$,
    we conclude that $|\langle w^k, u_i\rangle| \to 0$ as $k \to \infty$
    if $u_i$ is not an eigenvector corresponding to $\lambda_1$.
\end{proof}

\begin{remark}[Convergence Rates]
    \label{rem: absent of rate}
    Convergence rates toward global maximizers 
    are established in \cite{B8,B9,B12},
    and toward critical points in \cite{li2023stochastic}.
    Here, Lem.~\ref{lem: uniform lower bound} yields a
    geometric probability bound for each fixed objective accuracy:
    \[
        \bb P\bigl(\lambda_1-f(v^k)>\varepsilon\bigr)
        \leq(1-p_{\varepsilon/L_f,d})^k,
        \qquad \varepsilon>0,
    \]
    where $L>0$ is a Lipschitz constant from Lem.~\ref{lem: lipschitz func}.
    Indeed, 
    a projective line reaching the $\varepsilon/L$-neighborhood 
    of a sphere maximizer attains objective error below $\varepsilon$.
    Then, 
    exact plane maximization and successive conditioning give the bound.
    This guarantee applies to Alg.~\ref{alg: main}.
    It establishes neither a dimension-independent rate
    nor linear decay of the objective error.
\end{remark}

\subsection{Estimating the Riemannian Gradient and Hessian}
\label{sec: estimators}

Similar to \cite{B9},
Alg.~\ref{alg: main} can be varied using $q \in \bb N$ iid. sampled search directions
to estimate the Riemannian gradient.
By Thm.~\ref{thm: riemannian style}, 
the scalar $\alpha_k=b_kd_k-a_ke_k$ provides a moment estimator
of the Riemannian gradient
without finite differences.
Moreover,
the identity below shows that a properly rescaled average is an unbiased version of the Riemannian gradient.
Furthermore,
as Thm.~\ref{thm: svd} provides the optimal step size 
for the approximation of the sampled step direction,
the proposed method is the optimal possible construction 
within the sampled two-dimensional projective subspace.
Recall that due to the new sampling procedure, 
the distribution is not uniform.
However, together with exact line search, a quantitative ascent bound
implies almost-sure convergence of the gradient norm to zero.
This establishes convergence to the critical set for Alg.~\ref{alg:cmgRq}.
The global optimality is not asserted and is addressed in the numerical comparison.

\begin{algorithm}[htb]
    \caption{Full-Sphere $q$-Sample Gradient Method} 
    \label{alg:cmgRq}
    \begin{algorithmic}
    \State{\textbf{Given} $A\in\bb R^{m\times d}$, $B\in\bb R^{\ell\times d}$ with $\rank(B)=d$, and $q\in\bb N$.}
    \State{\textbf{Initialize} $v^0\in\bb S^{d-1}$.}
    \For{$k=0,1,2,\ldots$}
    \State{\textbf{Sample} $x_i^k=y_i/\norm{y_i}$ with $y_i\overset{\mathrm{iid.}}{\sim}\mc N(0,I_d)$, $i=1,\ldots,q$.}
    \State{\textbf{Form} $\alpha_{k,i}=b_{k,i}d_k-a_ke_{k,i}$ and
    \begin{equation*}
        \widehat g_q^k
        =P_{v^k}\left(\frac{2d}{d_k^2q}\sum_{i=1}^{q}\alpha_{k,i}x_i^k\right).
    \end{equation*}}
    \State{\textbf{Stop} if $\widehat g_q^k=0$; otherwise set $x^k=\widehat g_q^k/\norm{\widehat g_q^k}$.}
    \State{\textbf{Calculate} the maximizing projective step from the leading generalized eigenvector of $(M_k,N_k)$ as in Thm.~\ref{thm: svd}.}
    \State{\textbf{Update} $v^{k+1}=x^k$ if $\tau_k=\infty$; otherwise set
        $v^{k+1}=R_{v^k}(\tau_kx^k)=\tfrac{v^k+\tau_kx^k}{\norm{v^k+\tau_kx^k}}$.}
    \EndFor
    \State{\textbf{Return} $\norm{A/B}\approx\norm{Av^k}/\norm{Bv^k} = \sqrt{a_k/d_k}$.}
     \end{algorithmic}
\end{algorithm}

\begin{theorem}
    \label{thm: riemannian style}
    It holds
   \begin{equation*}
        \bb E_{x^k \sim \mc U(\bb S^{d-1}) \mid V^k = v^k}[\alpha_k x^k]
        = \tfrac{d_k^2}{2d}\nabla f(v^k)
        = \tfrac{d_k^2}{2d}\grad f(v^k).
    \end{equation*}
    Equivalently, $\frac{2d}{d_k^2}\alpha_kx^k$ is an unbiased
    estimator of $\grad f(v^k)$.
\end{theorem}
\begin{proof}
    For the first statement,
    simply rewrite the conditional expectation
    \begin{align*}
        &\bb E_{x^k \sim \mc U(\bb S^{d-1}) \mid V^k = v^k}[\alpha_k x^k] \\
        &= \bb E_{x^k \sim \mc U(\bb S^{d-1})\mid V^k = v^k}[(b_k d_k - a_k e_k) x^k] \\
        &= \bb E_{x^k \sim \mc U(\bb S^{d-1})\mid V^k = v^k}[x^k (x^k)^{\tT} \bigl(\langle B^{\tT}B v^k, v^k\rangle A^{\tT}A v^k - \langle A^{\tT}A v^k, v^k\rangle B^{\tT}B v^k\bigr)] \\
        &= \bb E_{x^k \sim \mc U(\bb S^{d-1})}[ x^k (x^k)^{\tT}] \bigl(\langle B^{\tT}B v^k, v^k\rangle A^{\tT}A v^k - \langle A^{\tT}A v^k, v^k\rangle B^{\tT}B v^k\bigr) \\
        &= \tfrac{1}{d} \bigl(\langle B^{\tT}B v^k, v^k\rangle A^{\tT}A v^k - \langle A^{\tT}A v^k, v^k\rangle B^{\tT}B v^k\bigr) \\
        &= \tfrac{d_k}{d} \bigl(A^{\tT}A v^k - \tfrac{a_k}{d_k} B^{\tT}B v^k\bigr)
        = \tfrac{d_k}{d} C(v^k) v^k,
    \end{align*}
    using \eqref{eq: x^k}~f., \eqref{eq: para mgRq}~f., 
    as well as \eqref{eq: riemannian gradient} and \eqref{eq: update},
    i.e. $\frac{d_k}{2}\grad f(v^k)=C(v^k)v^k$.
\end{proof}

\begin{remark}[Ascent and Convergence to the Critical Set]
\label{rem: gradient critical set}
For fixed iterate $v^k\in\bb S^{d-1}$ and $g_k \coloneqq \grad f(v^k)$,
where $\alpha_{k,i}=\tfrac{d_k^2}{2}\scp{g_k}{x_i^k}$,
the projected estimator satisfies
\[
    \widehat g_q^k
    = \tfrac{d}{q}\sum_{i=1}^q \scp{g_k}{x_i^k}P_{v^k}x_i^k.
\]
Using $P_{v^k}g_k=g_k$, we obtain
\[
    \scp{g_k}{\widehat g_q^k}
    =\tfrac{d}{q}\sum_{i=1}^q\scp{g_k}{x_i^k}^2
    >0
    \quad\text{a.s., whenever} \;
    g_k = \grad f(v^k) \neq 0,
\]
and hence an ascent direction 
as long as $v^k$ is not critical.
Moreover, Cauchy--Schwarz gives
$\norm{\widehat g_q^k}^2\leq
d\scp{g_k}{\widehat g_q^k}$.
Thus, for $x^k=\widehat g_q^k/\norm{\widehat g_q^k}$,
conditional unbiasedness implies
\[
    \bb E_{X^k \mid V^k = v^k}[\scp{g_k}{x^k}^2]
    \geq\tfrac{1}{d}
    \bb E_{X^k \mid V^k = v^k}[\scp{g_k}{\widehat g_q^k}]
    =\tfrac{\norm{g_k}^2}{d},
\]
Now, by Taylor's inequality along great circles 
and exact plane maximization yield
\[
    \bb E_{X^k \mid V^k = v^k}[f(v^{k+1})-f(v^k)]
    \geq\tfrac{\norm{g_k}^2}{2Ld},
\]
where $L$ is the bound of the Riemannian Hessian, cf.~Lem.~\ref{lem: lipschitz grad}.
Since $f$ is bounded above, telescoping gives $\sum_k\bb E[\norm{g_k}^2]<\infty$, 
hence $\norm{\grad f(v^k)}\to0$ almost surely by the Borel--Cantelli lemma \cite[Chapter VII, § 5]{renyi1966wahrscheinlichkeitsrechnung}.
Projection onto the tangent space simplifies this argument;
however, omitting it does not change the search plane, since
\[
    \spa\{v^k,\widehat g_q^k\}
    =\spa\{v^k,P_{v^k}\widehat g_q^k\}.
\]
Thus the same critical-set convergence conclusion holds
without projection when exact plane maximization is used.
Although omitting projection removes the tangent-support
obstruction to an absolute-continuity argument, 
that the bad events ($\gamma_k = 0$, i.e. $v^k$ is a generalized eigenvector) 
remain as measure zero sets.
Hence the almost sure convergence of Alg. ~\ref{alg:cmgRq} can be obtained
from the theory, which is here provided for a single sample direction.
\end{remark}

Next,
we derive an unbiased estimator of
the Riemannian Hessian \eqref{eq: riemannian hess special}. 
The idea is to obtain a Newton-style or quasi-Newton zeroth-order method \cite{gill1972quasinewton,Antoniou2007}.
Therefore, we first study the expectation 
of second- and fourth-order samples
for an arbitrary, quadratic matrix in Prop.~\ref{prop: rie hess}.
Secondly, Cor.~\ref{cor: riemannian hessian estimator} applies these
moments to the matrix $H(v)$ using a forward-evaluable weight,
then projects and rescales them to estimate the Riemannian Hessian.

\begin{proposition}
    \label{prop: rie hess}
    Let $C \in \bb R^{d \times d}$
    and $\symC \coloneqq \tfrac{1}{2}(C + C^{\tT})$ denotes the Hermitian part of $C$.
    It holds
    \begin{equation*}
        \bb E_{x \sim \mc U(\bb S^{d-1})}[\scp{x}{Cx}] 
        = \tfrac{\trace(\symC)}{d} %
        \quad\text{and}\quad
        \bb E_{x \sim \mc U(\bb S^{d-1})}[\scp{x}{Cx} xx^{\tT}] 
        = \tfrac{2\symC + \trace(C)I_d}{d(d+2)}.
    \end{equation*}
\end{proposition}
\begin{proof}
    See Appendix~\ref{sec: app}.
\end{proof}

\begin{corollary}
    \label{cor: riemannian hessian estimator}
    Let $A\in\bb R^{m\times d}$, $B\in\bb R^{\ell\times d}$,
    and $\rank(B)=d$.
    For $z\in\bb R^d$ let
    \begin{align*}
        h_v(z)
        &\coloneqq \norm{Az}^2-f(v)\norm{Bz}^2
        -\tfrac{4}{\norm{Bv}^2}
        \bigl(\scp{Av}{Az}-f(v)\scp{Bv}{Bz}\bigr)\scp{Bv}{Bz} \\
        &\,=\scp{z}{H(v)z}.
    \end{align*}
    It holds
    \begin{align*}
        \bb E_{x\sim\mc U(\bb S^{d-1})}
        \left[h_v(x)\left(xx^{\tT}-\frac{I_d}{d+2}\right)\right]
        &=\frac{2}{d(d+2)} H(v),\\
        \bb E_{y\sim\mc U(\bb T_v\cap\bb S^{d-1})}
        \left[h_v(y)\left(yy^{\tT}-\frac{P_v}{d+1}\right)\right]
        &=\frac{\norm{Bv}^2}{(d-1)(d+1)}\hess f(v).
    \end{align*}
\end{corollary}
\begin{proof}
    See Appendix~\ref{sec: app}.
\end{proof}

Cor.~\ref{cor: riemannian hessian estimator}
provides the following unbiased estimators of $\hess f(v)$
in its ambient representation.
For iid. full-sphere samples $x_i$ or iid. tangent-sphere
samples $y_i$, respectively, define
\begin{align}
    \label{eq: unbiased full Hessian estimator}
    \widehat H_{q,\mathrm{full}}(v)
    &=\frac{d(d+2)}{\norm{Bv}^2q}\sum_{i=1}^q
    h_v(x_i)\left(P_vx_ix_i^{\tT}P_v-\frac{P_v}{d+2}\right),\\
    \label{eq: unbiased tangent Hessian estimator}
    \widehat H_{q,\mathrm{tan}}(v)
    &=\frac{(d-1)(d+1)}{\norm{Bv}^2q}\sum_{i=1}^q
    h_v(y_i)\left(y_iy_i^{\tT}-\frac{P_v}{d+1}\right).
\end{align}
Using the estimator via the tangent space \eqref{eq: unbiased tangent Hessian estimator}
yields Alg.~\ref{alg:czmgRq}.
Crucially, there are, again, only forward evaluations necessary.

\begin{algorithm}[htb]
    \caption{Regularized Curvature-Based $q$-Sample Calculation}
    \label{alg:czmgRq}
    \begin{algorithmic}
    \State{\textbf{Given} $A\in\bb R^{m\times d}$, $B\in\bb R^{\ell\times d}$
    with $\rank(B)=d$, $d\geq2$, $q\in\bb N$, and $\delta>0$.}
    \State{\textbf{Initialize} $v^0\in\bb S^{d-1}$.}
    \For{$k=0,1,2,\ldots$}
    \State{\textbf{Construct} the projected gradient estimate $\widehat g_q^k\in\bb T_{v^k}$ as in Alg.~\ref{alg:cmgRq}.}
    \State{\textbf{Stop} if $\widehat g_q^k=0$. 
    \textbf{Cache} $a_k=\norm{Av^k}^2$ and $d_k=\norm{Bv^k}^2$.}
    \State{\textbf{Sample} $z_i^k=P_{v^k}y_i/\norm{P_{v^k}y_i}$ with
    fresh $y_i\overset{\mathrm{iid.}}{\sim}\mc N(0,I_d)$, $i=1,\ldots,q$.}
    \State{\textbf{Evaluate}
    $b_{k,i}=\scp{Av^k}{Az_i^k}$, $c_{k,i}=\norm{Az_i^k}^2$,
    $e_{k,i}=\scp{Bv^k}{Bz_i^k}$, $f_{k,i}=\norm{Bz_i^k}^2$,}
    \begin{equation*}
        h_{k,i}=c_{k,i}-\frac{a_k}{d_k}f_{k,i}
       -\frac{4}{d_k}\left(b_{k,i}-\frac{a_k}{d_k}e_{k,i}\right)e_{k,i}.
    \end{equation*}
    \State{\textbf{Form} the conditionally unbiased tangent Hessian estimate}
    \begin{equation*}
        \widehat H_q^k=\frac{(d-1)(d+1)}{d_kq}\sum_{i=1}^q
        h_{k,i}\left(z_i^k(z_i^k)^{\tT}-\frac{P_{v^k}}{d+1}\right).        
    \end{equation*}
    \State{\textbf{Construct} a matrix $Q_k\in\bb R^{d\times(d-1)}$,
    $Q_kQ_k^{\tT}=P_{v^k}$, $Q_k^{\tT}Q_k=I_{d-1}$.}
    \State{\textbf{Set} $S_k=Q_k^{\tT}\widehat H_q^kQ_k$ and choose
    $\lambda_k\geq\max\{0,\lambda_{\max}(S_k)+\delta\}$.}
    \State{\textbf{Solve} $(-S_k+\lambda_k I_{d-1})p^k=Q_k^{\tT}\widehat g_q^k$.}
    \State{\textbf{Set} $x^k=Q_kp^k/\norm{Q_kp^k}$.}
    \State{\textbf{Calculate} the maximizing projective step from the leading generalized eigenvector of $(M_k,N_k)$ as in Thm.~\ref{thm: svd}.}
    \State{\textbf{Update} $v^{k+1}=x^k$ if $\tau_k=\infty$; otherwise set
        $v^{k+1}=R_{v^k}(\tau_kx^k)$.}
    \EndFor
    \State{\textbf{Return} $\norm{A/B}\approx\norm{Av^k}/\norm{Bv^k}$.}
    \end{algorithmic}
\end{algorithm}

\begin{remark}[Finite Samples, Regularization, and Convergence]
    \label{rem: conv alg 3}
    The finite-sample Hessian can be singular or indefinite.
    The shift in Alg.~\ref{alg:czmgRq} makes the reduced system
    positive definite with smallest eigenvalue at least $\delta$.
    If $\widehat g_q^k\neq0$, then $p^k\neq0$ and
    \[
    (\widehat g_q^k)^{\tT}Q_kp^k
    \geq\delta\norm{p^k}^2>0.
    \]
    This is ascent relative to the sampled gradient
    and exact plane maximization guarantees nondecrease 
    of the objective. 

    For convergence,
    assume conditionally independent gradient and curvature batches, 
    a shift chosen from the curvature data only, 
    and an exact reduced solve.
    The bounded curvature weights imply $\norm{S_k}\leq K$
    for some deterministic $K<\infty$, so
    \[
        \kappa(\lambda_kI-S_k)\leq1+\tfrac{2K}{\delta}
        \eqqcolon\overline\kappa.
    \]
    Conditional unbiasedness, Jensen's inequality, 
    and $\norm{\widehat g_q^k}\leq d\norm{\grad f(v^k)}$
    then give
    \[
        \bb E_{X^k \mid V^k = v^k}[f(v^{k+1})-f(v^k)]
        \geq
        \tfrac{\norm{\grad f(v^k)}^2}{2Ld^2\overline\kappa^2}.
    \]
    As in Rem.~\ref{rem: gradient critical set},
    telescoping and the first Borel--Cantelli lemma yield
    vanishing of the Riemannian gradient, i.e. $\norm{\grad f(v^k)}\to0$,
    almost surely.
    Every accumulation point is therefore critical almost surely.
\end{remark}

\section{Computational Complexity}
\label{sec: complexity}

We distinguish forward-operator calls from vector arithmetic. 
Let $T_A$ and $T_B$ denote the costs of applying $A$ and $B$ to one vector. 
Speaking of matrices, this is in $\mc O(md) + \mc O(\ell d) = \mc O(d(m+\ell))$.
In Alg.~\ref{alg: main}, $Av^k$ and $Bv^k$ can be cached,
and each iteration needs one new application of each operator to $x^k$. 
The remaining work is $\mathcal O(d+m+\ell)$,
while the $2\times2$ generalized eigenproblem costs $\mathcal O(1)$. 
Thus the per-iteration cost is $T_A+T_B+\mathcal O(d+m+\ell)$,
with linear storage in the input and output vector dimensions.

The $q$-sample gradient variant requires $q$ applications of each operator
and $\mathcal O(q(d+m+\ell))$ vector work. 
Because $A$ and $B$ are linear,
the images of the averaged direction can be assembled from the sampled images, 
so no additional forward call is required for the line search.

For the curvature variant, 
explicitly assembling the stochastic Hessian costs $\mathcal O(qd^2)$
and a dense factorization costs $\mathcal O(d^3)$. 
This dense implementation is not suitable for the large-scale matrix-free regime. 
Instead, the sum of rank-one terms and a scalar multiple of $P_{v^k}$
applies $\widehat H_q^k$ to a vector in $\mc O(qd)$ operations.
For tangent sampling, all four scalars in $h_{k,i}$ use the same forward
images, so the correction requires no additional operator calls.
A sufficient shift can be obtained from
$\norm{\widehat H_q^k}\leq
\frac{(d-1)(d+1)}{d_kq}\sum_i|h_{k,i}|$
by taking this bound plus $\delta$ as $\lambda_k$.
The curvature direction generally requires one additional application of each
operator for its exact line search.
An iterative tangent-space solve with $r$ inner iterations
therefore costs $\mathcal O(rqd)$ in addition to the forward applications
used for the gradient and curvature samples. Storing all sampled
directions costs $\mathcal O(qd)$.

These costs clarify the intended use: 
Alg.~\ref{alg: main}, respectively Alg.~\ref{alg:cmgRq}, is fully matrix-free,
whereas the curvature variant, cf.~Alg.~\ref{alg:czmgRq}, 
is scalable only when the regularized system is solved
through matrix-vector products rather than by dense assembly.

\section{Numerical Experiments}
\label{sec: num}

To support the theoretical results with evidence 
this section contains,
first, proof-of-concept results 
on Gaußian matrices
for the proposed Alg.~\ref{alg: main}.
Secondly, 
the methods from \cite{B8,B9} and \cite{li2023stochastic}
are qualitatively compared with Alg.~\ref{alg: main}, respectively Alg.~\ref{alg:cmgRq},
on Gaußian matrices of different sizes 
and specifically constructed $B$ 
controlling the ill-conditioning of the setting.
Thirdly, Alg.~\ref{alg:czmgRq}
is quantitatively compared to the methods from \cite{B8,B9}
and Alg.~\ref{alg:cmgRq}.
Lastly, a genuine matrix-free example 
from \textit{signal-to-noise ratio} 
indicates the state-of-the-art behavior 
of the proposed methods.

All algorithms are implemented in Python and the code is publicly available%
\footnote{%
\url{https://github.com/JJEWBresch/ZerothOrderGeneralizedOperatorNorm}.}.
The experiments are performed on an off-the-shelf MacBook Pro 2020 
with Intel Core i5 (4‑Core CPU, 1.4~GHz) and 8~GB~RAM.

\subsection{Proof-of-concept Results}
\label{sec: num proof of concept}

In this first discussion,
we aim to substantiate the success of the proposed methods,
cf.~Alg.~\ref{alg: main},
and the $q$-sample variance-reduced approach, cf.~Alg.~\ref{alg:cmgRq},
by proof-of-concept experiments.

Here,
we choose $d \in \{10,100,1000\}$ for the Gaußian matrices
$A \in \bb R^{d \times d}$ and $B \in \bb R^{2d \times d}$.
For a qualitative comparison,
we consider the resulting \emph{relative quotient error},
defined as 
\begin{equation*}
    \textrm{RQE}_k 
    \coloneqq 
    \frac{\norm{A/B}^2 - \tfrac{\norm{A v^k}^2}{\norm{B v^k}^2}}{\norm{A/B}^2},
\end{equation*}
and the relative error of the generalized eigen equation
\begin{equation*}
    \textrm{MSQR}_k 
    \coloneqq 
    \min_{n\leq k} \; \norm[\Big]{A^{\tT}A v^n - \tfrac{\norm{A v^n}^2}{\norm{B v^n}^2} B^{\tT}B v^n}^2,
\end{equation*}
i.e. the \emph{minimal squared residual} in the eigenvector equation.
Additionally,
we consider the sequence $\abs{\alpha_k}_{k \in \bb N} = \abs{b_kd_k - a_ke_k}_{k \in \bb N}$, 
respectively $\abs{b_k}_{k \in \bb N}$ from \cite{B9}.
Both latter
are estimating the length of the Riemannian gradients%
---a stochastic diagnostic for stationarity---%
since
\begin{equation}
    \label{eq: riemannian gradient estimate}
    \bb E_{x^k\sim\mc U(\bb S^{d-1})\mid V^k=v^k}
    \left[\left(\tfrac{2\sqrt d}{d_k^2}\alpha_k\right)^2\right]
    =\norm{\grad f(v^k)}^2.
\end{equation}
see Thm.~\ref{thm: riemannian style},
and \cite{B8,B9}, respectively.
Furthermore, 
the sequence of step sizes is reported in Fig.~\ref{fig: proof of concept} and~\ref{fig:exp_x}.

The resulting means over $50$ runs of Alg.~\ref{alg: main} and~\ref{alg:cmgRq}
for different estimates 
and numbers of samples are reported in Fig.~\ref{fig: proof of concept}
over iteration.

We observe a tremendous improvement by the $q$-sampling approach 
compared to a single, 
uniform at random step direction on the unit sphere. 
Besides the observable linear convergence rate of the RQE, 
the MSQE as well as the estimate for the length of the true Riemannian gradient, 
is doing much better than a sublinear behavior.
The latter is absent in theory for the proposed method,
as discussed in Rem.~\ref{rem: absent of rate} and Appendix~\ref{sec: alg}.
However,
a sublinear rate is proven for the baseline methods \cite{B8,B9,B12} for $\mc R(A^{\tT}A, B^{\tT}B)$, 
cf.~\eqref{eq: gen rayleigh}.
Notably, 
in \cite{li2023stochastic} it is only proven for critical points.
Furthermore,
since we do not report a minimal value until the $k$-th iteration of $|\alpha_k| = |b_k d_k - a_k e_k|$ 
(third row, Fig.~\ref{fig: proof of concept}),
we observe a linear convergence behavior to critical points%
--without a square for the estimates. 
Combined with the (observable linear) convergence to the global maximum, 
as it is proven in Thm.~\ref{thm: convergence as}  only for $q = 1$,
this has clear advantages to state-of-the-art methods
\cite{li2023stochastic,li2014sketching} from §~\ref{sec: zeroth order}, 
cf.~Alg.~\ref{alg: zeroth order} from Appendix~\ref{sec: alg}.

\begin{figure}[t]
\resizebox{\linewidth}{!}{%
    \begin{threeparttable}
    \begin{tabular}{c c c c}
        & $d = m = \tfrac{\ell}{2} = 10$ & $d = m = \tfrac{\ell}{2} = 100$ & $d = m = \tfrac{\ell}{2} = 1000$ \\
        \rotatebox{90}{%
        \hspace{0.1cm} $\abs{b_kd_k - a_ke_k}$ \hspace{0.6cm} $\textrm{MSQR}_k$ \hspace{2.1cm} $\textrm{RQE}_k$}
        & \includegraphics[width=0.5\linewidth, clip=true, trim=10pt 185pt 30pt 80pt]{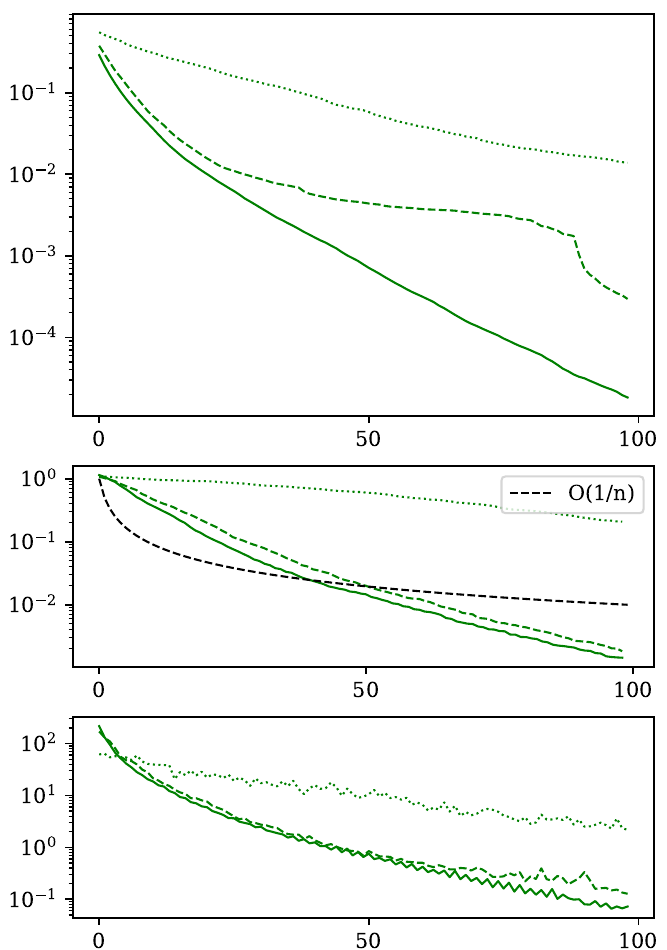}
        & \includegraphics[width=0.5\linewidth, clip=true, trim=10pt 185pt 30pt 80pt]{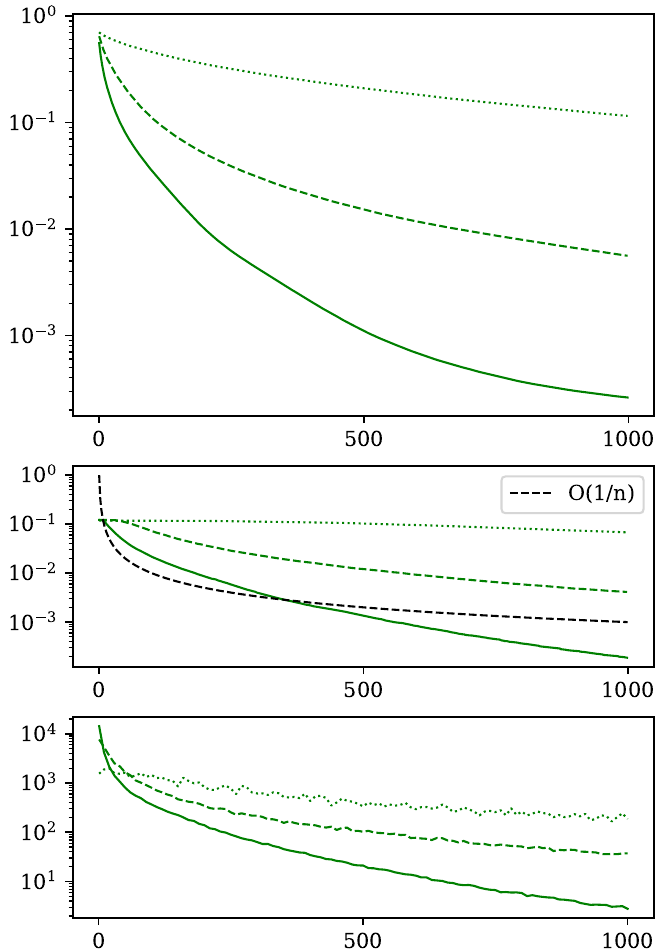}
        & \includegraphics[width=0.5\linewidth, clip=true, trim=10pt 185pt 30pt 80pt]{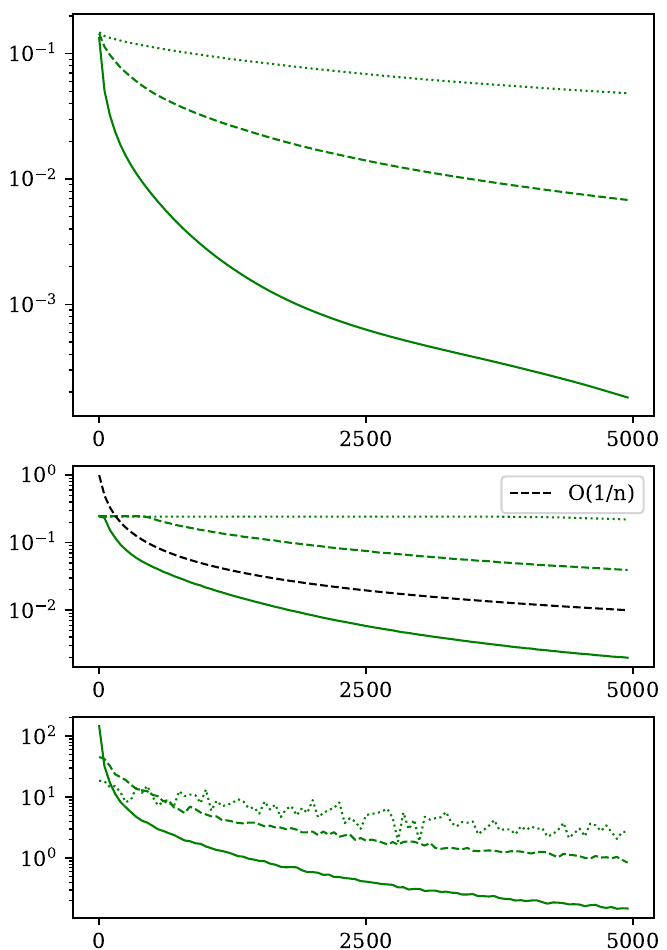}\\
        & Iterations $k$ & Iterations $k$ & Iterations $k$
    \end{tabular}
    \begin{tablenotes}
    \item[] \scalebox{1.2}{\hspace{3.4cm} $q=1$ \linegraydotted \hspace{2cm} $q=10$ \linegraydashed \hspace{2cm} $q=100$ \linegray}
    \end{tablenotes}
    \end{threeparttable}
    }
    \caption{Proof-of-concept experiment 
    on Alg.~\ref{alg:cmgRq} 
    for different sizes of matrices $(A, B)$.
    For all quantities 
    the mean over $50$ runs
    on random Gaußian matrices $(A, B)$ is reported.%
    }
    \label{fig: proof of concept}
\end{figure}

\subsection{Comparison with \cite{B9} and \cite{li2023stochastic}}
\label{sec: num comp}

As indicated in §~\ref{sec: num proof of concept}
a comparison with the recently proposed zeroth-order methods \cite{B8,B9}
is valuable.
In \cite{B9},
it is empirically shown,
and theoretically verified due to the optimal step-size selection 
in the two-frame $\{v^k, x^k\}$, see §~\ref{sec: construction},
that the tangent space construction, cf.~\eqref{eq: tangent surrogate},
outperforms state-of-the-art methods from \cite{li2023stochastic,boumal2023introduction},
cf.~Alg.~\ref{alg: zeroth order}, 
for the extremal generalized Rayleigh quotient \eqref{eq: gen rayleigh}
on the $B^{\tT}B$-sphere.
Here, we additionally compare the proposed method 
with the zeroth-order Riemannian method 
equipped with an Armijo line search
for the generalized operator norm 
on the unit sphere as baseline.
This is an improved version of \cite{li2023stochastic}
since convergence to critical points is theoretically only verified 
under constant step size, see Appendix~\ref{sec: app}.

\begin{figure}
    \resizebox{\linewidth}{!}{%
    \begin{threeparttable}
    \begin{tabular}{c c c c}
        & $d = m = \tfrac{\ell}{2} = 10$ & $d = m = \tfrac{\ell}{2} = 50$ & $d = m = \tfrac{\ell}{2} = 100$ \\
        \rotatebox{90}{\hspace{0.1cm} ${\color{orange}\abs{b_kd_k - a_ke_k}}, {\color{blue}\abs{b_k}}$ \hspace{0.4cm} $\textrm{MSQR}_k$ \hspace{1.9cm} $\textrm{RQE}_k$}
        & \includegraphics[width=0.5\linewidth, clip=true, trim=10pt 180pt 30pt 80pt]{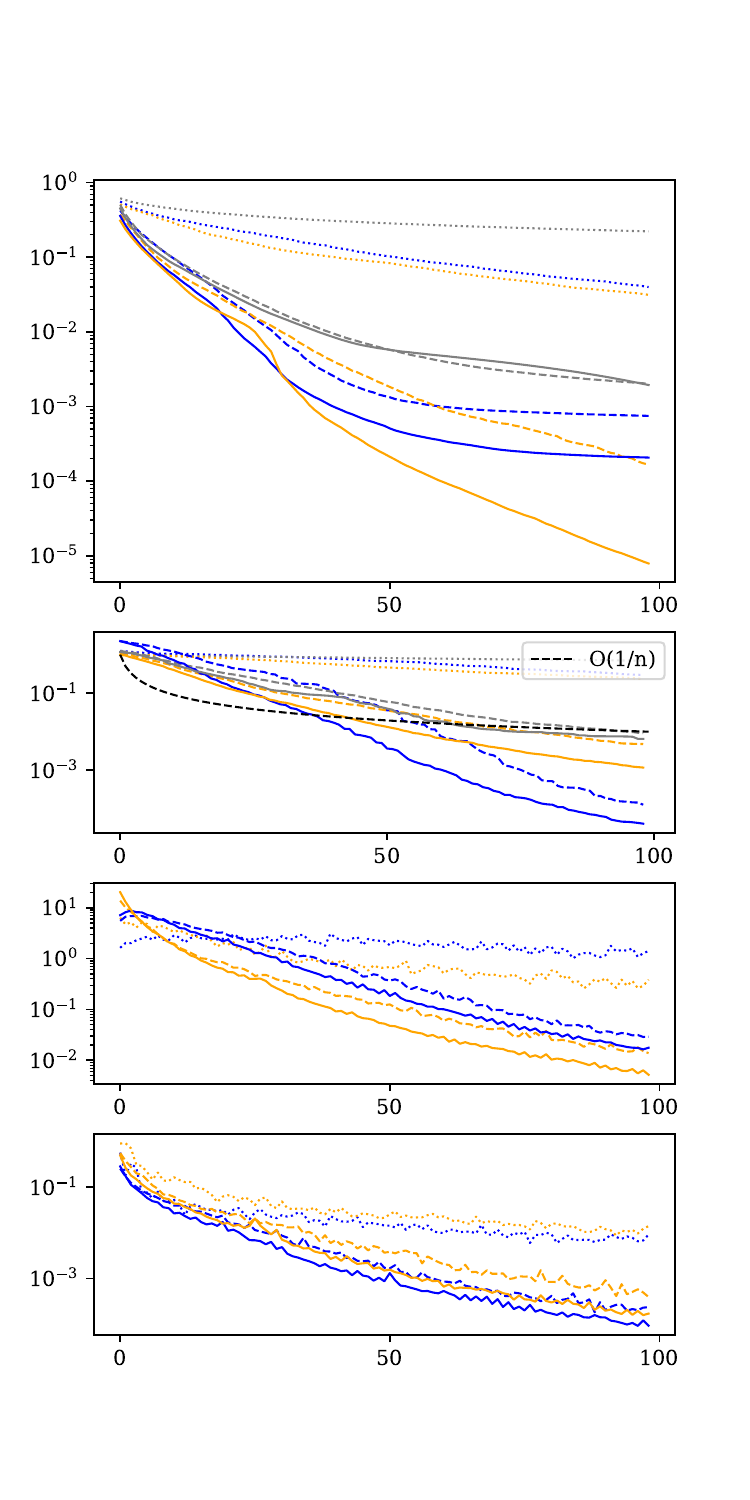}
        & \includegraphics[width=0.5\linewidth, clip=true, trim=10pt 180pt 30pt 80pt]{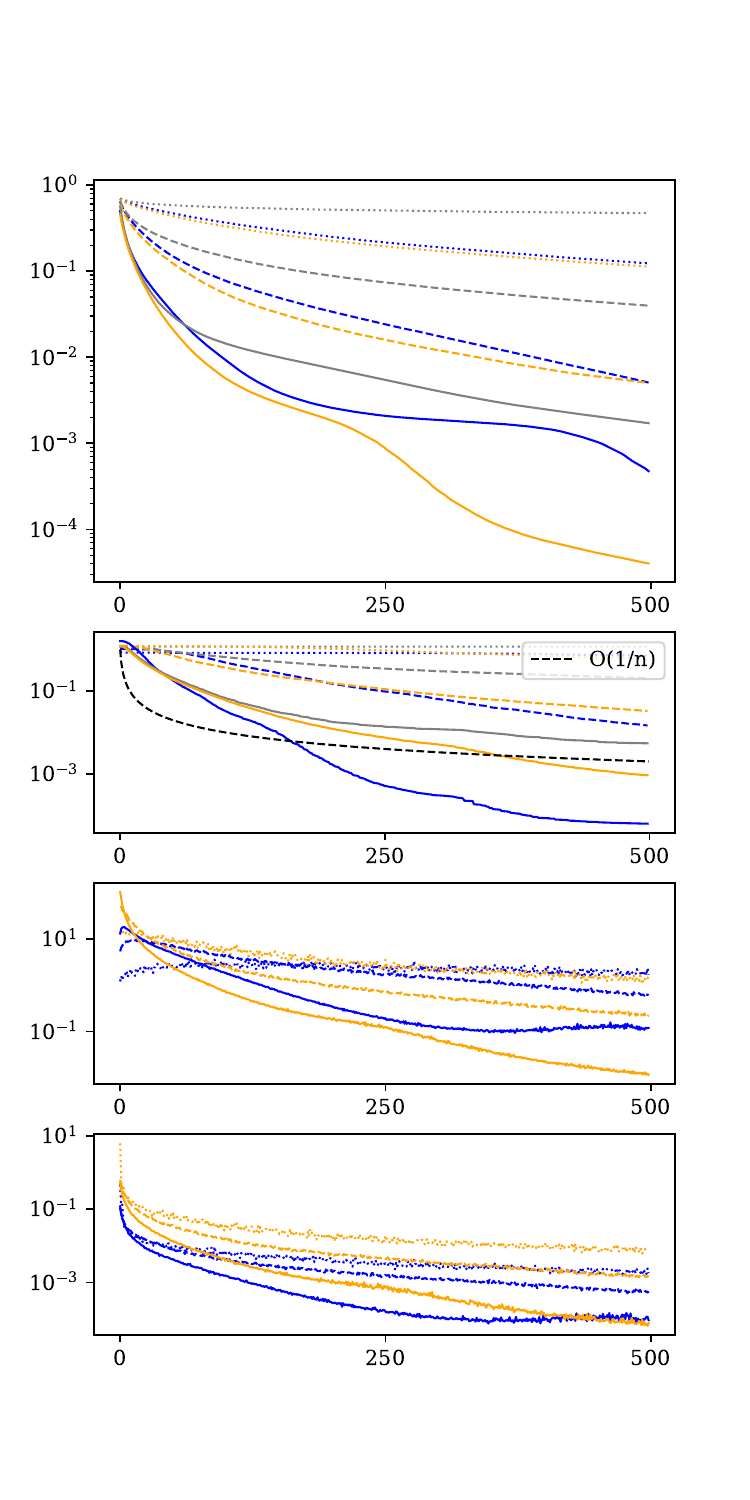}
        & \includegraphics[width=0.5\linewidth, clip=true, trim=10pt 180pt 30pt 80pt]{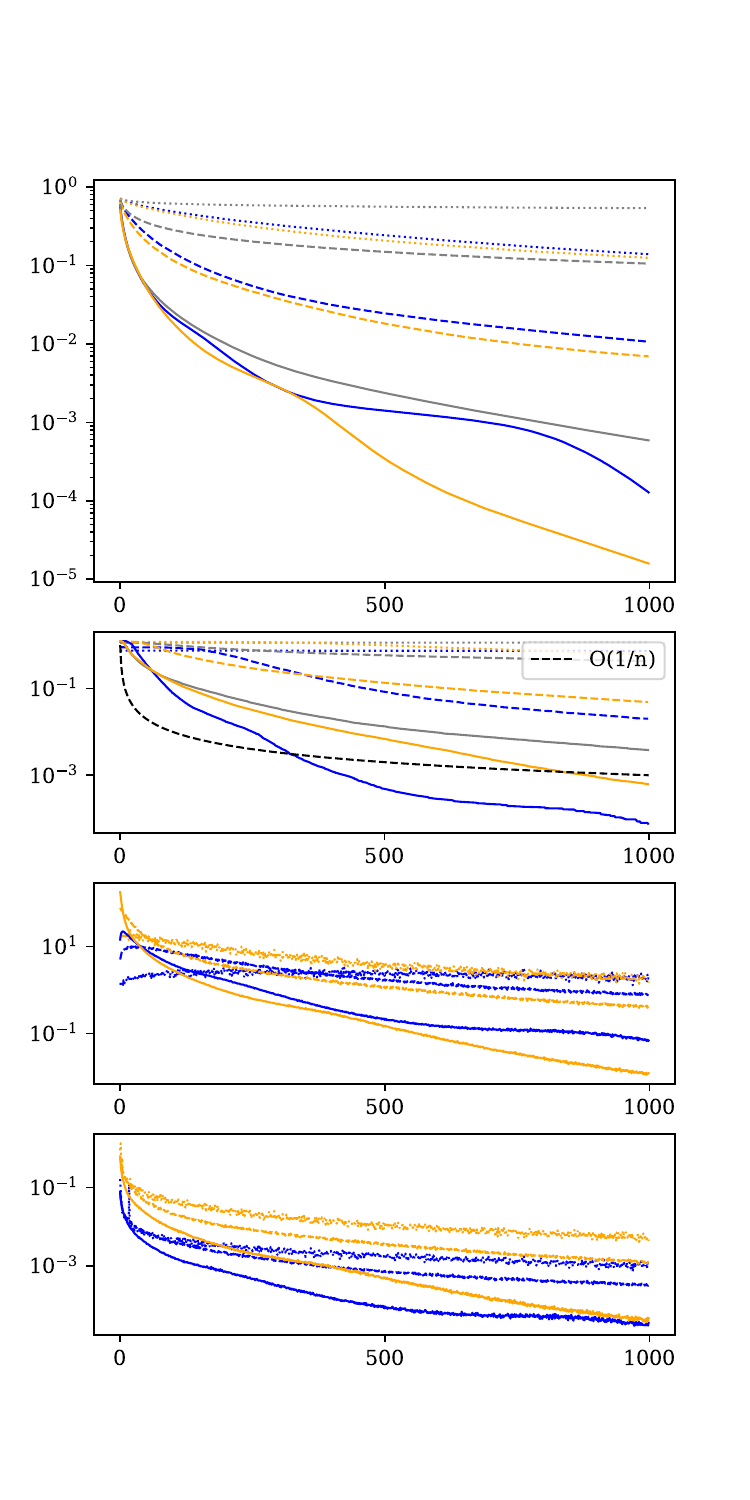}\\
        & Iterations $k$ & Iterations $k$ & Iterations $k$
    \end{tabular}
    \begin{tablenotes}
    \item[] \scalebox{1.2}{\hspace{4cm} $q=1$ \linegraydotted \hspace{2cm} $q=10$ \linegraydashed \hspace{2cm} $q=100$ \linegray}
    \end{tablenotes}
    \end{threeparttable}
    }
    \caption{Comparison 
    of Alg.~\ref{alg:cmgRq} (orange),
    \cite[Alg.~2]{B9} (blue) (on the extremal generalized Rayleigh quotient \eqref{eq: gen rayleigh}),
    and \cite{li2023stochastic} with Armijo line search, $(\mu_k)_{k \in \bb N} \in \mc O(1/k)$ (gray) 
    for the generalized operator norm
    for different sizes of matrices $(A, B)$.
    For all quantities 
    the mean over $50$ runs
    on random Gaußian matrices $(A, B)$ is reported.
    }
    \label{fig:exp_x}
\end{figure}

The experiments from §~\ref{sec: num proof of concept} are repeated 
within the same setting and reported over iteration 
and by time rescaling, see Fig.~\ref{fig:exp_x}, respectively Fig.~\ref{fig:exp_xt}.

It can be observed that the performances of both methods,
the proposed one, cf.~Alg.~\ref{alg: main}, respectively~\ref{alg:cmgRq},
and the zeroth-order $q$-sample approach \cite{B9} 
for the corresponding extremal generalized Rayleigh quotient 
are comparable in lower dimensions.
In higher dimensions ($m \gg 10$) 
the new sampling procedure,
cf.~\eqref{eq: x^k},
has lower mean error than the method
using \eqref{eq: tangent surrogate} from \cite{B9}.
This observation can be linked to the wider sampling space $\bb S^{d-1}$, 
which is not restricted to the tangent space $\bb T_{v^k}$ at any iteration.
Moreover, the projection in each iteration, especially for \cite[Alg.~2]{B9}, 
requires (at least) matrix-matrix multiplications,
where Alg.~\ref{alg:cmgRq} does not.
This reduces the computation time in higher dimensions
as it can be observed in Fig.~\ref{fig:exp_xt}.

Even though there is no convergence rate proven
for the considered estimates, cf.~Rem.~\ref{rem: absent of rate},
as it is possible in \cite{B8,B9,B12}, 
Fig.~\ref{fig:exp_x} and~\ref{fig:exp_xt} indicate a convergence rate $\mathcal O(\tfrac{1}{k})$ 
for MSQR and a linear rate for RQE.

\begin{figure}
    \resizebox{\linewidth}{!}{%
    \begin{threeparttable}
    \begin{tabular}{c c c c}
        & $d = m = \tfrac{\ell}{2} = 10$ & $d = m = \tfrac{\ell}{2} = 50$ & $d = m = \tfrac{\ell}{2} = 100$ \\
        \rotatebox{90}{\hspace{1.0cm} $\textrm{MSQR}_t$ \hspace{2.8cm} $\textrm{RQE}_t$}
        & \includegraphics[width=0.5\linewidth, clip=true, trim=10pt 40pt 30pt 50pt]{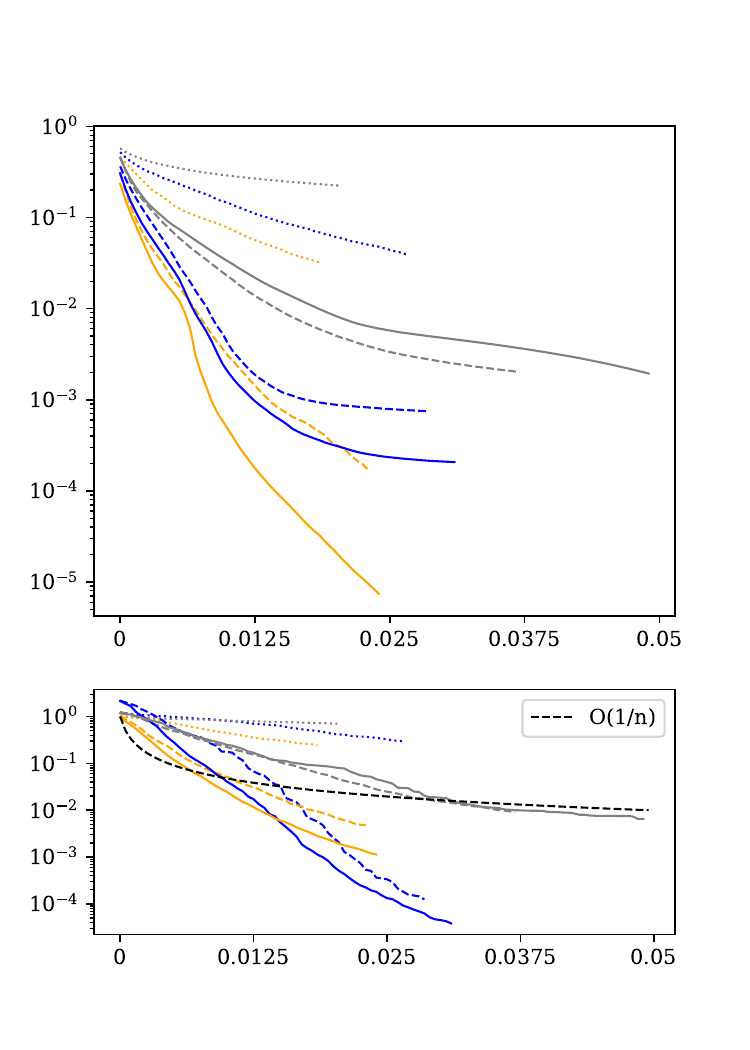}
        & \includegraphics[width=0.5\linewidth, clip=true, trim=10pt 40pt 30pt 50pt]{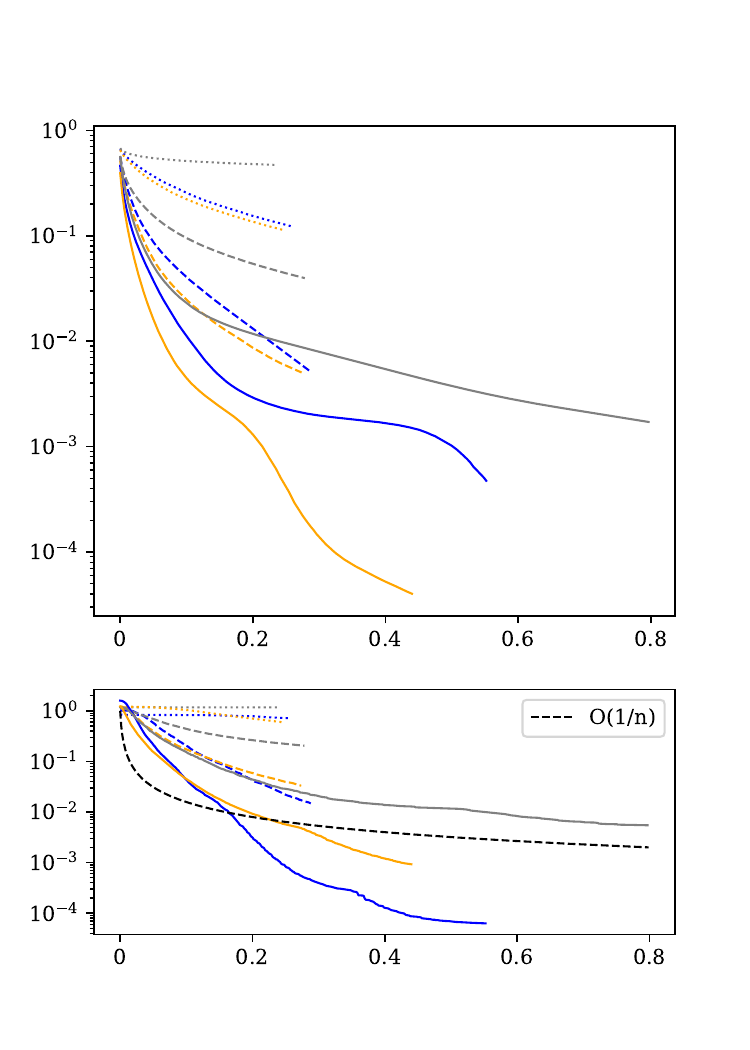}
        & \includegraphics[width=0.5\linewidth, clip=true, trim=10pt 40pt 30pt 50pt]{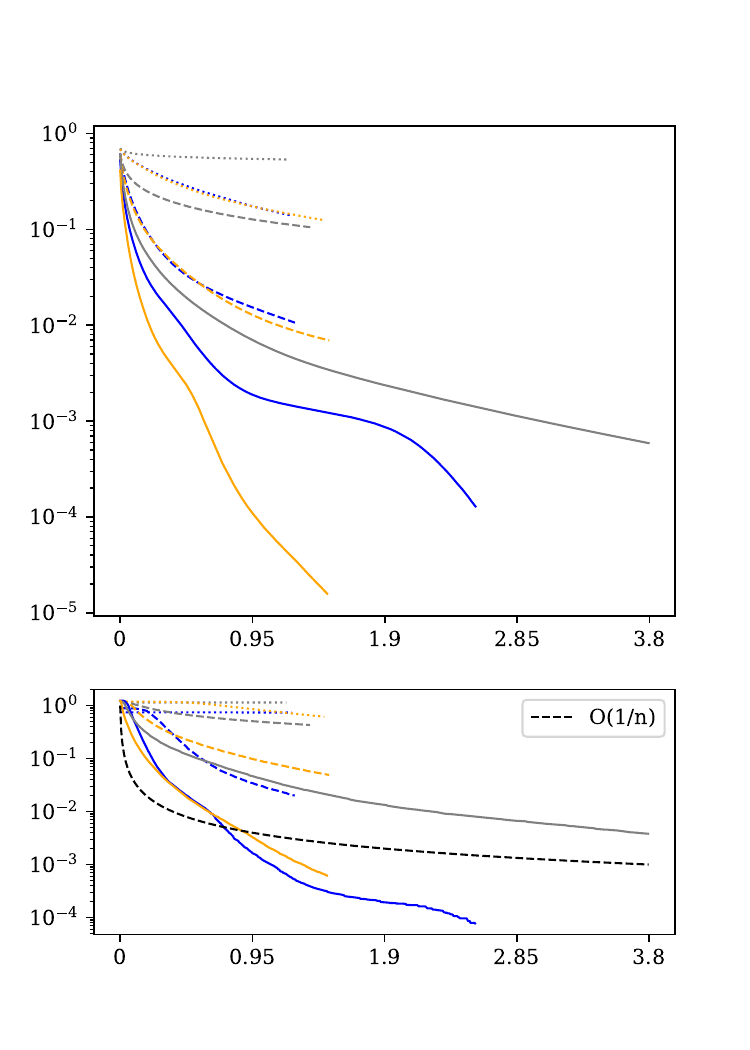} \\
        & Time, sec. $t$ & Time, sec. $t$ & Time, sec. $t$
    \end{tabular}
    \begin{tablenotes}
    \item[] \scalebox{1.2}{\hspace{4cm} $q=1$ \linegraydotted \hspace{2cm} $q=10$ \linegraydashed \hspace{2cm} $q=100$ \linegray}
    \end{tablenotes}
    \end{threeparttable}
    }
    \caption{Comparison
    from Fig.~\ref{fig:exp_x} by time rescaling.
    }
    \label{fig:exp_xt}
\end{figure}

\subsection{Study of Scalability}

In order to move beyond proof-of-concept results 
as provided in §~\ref{sec: num proof of concept}
and §~\ref{sec: num comp}
and conduct more in-depth studies, 
we are now comparing the considered methods from §~\ref{sec: num comp}
and vary the conditioning of \eqref{eq: prob}, i.e. of $B^{\tT}B$.
Therefore,
instead of Gaußian matrices,
we consider $r \in \{1,2,3\}$ 
and $B = \diag(10^0,..., 10^{r\nicefrac{i-1}{49}}, ...,10^r)^{\nicefrac{1}{2}}Q \in \bb R^{50\times50}$
where $Q \in \bb R^{50\times50}$ is uniformly sampled from the orthogonal group.

The time-rescaled means over $50$ runs of the methods
on Gaussian matrices $A$ and random sampled matrices $B$ from above
for the estimates as in §~\ref{sec: num comp} are reported in Fig.~\ref{fig:exp_t}.

We observe similar behaviors for Alg.~\ref{alg:cmgRq},
and the baseline method from \cite{B9} for the corresponding extremal generalized Rayleigh quotient computation in better conditioned settings.
For all settings of $r \in \{1,2,3\}$, 
the advantages of the proposed method 
over zeroth-order methods from the literature \cite{li2023stochastic}
are visible from the RQE and MSQE.
For the worst case, 
here $\kappa(B^{\tT} B) \approx 10^3$,
the advantage of the proposed method
to \cite{B9} becomes significant.

\begin{figure}
    \resizebox{\linewidth}{!}{%
    \begin{threeparttable}
    \begin{tabular}{c c c c}
        & $r=1$ & $r=2$ & $r=3$ \\
        \rotatebox{90}{\hspace{1.0cm} $\textrm{MSQR}_t$ \hspace{2.8cm} $\textrm{RQE}_t$}
        & \includegraphics[width=0.5\linewidth, clip=true, trim=10pt 40pt 30pt 50pt]{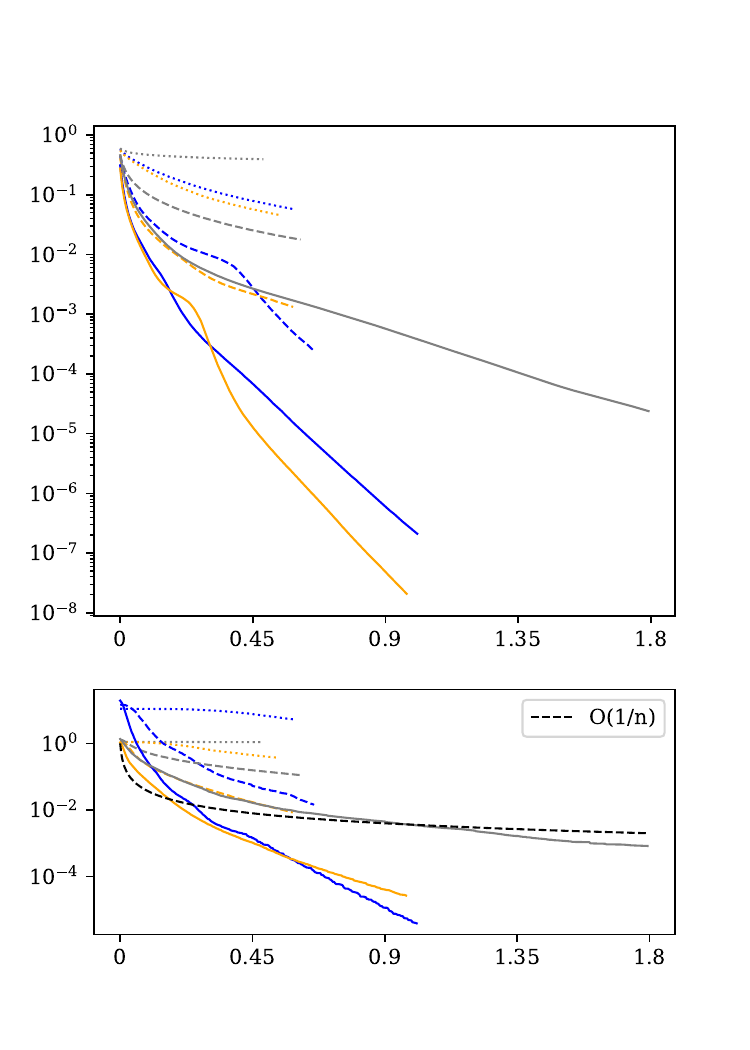}
        & \includegraphics[width=0.5\linewidth, clip=true, trim=10pt 40pt 30pt 50pt]{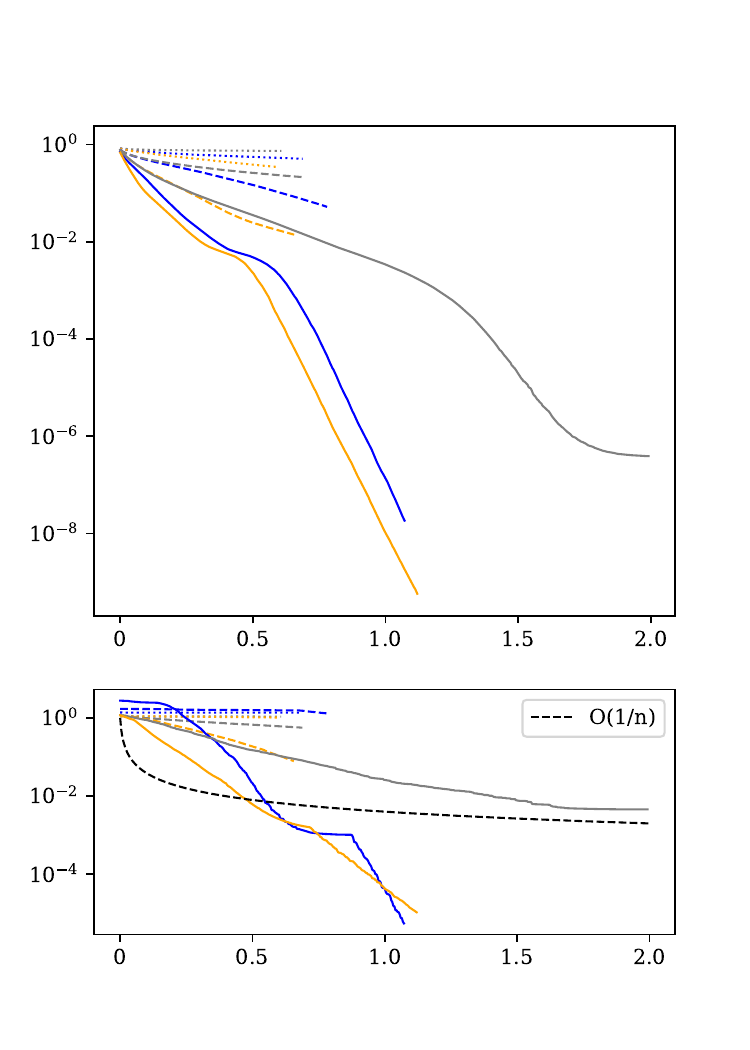}
        & \includegraphics[width=0.5\linewidth, clip=true, trim=10pt 40pt 30pt 50pt]{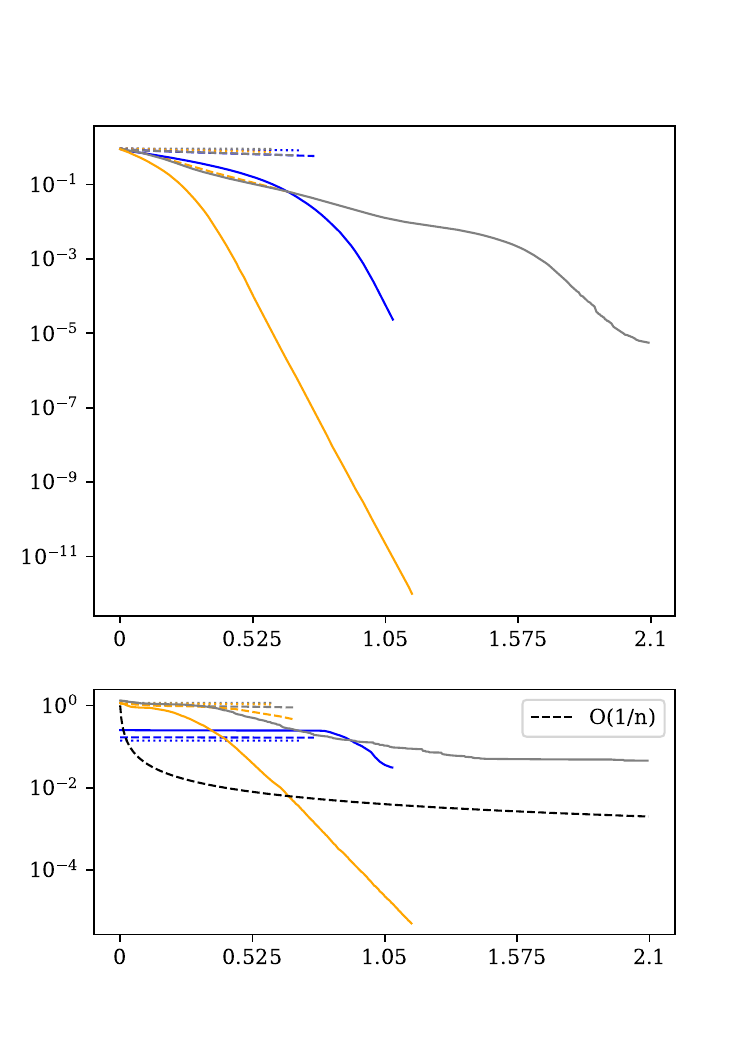} \\
        & Time, sec. $t$ & Time, sec. $t$ & Time, sec. $t$
    \end{tabular}
    \begin{tablenotes}
    \item[] \scalebox{1.2}{\hspace{3.4cm} $q=1$ \linegraydotted \hspace{2cm} $q=10$ \linegraydashed \hspace{2cm} $q=100$ \linegray}
    \end{tablenotes}
    \end{threeparttable}
    }
    \caption{Qualitative comparison
    of the methods from §~\ref{sec: num comp} as in Fig.~\ref{fig:exp_xt}
    on ill-conditioned problems by time rescaling.
    }
    \label{fig:exp_t}
\end{figure}

\subsection{Zeroth-Order Curvature-Based Method}
\label{sec: num-second}

Lastly,
we compare the novel sampling approach on the unit sphere, see Alg.~\ref{alg:cmgRq},
with the equipped $q$-sample variance-reduced zeroth-order estimator 
of the Riemannian Hessian, see Alg.~\ref{alg:czmgRq}, cf.~Cor.~\ref{cor: riemannian hessian estimator}.
We, again, report the estimates introduced in §~\ref{sec: num proof of concept}
and run Alg.~\ref{alg:cmgRq}, Alg.~\ref{alg:czmgRq}, and Alg.~\ref{alg: zeroth order}
in the setting as before $50$ times for different sizes of Gaussian matrices $(A, B)$.
The resulting mean of the estimates are shown in Fig.~\ref{fig: second comparison}.

First, 
we observe the global convergence of the quasi-Newton variant Alg.~\ref{alg:czmgRq},
which is more than we can prove from Rem.~\ref{rem: conv alg 3}.
More precisely, the convergence behavior indicates a linear rate, 
rather than the sub-linear convergence rate from \cite{B9}.
Second,
we observe that the quasi-Newton variant Alg.~\ref{alg:czmgRq} of Alg.~\ref{alg:cmgRq}
shows a stricter convergence behavior 
not only for the RQE 
but also for the MSQR compared to the original method.
Especially in higher dimensions, 
this effect becomes even more visible. 
Furthermore, the compared zeroth-order method \cite{balasubramanian2022zeroth,Marrinan2026zerothorderquasinewton}
gets stuck, so far we can observe.
This shows the advantage of including curvature information
of the function \eqref{eq: obj} from higher order
and empirically closes the vulnerable point of the original proposed methods,
i.e. Alg.~\ref{alg:cmgRq} and \cite[Alg.~2]{B9},
for reduced convergence speed 
when being close to critical points.
This behavior is well-known from classical first-order Riemannian ascent methods \cite{li2023stochastic},
as drawn in §~\ref{sec: first order}. i.e. Alg.~\ref{alg: first order} in Appendix~\ref{sec: alg}. 
Here, only the maximizers and not the saddle points 
are the problematic regions for the original proposed methods.
However, 
the method has a higher per-iteration cost, as indicated in §~\ref{sec: complexity},
which seems to be relatively small in practice.

\begin{figure}
    \resizebox{\linewidth}{!}{%
    \begin{threeparttable}
    \begin{tabular}{c c c c}
        & $d = m = \tfrac{\ell}{2} = 10$ & $d = m = \tfrac{\ell}{2} = 50$ & $d = m = \tfrac{\ell}{2} = 100$ \\
        \rotatebox{90}{\hspace{1cm} $\textrm{MSQR}_k$ \hspace{3.5cm} $\textrm{RQE}_k$}
        & \includegraphics[width=0.5\linewidth, clip=true, trim=10pt 10pt 10pt 0pt]{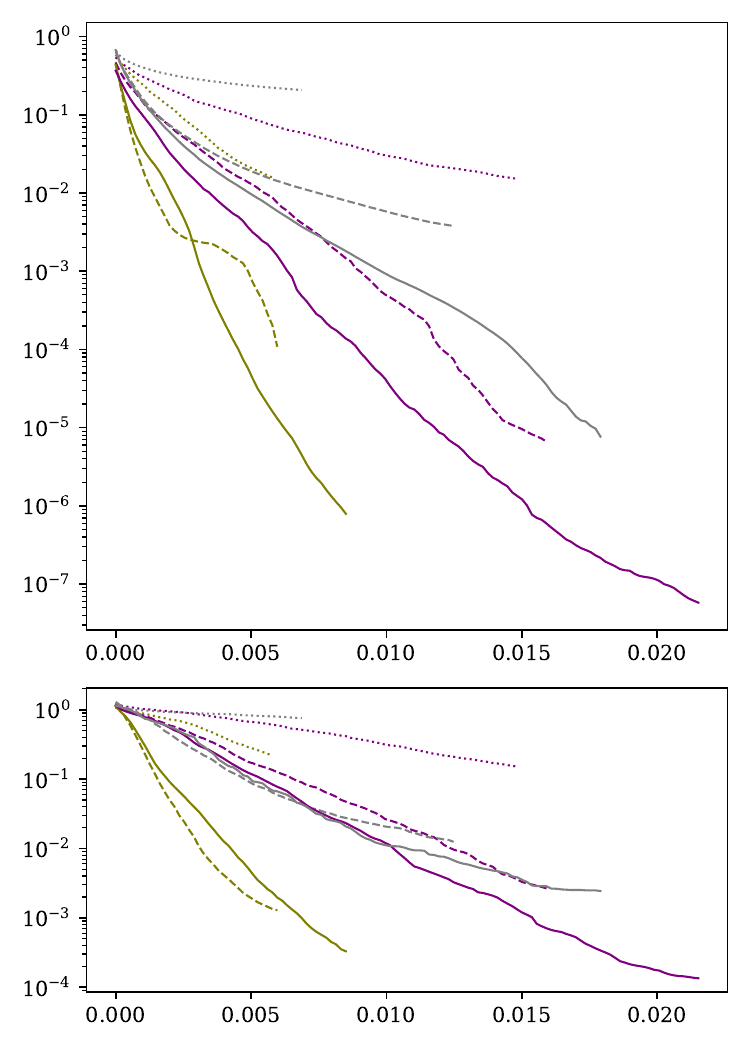}
        & \includegraphics[width=0.5\linewidth, clip=true, trim=10pt 10pt 10pt 0pt]{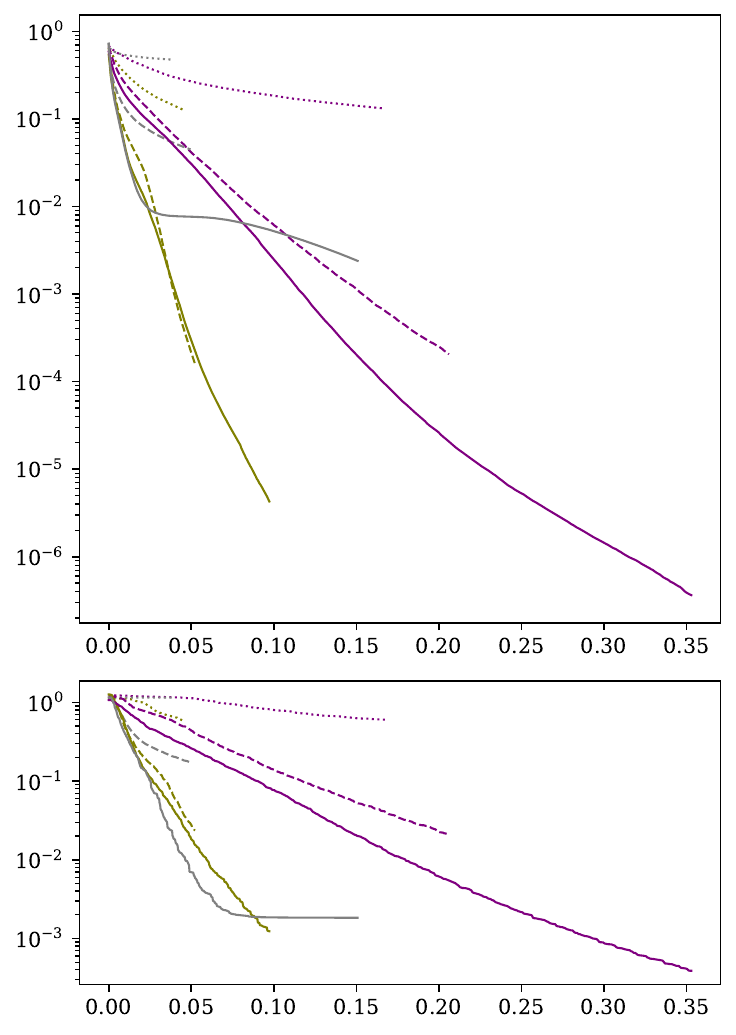}
        & \includegraphics[width=0.5\linewidth, clip=true, trim=10pt 10pt 10pt 0pt]{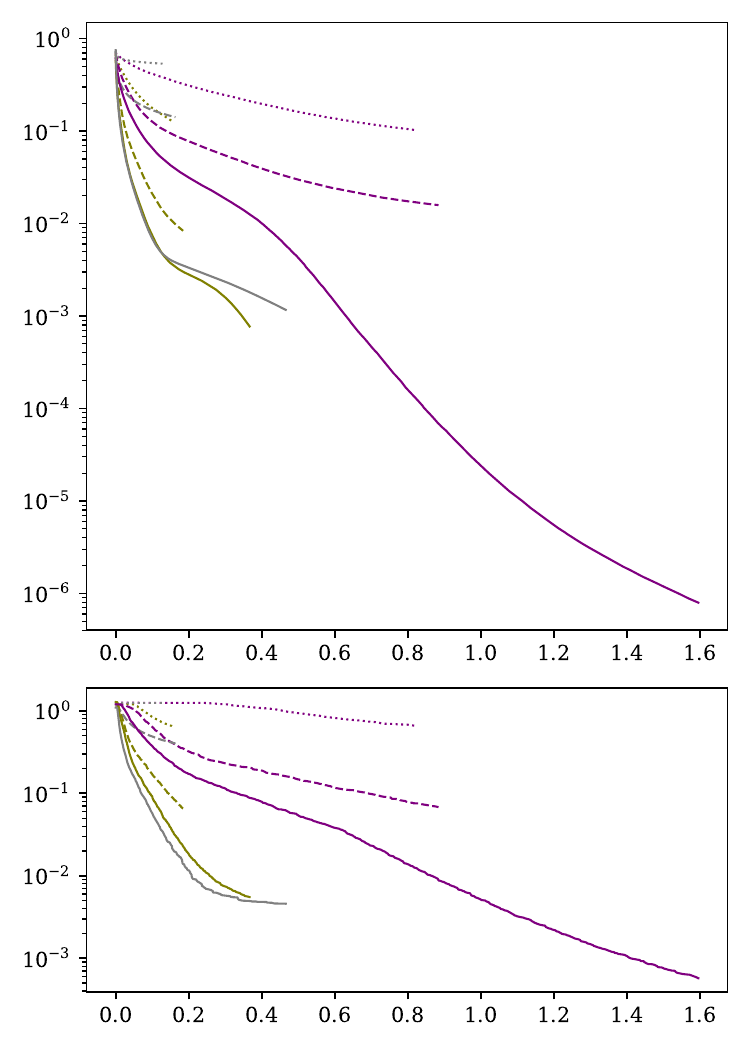} \\
        & Time, sec. $t$ & Time, sec. $t$ & Time, sec. $t$
    \end{tabular}
    \begin{tablenotes}
    \item[] \scalebox{1.2}{\hspace{3.4cm} $q=1$ \linegraydotted \hspace{2cm} $q=10$ \linegraydashed \hspace{2cm} $q=100$ \linegray}
    \end{tablenotes}
    \end{threeparttable}
    }
    \caption{Comparison 
    of Alg.~\ref{alg:cmgRq} (olive), Alg.~\ref{alg:czmgRq} (purple),
    and \cite{li2023stochastic} with Armijo line search (gray).
    for different sizes of matrices $(A, B)$.
    For all quantities 
    the mean over $50$ runs
    on random Gaußian matrices $(A, B)$ is reported
    by time rescaling.
    }
    \label{fig: second comparison}
\end{figure}

\subsection{Maximum-SNR Beamforming}
\label{sec:snr-beamforming}

As a structured signal-processing application, 
we consider maximum \emph{signal-to-noise ratio} (SNR) beamforming 
for a microphone array \cite{vantrees2002optimumarray,benesty2008microphone,souden2010optimal}. 
The objective is to determine a spatial filter 
that preserves a target source 
while suppressing directional interference and sensor noise. 
This problem is a natural instance of the generalized operator-norm problem 
because the target and noise powers are represented by two different data operators \cite{heymann2016neural}. 
Moreover, products with these operators can be evaluated directly from multichannel snapshots, 
without explicitly forming the corresponding covariance matrices.

\paragraph{Speech and Noise Snapshots.}
Let $M$ be the number of microphones.
At a fixed frequency $f_0$, 
let $s_t\in\mathbb{C}^{M},$ for $t=1,\ldots,T_s$,
denote target-speech snapshots, 
and let $n_t\in\mathbb{C}^{M}$, for $t=1,\ldots,T_n$,
denote noise-only snapshots. 
We collect the samples in
\[
    S=
    \begin{bmatrix}
    s_1 & \cdots & s_{T_s}
    \end{bmatrix}
    \in\mathbb{C}^{M\times T_s},
    \qquad
    N=
    \begin{bmatrix}
    n_1 & \cdots & n_{T_n}
    \end{bmatrix}
    \in\mathbb{C}^{M\times T_n}.
\]
For a beamforming vector $w\in\mathbb{C}^{M}$, the empirical target and
noise powers at the beamformer output are
\begin{equation}
    \label{eq: mat snr}
    \begin{aligned}
    P_s(w)
    &=
    \frac{1}{T_s}
    \sum_{t=1}^{T_s}
    |w^{*}s_t|^2
    =
    w^{*}R_s w,
    \qquad
    R_s=\frac{SS^{*}}{T_s}, \\
    P_n(w)
    &=
    \frac{1}{T_n}
    \sum_{t=1}^{T_n}
    |w^{*}n_t|^2
    =
    w^{*}R_n w,
    \qquad
    R_n=\frac{NN^{*}}{T_n}.
    \end{aligned}
\end{equation}
To avoid a singular or nearly singular denominator, we introduce a diagonal
loading parameter $\delta>0$ and define the regularized empirical SNR \cite{carlson1988diagonalLoading,li2003robustCapon} by
\[
    f_\delta(w)
    =
    \frac{P_s(w)}
    {P_n(w)+\delta\|w\|_{\bb C}^2},
    \quad \text{with} \quad 
    \|w\|_{\bb C}^2 \coloneqq w^*w \coloneqq \overline{w}^{\tT}w
\]
the complex inner product.
Define
\begin{equation}
    \label{eq: mat A B snr}
    A
    =
    \frac{S^{*}}{\sqrt{T_s}}
    \in\mathbb{C}^{T_s\times M},
    \qquad
    B
    =
    \begin{bmatrix}
    N^{*}/\sqrt{T_n}\\[1mm]
    \sqrt{\delta}\,I_M
    \end{bmatrix}
    \in\mathbb{C}^{(T_n+M)\times M}.
\end{equation}
Then it holds $\|Aw\|_{\bb C}^2=P_s(w)$ and $\|Bw\|_{\bb C}^2=P_n(w)+\delta\|w\|_{\bb C}^2$
and the beamforming problem is exactly
\[
    \max_{\|w\|_{\bb C}\neq0} f_\delta(w)
    =
    \max_{\|w\|_{\bb C}\neq0}
    \frac{\|Aw\|_{\bb C}^2}{\|Bw\|_{\bb C}^2}
    \coloneqq 
    \left\|A/B\right\|_{\bb C}^2.
\]
The maximizing vector is the spatial filter with the largest
regularized output SNR.

\paragraph{Real Formulation.}
The algorithms in this work are implemented for real vectors. 
For a complex matrix $C \in \bb C^{m\times d}$, 
we use, similar to \cite[App.~B]{B9}
\[
    M_{\Re}(C)
    \coloneqq
    \begin{bmatrix}
    \Re(C) & -\Im(C)\\
    \Im(C) &  \Re(C)
    \end{bmatrix} \in \bb R^{2m\times2d}.
\]
Writing for $w = \Re[w] + \rm i \, \Im[w] \in \bb C^d$
vectorized $\vv{w} \coloneqq [\Re[w], \Im[w]]^{\tT} \in \bb R^{2d}$
gives $\|M_{\Re}(A) \vv{w}\| = \|A w\|_{\bb C}$ and $\|M_{\Re}(B) \vv{w}\| = \|B w\|_{\bb C}$.
Thus, the complex beamforming problem is equivalent to the real generalized
operator-norm problem associated with $M_{\bb R}(A)$ and $M_{\bb R}(B)$.

\paragraph{Synthetic Array Model.}
For the controlled experiment, we use a uniform linear array with ($M=6$)
microphones, inter-microphone spacing
$d_{\mathrm{mic}}=0.04 \rm m$, sound speed
$c=343 \rm m/\rm s$, and frequency
$f_0=1000 \rm{Hz}$ \cite{barker2015chime3,barker2017chime3analysis}. 
The $m$-th entry of the steering vector is
\[
    [a(\theta)]_m
    =
    \exp\left(
    -2\pi\mathrm{i} f_0
    \frac{(m-1)d_{\mathrm{mic}}\sin(\theta)}{c}
    \right),
    \qquad m=1,\ldots,M.
\]
The target source is placed at $\theta_s=20^\circ$, while two interfering
sources are placed at $\theta_1=-50^\circ$ 
and $\theta_2=65^\circ$. The snapshots are generated according to
\[
    s_t=a(\theta_s)z_t+\varepsilon_t,
    \quad \text{and} \quad 
    n_t
    =
    a(\theta_1)u_{1,t}
    +
    a(\theta_2)u_{2,t}
    +
    \eta_t,
\]
where $z_t,u_{1,t},u_{2,t}$ are independent circular complex Gaußian
amplitudes and $\varepsilon_t,\eta_t$ model spatially white sensor noise.
Unless otherwise stated, we use $T_s=T_n=4000$ snapshots.

\paragraph{Reference Solution and Reported Quantities.}
For this moderate-dimensional experiment, 
an offline reference is obtained, due to \eqref{eq: mat snr} and \eqref{eq: mat A B snr},
from $R_s w = \lambda \left(R_n+\delta I_M\right)w.$
The largest generalized eigenvalue
equals the optimal regularized squared quotient. 
For an approximate beamformer $\widehat w \in \bb C^M$, we report
\[
    \operatorname{SNR}_{\delta,\rm{dB}}(\widehat w)
    =
    10\log_{10} f_\delta(\widehat w).
\]

\paragraph{Beampattern visualization.}
Each final beamformer $w \in \bb C^M$ is normalized $\widetilde w$ to have unit response in the target
direction,
i.e. $\widetilde w \coloneqq \frac{w}{\overline{w^{*}a(\theta_s)}}$ 
such that $\widetilde w^{*}a(\theta_s)=1$.
The normalized beampattern is
\[
    G_{\widetilde w}(\theta)
    =
    20\log_{10}
    \left|
    \widetilde w^{*}a(\theta)
    \right|,
    \qquad
    \theta\in[-90^\circ,90^\circ].
\]
For visualization, the response is truncated below $G_{\min}=-60\,\rm{dB}$. 
Since all curves equal $0\,\rm{dB}$ in the target direction, 
their values at the two interference directions
directly measure spatial suppression.

\begin{figure}[t]
\begin{minipage}{0.50\linewidth}
    \resizebox{\linewidth}{!}{%
    \begin{tabular}{c}
        \includegraphics[width=1.5\linewidth, clip=true, trim=7pt 55pt 190pt 110pt]{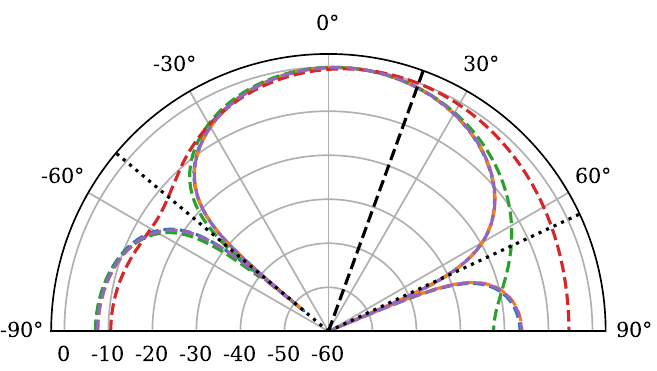} \\
        dB
    \end{tabular}}
    \includegraphics[width=0.48\linewidth, clip=true, trim=380pt 275pt 10pt 10pt]{plots/snr_beampattern_comparison_n6.pdf}%
    \includegraphics[width=0.48\linewidth, clip=true, trim=380pt 230pt 10pt 70pt]{plots/snr_beampattern_comparison_n6.pdf}
\caption{Normalized polar beampatterns at $f_0=1000$ (hertz).
All beamformers are normalized to unit response at the target
direction ($\theta_s=20^\circ$). The dotted radial lines mark the
interference directions ($\theta_1=-50^\circ$) and
($\theta_2=65^\circ$).}
\label{fig:snr-beampattern}
\end{minipage}%
\hfill
\begin{minipage}{0.44\linewidth}
    \resizebox{\linewidth}{!}{%
    \begin{tabular}{c c}
        \rotatebox{90}{\hspace{40pt}SNR$_{\delta, \rm{dB}, k}$} 
        & \includegraphics[width=1.5\linewidth, clip=true, trim=160pt 285pt 220pt 80pt]{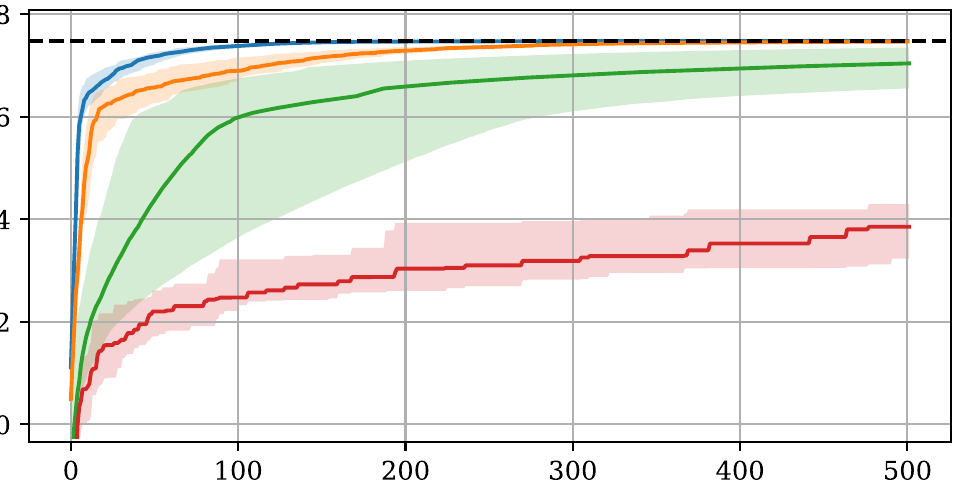}
        \\
        & Iterations $k$
    \end{tabular}}
    \resizebox{\linewidth}{!}{%
    \begin{tabular}{c c c c c c}
        \rotatebox{90}{\hspace{40pt}SNR$_{\delta, \rm{dB}}$} 
        & \multicolumn{5}{c}{\includegraphics[width=1.5\linewidth, clip=true, trim=0pt 0pt 0pt 0pt]{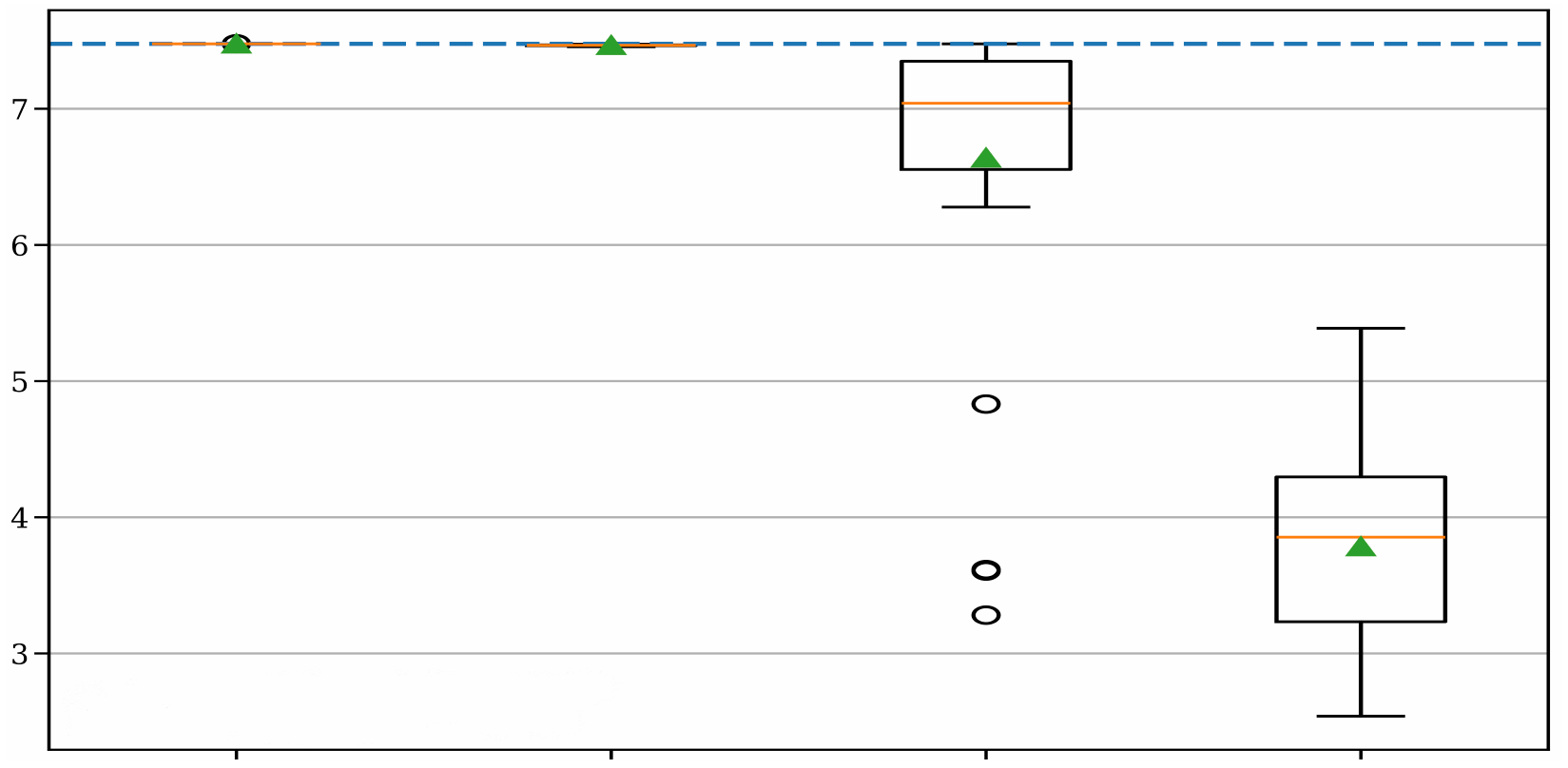}} 
        \\
        & \hspace{6pt} & \hspace{7pt} Alg.~2 \hspace{7pt}
        & \hspace{7pt} Alg.~3 \hspace{7pt}
        & \hspace{7pt} Alg.~5 \hspace{7pt}
        & Rand. search
    \end{tabular}}
    \caption{
    \textit{Top}: Regularized output SNR over iterations $k$ and $50$ independent runs.
    \textit{Bottom}: Distribution of the final regularized output SNR
    over $50$ independent runs. 
    The center line of each box
    is the median, the box contains the interquartile range.
    }
    \label{fig:snr-distribution}
\end{minipage}
\label{fig:snr}
\end{figure}

\paragraph{Distribution over Randomized Runs.}
All randomized algorithms are evaluated over $50$ independent runs. 
For run ($r = 1,...,R$) of method ($\rm{m}$), let
\[
    q_{\rm{m},r}
    \coloneqq
    10\log_{10}
    f_\delta\bigl(\widehat w_{\rm{m}}^{(r)}\bigr)
\]
be the final regularized output SNR in decibels. We summarize the distribution
using the median $\operatorname{med}_{\rm{m}} \coloneqq \operatorname{median}_{1\leq r\leq R}q_{\rm{m},r}$,
the first $Q_{1,\rm{m}} \coloneqq Q_{0.25}\left(\{q_{\rm{m},r}\}_{r=1}^{R}\right)$
and third quartiles $Q_{3,\rm{m}} \coloneqq Q_{0.75}\left(\{q_{\rm{m},r}\}_{r=1}^{R}\right)$,
and the interquartile range $\operatorname{IQR}_{\rm{m}}=Q_{3,\rm{m}}-Q_{1,\rm{m}}$.
The box plots in Figure~\ref{fig:snr-distribution} display the median,
interquartile range, and outlying runs.

We compare the proposed methods Alg.~\ref{alg:cmgRq} and~\ref{alg:czmgRq},
with the zeroth-order Riemannian gradient methods from the literature \cite{li2023stochastic},
cf.~Alg.~\ref{alg: zeroth order}, and a random search.
Since we have $\vv w \in \bb R^{12}$,
we use $q = 8$ and $500$ iterations for all computations 
and the setting for \cite{li2023stochastic} as before.

We observe that Alg.~\ref{alg:cmgRq} 
converges fastest regarding the SNR in Fig.~\ref{fig:snr-distribution}
with the smallest errors regarding the interquartile range over iterations.
Due to the lower-dimensional setting,
Alg.~\ref{alg:czmgRq} cannot develop the advantages from the curvature information
as observed in §~\ref{sec: num-second}.
Hence, it performs slightly worse than Alg.~\ref{alg:cmgRq}.
The zeroth-order Riemannian gradient method 
requires many more iterations to reach at least with the interquartile range the optimal value.
From Fig.~\ref{fig:snr-distribution}, there are some outliers
due to sensitivity to the theoretically prescribed parameter choice.
The na\"{\i}v approach by a random search is considered;
however, as expected, it is not successful.
Furthermore, from Fig.~\ref{fig:snr-beampattern}
Alg.~\ref{alg:cmgRq} closely reproduces the
reference main lobe while creating deep attenuation in the
interference directions.
The computation time is comparable to Fig.~\ref{fig:exp_xt}~(left).

\section{Conclusions}

In this paper,
we proposed a zeroth-order method for maximizing
quotients of quadratic functions.
The method replaces tangent-space sampling with
full-sphere sampling and yields an unbiased estimator
of the Riemannian gradient after appropriate rescaling.
Each sampled plane is optimized exactly through a
$2\times2$ generalized eigenvalue problem,
with a closed-form solution in the generic case.
We also derived an unbiased Riemannian Hessian estimator
and introduced a regularized quasi-Newton-type variant.
For the full-sphere sampling method,
we proved almost-sure convergence of the objective values
to the largest generalized eigenvalue
and convergence in distance of the iterates
to the leading generalized eigenspace.
The numerical experiments illustrate these results,
compare the proposed algorithms with existing methods,
and explore real-world applications in the complex setting.
The experiments additionally suggest linear convergence
and convergence to global maximizers for the gradient-
and curvature-based variants, beyond the guarantees
established here.

The method can also be adapted to compute the smallest
generalized Rayleigh quotient by selecting the minimizing
solution in Thm.~\ref{thm: svd}.
Future research could combine the sampling strategy
with iteratively reweighted least-squares techniques
to address more general optimization problems.
The computational benefit of curvature information
is currently supported by numerical evidence.
Further work includes a variance analysis of the Hessian
estimator and stronger convergence guarantees
for the gradient- and curvature-based variants.

\paragraph{Acknowledgement}
The author thanks  
O.~Melnyk, M.~Schoen, and G.~Steidl
for fruitful discussions
during the research for \cite{B9}.

\printbibliography%

\appendix

\section{Algorithms}\label{sec: alg} %

\begin{algorithm}[H]
    \caption{First-Order Riemannian Gradient Ascent Method}
    \label{alg: first order}
    \begin{algorithmic}
        \State{\textbf{Given} $A \in \bb R^{m \times d}$ and $B \in \bb R^{\ell \times d}$ such that $\rank(B) = d$.}
        \State{\textbf{Initialize} with $v^0 \in \bb S^{d-1}$.}
        \For{$k = 1,2,3,...$}
        \State{\textbf{Calculate} $x^k = \nabla f(v^k)$,
        \textbf{select} step size $\tau_k > 0$,
        \textbf{update} $v^{k+1} = R_{v^k}(\tau_k x^k)$}
        \EndFor
        \State{\textbf{Return} $\norm{A/B}^2 \approx \tfrac{\norm{A v^k}^2}{\norm{B v^k}^2}$.}
    \end{algorithmic}
\end{algorithm}

For the first-order Riemannian ascent method,
the convergence to critical points for solving \eqref{eq: prob sphere}
can be guaranteed in a sub-linear rate \cite{Lee2012smoothmanifolds}.
For that, choose $\tau_k=1/L$,
where $L>0$ satisfies
\[
    L\geq
    \tfrac{20\norm{A}^2\norm{B}^6}{\lambda_d(B^{\tT}B)^4},
\]
as in Lem.~\ref{lem: lipschitz grad}.
Since $\scp{v^k}{\nabla f(v^k)}=0$, 
the segment $v^k+t \nabla f(v^k)$, $0\leq t\leq1/L$, 
lies in $\{z:\norm{z}\geq1\}$.
The Euclidean Hessian, therefore, satisfies
\[
    \norm{\nabla^2 f(z)}
    \leq
    \tfrac{20\norm{A}^2\norm{B}^6}
    {\lambda_d(B^{\tT}B)^4\norm{z}^2}
    \leq L,
    \qquad \norm{z}\geq1.
\]
Hence, Taylor's inequality and scale invariance therefore yield
\[
    f(v^{k+1}) - f(v^k)
    = f(v^k+\nabla f(v^k)/L) - f(v^k)
    \geq
    \tfrac{\norm{\nabla f(v^k)}^2}{2L}.
\]
Consequently,

\begin{equation*}
    \label{eq: sublinear}
    \min_{0\leq k\leq n}\norm{\nabla f(v^k)}^2
    \leq
    \tfrac{2L}{n+1}
    \left(\norm{A/B}^2-f(v^0)\right).
\end{equation*}
which provides the sub-linear rate \cite{Lee2012smoothmanifolds}.
Thus the gradient norm converges to zero, 
and every accumulation point is critical.
These guarantees do not imply convergence to a global
maximizer from arbitrary initialization.
In particular, 
linear convergence to the global maximizer can be ensured \cite{alimisis2021distributed}
under more restrictive assumptions onto $A^{\tT}A$, 
e.g. positive leading eigengap, 
positive initial alignment with a leading eigenvector.

\begin{algorithm}
    \caption{Zeroth-Order Riemannian Gradient Ascent Method}
    \label{alg: zeroth order}
    \begin{algorithmic}
        \State{\textbf{Given} $A \in \bb R^{m \times d}$ and $B \in \bb R^{\ell \times d}$ such that $\rank(B) = d$.}
        \State{\textbf{Initialize} with $v^0 \in \bb S^{d-1}$ for instance $v^0 = x / \norm{x}$ for $x \sim \mc N(0, I_d)$.}
        \For{$k = 1,2,3,...$}
        \State{\textbf{Sample} $x_i \overset{\text{iid.}}{\sim} \mc N(0, I_d)$ for some $m \in \bb N$.}
        \State{\textbf{Construct} $x^k = \widehat{\grad}_m f(v^k)$ and $\mu_k > 0$ 
        as in §~\ref{sec: zeroth order}.}
        \State{\textbf{Select} appropriate step size $\tau_k > 0$ and
        \textbf{update} $v^{k+1} = R_{v^k}(\tau_k x^k)$.}
        \EndFor
        \State{\textbf{Return} $\norm{A/B}^2 \approx \tfrac{\norm{A v^k}^2}{\norm{B v^k}^2}$.}
    \end{algorithmic}
\end{algorithm}

For the zeroth-order method,
taking $\tau_k = \nicefrac{1}{2(d+4)L}$
and a square summable but not summable scaling sequence,
i.e. $(\mu_k)_{k \in \bb N} \subset \bb R_{> 0}$ 
such that $\sum_{k = 1}^\infty \mu_k^2 < \infty  = \sum_{k = 1}^\infty \mu_k$,
ensures the sublinear convergence rate to critical points \cite{li2023stochastic}.
Notably, 
where there is no guarantee for the convergence to global maximizer at all,
the recent proposed methods \cite{B8,B9,B12}
provably converge in a sublinear manner as for the first-order method in Alg.~\ref{alg: first order}
especially to the global maximizers.
Here, for the generalized sampling procedure, cf.~Alg.~\ref{alg: main} and~\ref{alg:czmgRq},
the fixed-accuracy geometric probability bound
does not establish linear decay of its objective error,
due to the missing orthogonal construction.

\section{Proofs}\label{sec: app}

\noindent\textbf{Lemma~\ref{lem: lipschitz func}}
    The objective from \eqref{eq: prob sphere}, i.e. \eqref{eq: obj}, is
    \smash{$2\norm{A}^2\left(
        \tfrac{1}{\lambda_d(B^{\tT}B)}
        +\tfrac{\norm{B}^2}{\lambda_d(B^{\tT}B)^2}
        \right)$}-Lipschitz continuous on $\bb S^{d-1}$.
\begin{proof}[of Lem.~\ref{lem: lipschitz func}]
    For $u \in\bb S^{d-1}$, write
    $a(u) \coloneqq \|A u\|^2 =\langle u,A^{\tT}Au\rangle$ 
    and
    $d(u) \coloneqq \|B u\|^2=\langle u,B^{\tT}Bu\rangle$. 
    Then
    \begin{align*}
       |f(v) - f(w)|
       &\leq \frac{|a(v) - a(w)|}{d(v)} + |a(w)| \cdot \bigl|\tfrac{1}{d(v)} - \tfrac{1}{d(w)}\bigr| \\
       &\leq \tfrac{2\|A\|^2}{\lambda_d(B^{\tT} B)} \|v - w\| + \|A\|^2 \tfrac{|d(v) - d(w)|}{\lambda_d(B^{\tT} B)^2} \\
       &\leq 2 \|A\|^2 \left(\tfrac{1}{\lambda_d(B^{\tT} B)} + \tfrac{\|B\|^2}{\lambda_d(B^{\tT} B)^2}\right) \|v - w\|,
       \quad \forall \, v,w \in \bb S^{d-1}.
       \tag*{\qedhere}
    \end{align*}
\end{proof}

\noindent\textbf{Lemma~\ref{lem: lipschitz grad}}
    The (extended) gradient 
    of the objective from \eqref{eq: prob sphere}, i.e. \eqref{eq: obj},
    is \smash{$\tfrac{20}{\lambda_{d}(B^{\tT}B)^4} \|A\|^2\|B\|^6$}-Lipschitz continuous 
    on 
    $\bb S^{d-1}$. 
\begin{proof}[of Lem.~\ref{lem: lipschitz grad}]
    By \eqref{eq: riemannian gradient}~f. and \eqref{eq: euclidean grad},
    we obtain 
    \begin{align*}
        &\hspace{-20pt}\|\grad f(v) - \grad f(w)\|
        = \|\nabla f(v) - \nabla f(w)\|\\
        & = 2\left\|\frac{\norm{B v}^2 A^{\tT}A v - \norm{A v}^2 B^{\tT}B v}{\norm{B v}^4}
        - \frac{\norm{B w}^2 A^{\tT}A w - \norm{A w}^2 B^{\tT}B w}{\norm{B w}^4}\right\| \\
        & = 2\biggl\|\frac{\norm{B w}^4\norm{B v}^2 A^{\tT}A v - \norm{A v}^2 \norm{B w}^4B^{\tT}B v}{\norm{B v}^4\norm{B w}^4} \\
        & \qquad\qquad - \frac{\norm{B v}^4\norm{B w}^2 A^{\tT}A w - \norm{A w}^2 \norm{B v}^4B^{\tT}B w}{\norm{B v}^4\norm{B w}^4}\biggr\| \\
        & \leq \tfrac{2}{\lambda_{d}(B^{\tT}B)^4} 
        \Bigl[\norm{B v}^2 \norm{B w}^4 \|A^{\tT}A(v - w)\| %
        + \norm{A v}^2 \norm{B w}^4 \|B^{\tT}B(v - w)\| \\
        &\qquad\qquad\qquad\quad + \bigl|\norm{B v}^2 \norm{B w}^4 - \norm{B v}^4 \norm{B w}^2\bigr| \norm{A^{\tT}A w} \\
        &\qquad\qquad\qquad\quad + \bigl|\norm{A v}^2 \norm{B w}^4 - \norm{A w}^2 \norm{B v}^4\bigr| \norm{B^{\tT}B w}\Bigl] \\
        & \leq \tfrac{20}{\lambda_{d}(B^{\tT}B)^4}
        \|A\|^2\|B\|^6
        \|v - w\|,
        \qquad\qquad v, w \in \bb S^{d-1}
    \end{align*}
    since 
    \begin{align*}
        \bigl|\norm{B v}^2 \norm{B w}^4 - \norm{B v}^4 \norm{B w}^2\bigr|
        & = \norm{B v}^2 \norm{B w}^2 \underbracket{\bigl|\norm{B w}^2 - \norm{B v}^2\bigr|}_{\leq 2 \norm{B}^2 \norm{v - w}} %
        \leq 2 \norm{B}^6 \norm{v - w},
    \end{align*}
    for any $v, w \in \bb S^{d-1}$ and 
    \begin{align*}
        &\bigl|\norm{A v}^2 \norm{B w}^4 - \norm{A w}^2 \norm{B v}^4\bigr|
        \leq \norm{A v}^2 \bigl|\norm{B v}^4 - \norm{B w}^4\bigr|
        + \norm{B v}^4\bigl|\norm{A v}^2 - \norm{A w}^2\bigr| \\
        &\qquad \leq \norm{A v}^2\underbracket{\bigl|\norm{B v}^2 + \norm{B w}^2\bigr|}_{\leq 2 \|B\|^2}
        \cdot \underbracket{\bigl|\norm{B v}^2 - \norm{B w}^2\bigr|}_{\leq 2 \|B\|^2 \|v - w\|} %
        + \norm{B v}^4\underbracket{\bigl|\norm{A v}^2 - \norm{A w}^2\bigr|}_{\leq 2\|A\|^2\|v - w\|} \\
        &\qquad\leq 6 \|A\|^2 \|B\|^4 \norm{v - w},
    \end{align*}
    for any $v, w \in \bb S^{d-1}$.
\end{proof}

\noindent\textbf{Proposition~\ref{prop: as conv eigenspace}}
    Let $\dim(G_{\lambda_1}) < d-1$
    and $(v^k)_{k \in \bb N}$ be the sequence of random variables 
    generated by Alg.~\ref{alg: main}.
    Then there exists a subsequence $(v^{k_j})_{j \in \bb N}$
    and a random variable $\lambda$ a.s. taking values in $\spec((B^{\tT}B)^{-1/2}A^{\tT}A(B^{\tT}B)^{-1/2})$
    such that 
    $\dist(G_\lambda, v^{k_j}) \to 0$ for $j \to \infty$ a.s.
    Moreover,
    the sequence $(\tfrac{a_k}{d_k})_{k \in \bb N}$
    converges to $\lambda$ a.s.
\begin{proof}[of Prop.~\ref{prop: as conv eigenspace}]
    Conclude by the $L^1$-convergence from Lem.~\ref{lem: EE_a_k}
    and the second statement from Lem.~\ref{lem: EE_a_k}
    using the law of total expectation that 
    \begin{equation*}
        \bb E[(\tfrac{\alpha_k}{d_k})^2]
        = \bb E_{v^k \sim V^k}[\bb E_{x^k \sim \mc U(\bb S^{d-1}) \mid V^k = v^k} [(\tfrac{\alpha_k}{d_k})^2]]
        = \tfrac{1}{4d}  \bb E_{v^k \sim V^k}[d_k^2 \norm{\grad f(v^k)}^2] \to 0,
    \end{equation*}
    as $k \to \infty$.
    Hence,
    there exists a subsequence $(v^{k_j})_{j \in \bb N} \subset \bb S^{d-1}$
    such that 
    \begin{equation*}
        d_{k_j}\norm{\grad f(v^{k_j})} \to 0,
        \quad \text{as} \quad j \to \infty,
        \quad \text{a.s.}
    \end{equation*}
    Due to the boundedness $0 < \lambda_d(B^{\tT}B) \leq d_k \leq \lambda_1(B^{\tT}B)$,
    we obtain 
    \begin{equation*}
        \norm{\grad f(v^{k_j})} \to 0,
        \quad \text{as} \quad j \to \infty,
        \quad \text{a.s.}
    \end{equation*}
    Hence, by construction, cf.~\ref{eq: eigen},
    there exists a sequence of eigenvalues 
    $$
        \lambda_j \in \spec((B^{\tT}B)^{-1/2}A^{\tT}A (B^{\tT}B)^{-1/2}),
        \qquad j \in \bb N,
    $$
    such that 
    \begin{equation*}
        \norm{(A^{\tT}A - \lambda_j B^{\tT}B) v^{k_j}} \to 0,
        \quad \text{as} \quad 
        j \to \infty.
    \end{equation*}
    Therefore,
    we have 
    \begin{equation*}
        a_{k_j}
        = \scp{v^{k_j}}{A^{\tT}A v^{k_j}}
        = \underbracket{\scp{v^{k_j}}{(A^{\tT}A - \lambda_j B^{\tT}B) v^{k_j}}}_{\to 0, \quad j \to \infty} + \underbracket{\scp{v^{k_j}}{\lambda_j B^{\tT}B v^{k_j}}}_{ = \lambda_j d_{k_j}}
    \end{equation*}
    such that by Cor.~\ref{cor: monoton conv} we obtain
    \begin{align*}
        \lim_{k \to \infty} f(v^k)
        &= \lim_{j \to \infty} f(v^{k_j})
        = \lim_{k \to \infty} \frac{a_{k}}{d_{k}} \\
        &= \lim_{j \to \infty} \frac{a_{k_j}}{d_{k_j}}
        = \lim_{j \to \infty} \lambda_j
        =: \lambda \in \spec((B^{\tT}B)^{-1/2}A^{\tT}A (B^{\tT}B)^{-1/2}).
    \end{align*}
    Since $\bb S^{d-1}$ is compact,
    there exists a convergent subsequence $(v^{k_{j_\ell}})_{\ell \in \bb N} \subset \bb S^{d-1}$
    with limit $v \in \bb S^{d-1}$ such that 
    \begin{equation*}
        \lim_{\ell \to \infty} \norm{(A^{\tT}A - \lambda_{j_\ell} B^{\tT}B) v^{k_{j_\ell}}}
        = \norm{(A^{\tT}A - \lambda B^{\tT}B) v} = 0
    \end{equation*}
    and hence $(v, \lambda)$ is a generalized eigen pair of $(A^{\tT}A, B^{\tT}B)$.
    Finally, we conclude by 
    \begin{equation*}
        \norm{(A^{\tT}A - \lambda B^{\tT}B) v^{k_j}}
        \leq \norm{(A^{\tT}A - \lambda_j B^{\tT}B) v^{k_j}} + |\lambda - \lambda_j| \cdot \norm{B^{\tT}B v^{k_j}}
        \to 0, \quad \text{as} \quad j \to \infty,
    \end{equation*}
    such that $\dist(G_\lambda, v^{k_j}) \to 0$ as $j \to \infty$ a.s.
\end{proof}

\noindent\textbf{Lemma~\ref{lem: uniform lower bound}}
    Let $v, \tilde v \in \bb S^{d-1}$ and $\varepsilon > 0$. 
    Define $\bb B_\varepsilon(\tilde v) \coloneq \{x \in \bb R^d \mid \; \norm{\tilde v - x} < \varepsilon\}$ and
    \begin{align*}
        \mc D_v \coloneqq \left\{ x \in \bb S^{d-1} 
        \;\middle|\; 
        \exists \tau \in \bb R, 
        R_v(\tau x) 
        = \tfrac{v + \tau x}{\norm{v + \tau x}} \in \bb B_\varepsilon(\tilde v) \cup \bb B_\varepsilon(-\tilde v)\right\}.
    \end{align*}
    There exists some $p_{\varepsilon,d} > 0$ 
    such that $\inf_{v \in \bb S^{d-1}} \sigma_{d-1}(\mc D_v) > p_{\varepsilon,d}$.
\begin{proof}[of Lem.~\ref{lem: uniform lower bound}]
    W.l.o.g assume that $\scp{v}{\tilde v} \geq 0$, 
    otherwise consider $-\tilde v$.
    Hence it holds $\norm{v - \tilde v} \leq \sqrt{2}$.
    Furthermore, if $\norm{v - \tilde v} < \varepsilon$, it suffices for any $x \in \bb S^{d-1}$
    to take $\tau = 0$ such that $\sigma_{d-1}(\mc D_v) = 1$.
    Now, 
    since for any $y \in \bb S^{d-1}$ and $w \in \bb R^d\setminus\{0\}$ holds that  
    \begin{align}
        \label{eq: uniform lower bound}
        \norm[\Big]{\tfrac{w}{\norm{w}} - y}
        &\leq \norm[\Big]{\tfrac{w}{\norm{w}} - w} + \norm{w - y}
        = \abs[\Big]{\tfrac{1}{\norm{w}} - 1} \norm{w} + \norm{w - y} \notag \\
        &= \abs[\big]{1 - \norm{w}} + \norm{w - y} 
        = \abs[\big]{\norm{y} - \norm{w}} + \norm{w - y} 
        \leq 2 \norm{w - y},
    \end{align}
    it suffices to take 
    \begin{equation*}
        \tilde{\mc D}_v \coloneqq \{x \in \bb S^{d-1} \mid \exists \tau \in \bb R, v + \tau x \in \bb B_{\nicefrac{\varepsilon}{2}}(\tilde v) \cup \bb B_{\nicefrac{\varepsilon}{2}}(-\tilde v)\}
        \subset \mc D_v.
    \end{equation*}
    As the latter can be trivially rewritten 
    \[
        \tilde{\mc D}_v = \{x \in \bb S^{d-1} \mid \exists \tau \in \bb R, \tau x \in \bb B_{\nicefrac{\varepsilon}{2}}(\tilde v - v) \cup \bb B_{\nicefrac{\varepsilon}{2}}(v+\tilde v)\}
    \]
    it includes a \textit{hyperspherical cap} $\mc K_\varepsilon \subset \tilde{\mc D}_v$ \cite{li2011concise} with opening angle 
    \[
        \tfrac{1}{2} \geq \sin(\alpha_\varepsilon) = \frac{\nicefrac{\varepsilon}{2}}{\norm{v - \tilde v}} \geq \tfrac{\varepsilon}{2\sqrt{2}} =: \sin(\beta_\varepsilon).
    \]
    An other hyperspherical cap $\tilde{\mc K}_\varepsilon \subset \mc K_\varepsilon$ is included
    defined by the opening angle $\beta_\varepsilon$.
    The surface area is given by the \textit{regularized incomplete Beta-function} \cite{li2011concise}
    such that by the monotonicity holds
    \begin{align*}
        \sigma_{d-1}(\mc D_v)
        &\geq \sigma_{d-1}(\tilde{\mc D}_v)
        \geq \sigma_{d-1}(\mc K_\varepsilon)
        \geq \sigma_{d-1}(\tilde{\mc K}_\varepsilon)
        =
        \frac{1}{2\mathrm B(\frac{d-1}{2},\frac12)}
        \int_0^{\sin^2(\beta_\varepsilon)}
        t^{\frac{d-3}{2}}(1-t)^{-\frac12}\,\mathrm dt\\
        &\geq
        \frac{1}{2\mathrm B(\frac{d-1}{2},\frac12)}
        \int_0^{\sin^2(\beta_\varepsilon)}
        t^{\frac{d-3}{2}}\,\mathrm dt
        =
        \frac{1}{(d-1)\mathrm B(\frac{d-1}{2},\frac12)}
        \left(\frac{\varepsilon}{2\sqrt2}\right)^{d-1}
        =: 
        p_{\varepsilon,d}
        >0,
    \end{align*}
    using that $(1-t)^{-1/2}\geq 1$ for $0\leq t\leq 1$
    in the last inequality.
\end{proof}

\noindent\textbf{Proposition~\ref{prop: rie hess}}
Let $C \in \bb R^{d \times d}$
    and $\symC \coloneqq \tfrac{1}{2}(C + C^{\tT})$ denotes the Hermitian part of $C$.
    It holds
    \begin{equation*}
        \bb E_{x \sim \mc U(\bb S^{d-1})}[\scp{x}{Ax}] = \frac{\trace(\symA)}{d} 
        \quad\text{and}\quad
        \bb E_{x \sim \mc U(\bb S^{d-1})}[\scp{x}{Ax} xx^{\tT}] = \frac{2\symA + \trace(A)I_d}{d(d+2)}.
    \end{equation*}
\begin{proof}[of Prop.~\ref{prop: rie hess}]
    For the first assertion,
    it holds 
    \begin{equation}
        \label{eq: symA inner}
        \scp{x}{Cx}
        = \frac{1}{2}(\scp{x}{Cx} + \langle C^{\tT}x, x\rangle)
        = \langle x, \underbracket{\tfrac{1}{2}(C + C^{\tT})}_{=: \symC} x\rangle
        = \scp{x}{\symC x}
        = \trace(\symC xx^{\tT}).
    \end{equation}
    Hence, 
    by the linearity of the trace,
    we have 
    \begin{equation*}
        \bb E_{x \sim \mc U(\bb S^{d-1})}[\scp{x}{Cx}]
        = \bb E_{x \sim \mc U(\bb S^{d-1})}[\trace(\symC xx^{\tT})]
        = \trace(\symC \bb E_{x \sim \mc U(\bb S^{d-1})}[xx^{\tT}])
        = \frac{1}{d} \trace(\symC),
    \end{equation*}
    using \eqref{eq: x^k}~f.
    For the second assertion 
    by utilizing that the mapping 
    \begin{equation*}
        \mc M : \mc S^d \to \mc S^d,
        \quad 
        C \mapsto \bb E_{x \sim \mc U(\bb S^{d-1})}[\scp{x}{Cx}xx^{\tT}]
    \end{equation*}
    is orthogonal equivariant, 
    which means that 
    \begin{align*}
        \mc M(Q^{\tT} C Q)
        &= \bb E_{x \sim \mc U(\bb S^{d-1})}[\scp{x}{Q^{\tT} C Q x} x x^{\tT}] 
        = \bb E_{x \sim \mc U(\bb S^{d-1})}[\scp{Q x}{C Q x} Q^{\tT}Q x (Q^{\tT}Q x)^{\tT}] \\
        &= \bb E_{x \sim \mc U(\bb S^{d-1})}[\scp{x}{C x} Q^{\tT} x x^{\tT} Q] 
        = Q^{\tT} \bb E_{x \sim \mc U(\bb S^{d-1})}[\scp{x}{C x} x x^{\tT}] Q \\
        &= Q^{\tT} \mc M(C) Q,
    \end{align*}
    holds by Schur's lemma \cite[§~1.2]{fulton2004representation} 
    and by the decomposition 
    of the set of Hermitian matrices 
    \begin{equation*}
        \mc S^d = \spa(I_d) \oplus \{C \in \mc S^d : \trace(C) = 0\}
    \end{equation*}
    that some parameters $\varkappa_1, \varkappa_2 \in \bb R$ exist
    such that
    \begin{equation*}
        \mc M(C) = \varkappa_1 C + \varkappa_2 \trace(C) I_d.
    \end{equation*}
    Furthermore,
    by Isserlis--Wick's theorem \cite{Isserlis1918,Wick1950}
    holds, by \eqref{eq: x^k}~f., with 
    \begin{equation*}
        \mc M(I_d) 
        = \bb E_{x \sim \mc U(\bb S^{d-1})}[xx^{\tT}] 
        = \tfrac{1}{d} I_d 
        = \varkappa_1 I_d + \varkappa_2 d I_d
    \end{equation*}
    and $\mc M(e_1e_1^{\tT}) = \bb E_{x \sim \mc U(\bb S^{d-1})}[x_1^2 x x^{\tT}]$
    from the $(1,1)$- and $(1,2)$-entry 
    that
    \begin{equation*}
        \bb E_{x \sim \mc U(\bb S^{d-1})}[x_1^4] 
        = \frac{3}{d(d+2)},
        \quad \text{and} \quad 
        \bb E_{x \sim \mc U(\bb S^{d-1})}[x_1^2 x_2^2] = \frac{1}{d(d+2)}.
    \end{equation*}
    Hence,
    by symmetry reasons holds 
        $\mc M(e_1 e_1^{\tT})
        = \tfrac{1}{d(d+2)} 
        \diag(3,1,..,1)
        = \varkappa_1 e_1 e_1^{\tT} + \varkappa_2 I_d$.
    Comparing the first two diagonal entries 
    in the latter equation 
    yields 
    \begin{equation*}
        \varkappa_1 + \varkappa_2 = \frac{3}{d(d+2)},
        \quad \varkappa_2 = \frac{1}{d(d+2)},
        \quad\text{such that}\quad
        \varkappa_1 = \frac{2}{d(d+2)},
        \quad \varkappa_2 = \frac{1}{d(d+2)}.
    \end{equation*}
    Applying all this onto $\symC$ due to \eqref{eq: symA inner}
    yields the assertion.
\end{proof}

\noindent\textbf{Corollary~\ref{cor: riemannian hessian estimator}}
Let $A\in\bb R^{m\times d}$, $B\in\bb R^{\ell\times d}$,
    and $\rank(B)=d$.
    For $z\in\bb R^d$ let
    \begin{align*}
        h_v(z)
        \coloneqq \norm{Az}^2-f(v)\norm{Bz}^2
        -\tfrac{4}{\norm{Bv}^2}
        \bigl(\scp{Av}{Az}-f(v)\scp{Bv}{Bz}\bigr)\scp{Bv}{Bz} 
        =\scp{z}{H(v)z}.
    \end{align*}
    It holds
    \begin{align*}
        \bb E_{x\sim\mc U(\bb S^{d-1})}
        \left[h_v(x)\left(xx^{\tT}-\frac{I_d}{d+2}\right)\right]
        &=\frac{2H(v)}{d(d+2)},\\
        \bb E_{y\sim\mc U(\bb T_v\cap\bb S^{d-1})}
        \left[h_v(y)\left(yy^{\tT}-\frac{P_v}{d+1}\right)\right]
        &=\frac{\norm{Bv}^2}{(d-1)(d+1)}\hess f(v).
    \end{align*}
\begin{proof}[of Cor.~\ref{cor: riemannian hessian estimator}]
    Since $\norm{Bv}^2\grad f(v)=2C(v)v$, it holds
    \begin{align*}
        \scp{z}{H(v)z}
        &=\scp{z}{C(v)z}-2\scp{\grad f(v)}{z}\scp{B^{\tT}Bv}{z}\\
        &=\norm{Az}^2-f(v)\norm{Bz}^2
         -\frac{4}{\norm{Bv}^2}\bigl(\scp{Av}{Az}-f(v)\scp{Bv}{Bz}\bigr)\scp{Bv}{Bz}.
    \end{align*}
    Applying Prop.~\ref{prop: rie hess} to the symmetric matrix $H(v)$,
    then subtracting $\bb E[h_v(x)]I_d/(d+2)$, 
    yields the first assertion.
    For the tangent identity choose $Q\in\bb R^{d\times(d-1)}$ 
    with $Q^{\tT}Q=I_{d-1}$ and $QQ^{\tT}=P_v$.
    Write $y=Qu$, 
    where $u \sim \mc U(\bb S^{d-2})$ is uniform on $\bb S^{d-2}$, 
    and apply the Prop.~\ref{prop: rie hess} 
    in dimension $d-1$ to $Q^{\tT}H(v)Q$. 
    This yields
    \begin{align*}
        \bb E[h_v(y)]&=\frac{\trace(P_vH(v)P_v)}{d-1},\\
        \bb E[h_v(y)yy^{\tT}]
        &=\frac{2P_vH(v)P_v+\trace(P_vH(v)P_v)P_v}{(d-1)(d+1)}.
    \end{align*}
    Subtracting $\bb E[h_v(y)]P_v/(d+1)$ and using
    $\hess f(v)=2P_vH(v)P_v/\norm{Bv}^2$ proves the second identity.
    Projection and the displayed scaling factors 
    give the two unbiased empirical estimators by linearity of expectation.
    \end{proof}

\end{document}